\documentclass[10 pt]{amsart}
\usepackage{amsmath} 
\usepackage{amsfonts}
\usepackage{amssymb}
\usepackage{amstext}
\usepackage{amsbsy}
\usepackage{amsopn}
\usepackage{amsthm}
\usepackage{amsxtra}
\usepackage{graphicx}
\usepackage{caption}
\usepackage{subcaption}
\usepackage{color}
\usepackage{hyperref}
\usepackage{enumerate}

\newtheorem{theorem}{Theorem}[section]
\newtheorem{lemma}[theorem]{Lemma}
\newtheorem{corollary}[theorem]{Corollary}
\newtheorem{proposition}[theorem]{Proposition}

\newtheorem{claim}[theorem]{Claim}
\newtheorem*{conjecture*}{Conjecture}
\newtheorem*{claim*}{Claim}
\newtheorem*{theorem*}{Theorem}
\newtheorem*{modnivatconj}{Modified Nivat Conjecture}
\newtheorem*{modnivatconj2}{Modified Nivat Conjecture (second version)}

\theoremstyle{remark}
\newtheorem{remark}[theorem]{Remark}
\theoremstyle{definition}
\newtheorem{definition}[theorem]{Definition}
\newtheorem{notation}[theorem]{Notation}

\newcommand{\A}{\mathcal{A}}
\renewcommand{\S}{\mathcal{S}}
\newcommand{\T}{\mathcal{T}}
\newcommand{\ZZ}{\mathbb{Z}^2}

\newcommand{\W}{\mathcal{R}}

\newcommand{\R}{\mathbb{R}}
\newcommand{\Z}{\mathbb{Z}}
\newcommand{\Q}{\mathbb{Q}}

\newcommand{\N}{\mathbb{N}}
\newcommand{\rst}[1]{\ensuremath{{\mathbin\upharpoonright}%
\raise-.5ex\hbox{$#1$}}}

\newcommand{\conv}{{\rm conv}}
\newcommand{\suc}{{\rm succ}}
\newcommand{\pred}{{\rm pred}}
\newcommand{\diam}{{\rm diam}}

\newcommand{\Ext}{{\rm Ext}}

\title{Nonexpansive $\ZZ$-subdynamics and Nivat's conjecture}
\author{Van Cyr}
\address{Northwestern University, Evanston, IL 60208 USA}
\email{cyr@math.northwestern.edu}
\author{Bryna Kra}
\address{Northwestern University, Evanston, IL 60208 USA}
\email{kra@math.northwestern.edu}

\subjclass[2010]{37B50 (primary), 68R15, 37B10}
\keywords{Nivat's conjecture, $\ZZ$-subshift, nonexpansive subdynamics, block complexity, periodicity}

\thanks{The second author was partially supported by NSF grant $1200971$.}

\begin{document}

\begin{abstract}
For a finite alphabet $\A$ and $\eta\colon \Z\to\A$, the Morse-Hedlund Theorem
states that $\eta$ is periodic if and only if there exists $n\in\N$ such that
the block complexity function $P_\eta(n)$ satisfies $P_\eta(n)\leq n$, and this
statement is naturally studied by analyzing the dynamics of a $\Z$-action
associated with $\eta$.  In dimension two, we analyze the subdynamics of a 
$\ZZ$-action associated with $\eta\colon\ZZ\to\A$ and show that if there exist $n,k\in\N$ such that the $n\times k$ rectangular complexity $P_{\eta}(n,k)$ satisfies $P_{\eta}(n,k)\leq nk$, then the periodicity of $\eta$
is equivalent to a statement about the expansive subspaces of this action.  As a corollary, we show that if there exist $n,k\in\N$ such that $P_{\eta}(n,k)\leq \frac{nk}{2}$, then $\eta$ is periodic.  This proves a weak form of a conjecture of Nivat in the combinatorics of words.
\end{abstract}

\maketitle

\section{Introduction}

\subsection{Periodicity and complexity} 
Given a finite alphabet $\A$, if $\eta\in \A^\Z$ is an infinite word, 
the {\em block complexity function} $P_\eta(n)$ is defined to be the number of distinct words of length $n$ appearing in $\eta$.  The word
$\eta= (\eta_n)_{n\in\Z}$ is said to be {\em periodic} if there exists an integer $m\in\N$ 
such that $\eta_n = \eta_{n+m}$ for all $n\in\Z$.
The classical Morse-Hedlund Theorem gives the relationship 
between these two notions:
\begin{theorem}[Morse-Hedlund~\cite{MH}]
The infinite word $\eta\in\A^{\Z}$ is periodic if and only if there exists an integer $n\geq 1$ such that $P_{\eta}(n)\leq n$.
\end{theorem}

For $\eta\in\A^{\Z^d}$, the {\em $(n_1\times\ldots \times n_d)$-block 
complexity function} $P_\eta(n_1, \ldots, n_d)$ is the number of distinct $n_1\times \ldots \times 
n_d$ blocks occurring in $\eta$.  Periodicity also has a natural higher 
dimensional generalization, and we say that the infinite word 
$\eta = (\eta_{\vec n})_{\vec n\in{\Z^d}}$ 
is {\em periodic} if there exists a {\em period vector}, meaning a vector 
$\vec m\in\Z^d$ such that $\eta_{\vec n} = \eta_{\vec n+\vec m}$ 
for all $\vec n\in\Z^d$.

Nivat conjectured that there is a simple analog of the Morse-Hedlund Theorem in two dimensions: 
\begin{conjecture*}[Nivat~\cite{N}]
For $\eta\in\A^{\ZZ}$, if there exist integers $n,k\geq 1$ such that $P_{\eta}(n,k)\leq nk$, then $\eta$ is periodic.
\end{conjecture*}

In a first step toward the conjecture, Sander and Tijdeman~\cite{ST} showed that if there is some $n$ such that $P_{\eta}(n,2)\leq 2n$ (or such that $P_\eta(2,n)\leq 2n$), then $\eta$ is periodic.  Soon after, Epifanio, Koskas and Mignosi~\cite{EKM} proved a weak 
version of the conjecture showing that if $P_{\eta}(n,k)\leq\frac{nk}{144}$ for some $n$ and $k$, then $\eta$ is periodic; Quas and Zamboni~\cite{QZ} improved the constant to $\frac{1}{16}$.

Conversely, Sander and Tijdeman~\cite[Example 5]{ST2} found counterexamples to the analog of Nivat's Conjecture in higher dimensions: if $d\geq3$ and $n\in\N$, there exists aperiodic $\eta\in\{0,1\}^{\Z^d}$, 
depending on $n$ and $d$, for which $P_{\eta}(n,\dots,n)=2n^{d-1}+1$.  This even rules out the possibility that an analog of Quas and Zamboni's Theorem holds for $d\geq3$.  The construction 
described by Sander and Tijdeman is a discretization of two skew lines in $\R^d$, 
and so does not provide a counterexample to Nivat's Conjecture in dimension two.

As with the Morse-Hedlund Theorem,  Nivat's conjectured relation between complexity 
and periodicity is sharp: the aperiodic coloring $\delta\in\{0,1\}^{\ZZ}$ with a  $1$ at $(0,0)$ and $0$ elsewhere satisfies $P_{\delta}(n,k)=nk+1$ for all integers $n,k\geq 1$.  
In contrast to the Morse-Hedlund Theorem, the relation is not an equivalence. Berth\'e and Vuillon~\cite{BV} and Cassaigne~\cite{C} gave examples of infinite $2$-dimensional 
words $\eta$ that are periodic, but whose block complexity satisfies 
$P_{\eta}(n,k)=2^{n+k-1}$ for all integers $n,k\geq 1$.

Further partial results connected to Nivat's Conjecture and its generalizations are given 
in~\cite{BV, Br, DR, ST2,  ST3, ST}, and we refer the reader to~\cite{B, C, F} for additional discussion.

Our main result is an improvement on Quas and Zamboni's Theorem:
\begin{theorem}\label{mainthm}
For $\eta\in\A^{\ZZ}$, if there exist integers $n,k\geq 1$ such that $P_{\eta}(n,k)\leq\frac{nk}{2}$, then $\eta$ is periodic.
\end{theorem}

Our proof is dynamical in nature: we associate a $\ZZ$-dynamical system with $\eta$ and study its subdynamics to prove the periodicity of $\eta$.

\subsection{Expansive subdynamics and the conjecture}
Suppose $\A$ is a finite alphabet; throughout we  assume that $\left|\A\right|>1$. 
In a classical way, we endow $\A$ with the discrete topology, $\A^{\Z^d}$ with the product topology, 
and define a $\Z^d$-action on $X=\A^{\Z^d}$ by $(T^{\vec u}\eta)(\vec x):=\eta(\vec x+\vec u)$ for $\vec u \in\Z^d$.  With respect to this topology, 
the maps $T^{\vec u}\colon X\to X$ are continuous.  
In a slight abuse, we omit the transformations $T^{\vec u}$ from our notation, and 
let $\mathcal{O}(\eta):=\{T^{\vec u}\eta\colon\vec u\in\Z^d\}$ denote the 
$\Z^d$-orbit of $\eta\in\A^{\Z^d}$ and write $X_{\eta}:=\overline{\mathcal{O}(\eta)}$.  

In this dynamical setup, we can rephrase periodicity.  The statement that 
$\eta$ is periodic is equivalent to saying that $\Z^d$ 
does not act faithfully on $X_{\eta}$.  A word $\eta$ is doubly periodic if it 
has two non-commensurate period vectors, and for $d=2$, the statement $\eta$ is doubly periodic
becomes $X_{\eta}$ is finite.

Expansiveness is a classical notion: a $\Z^d$-action by continuous maps $(T^{\vec n}\colon \vec n\in\Z^d)$ on a compact metric space 
$X$ is {\em expansive} if there exists $\delta > 0$ such that for any distinct points $x,y\in X$, 
there exists $\vec n\in\Z^d$ such that $d(T^{\vec n}x, T^{\vec n} y)\geq \delta$.  
In a symbolic setting, every $\Z^d$ action is expansive.  In particular, the space $X_\eta$ endowed with the translations $T^{\vec u}$ for $\vec u \in\Z^d$ is expansive.  
To study $\Z^d$-dynamical systems, Boyle and Lind~\cite{BL} introduced a finer notion, that of expansiveness for subspaces of $\R^d$.

The condition of expansiveness for a given $\Z^d$-action is open in each of the Grassmannian manifolds of $\R^d$ and important dynamical quantities, such as measure-theoretic and topological directional entropy, vary in a controlled manner within each connected component of this set~\cite{BL}.  Boyle and Lind define a subspace $V\subseteq\R^d$ to be {\em expansive} if there exist an {\em expansiveness radius} $r > 0$ and an {\em expansiveness constant} $\delta > 0$ such that whenever $x,y\in X$ satisfy
$$
d(T^{\vec u}x,T^{\vec u}y)<\delta
$$
for all $\vec u$ with $d(\vec u,V)<r$, then $x=y$.  If $V = \R^d$, we recover the 
usual definition of expansiveness.  Any subspace that is not expansive is called a {\em nonexpansive} subspace.   They showed that $\Z^d$-dynamical systems with nonexpansive subspaces are common:

\begin{theorem}[Boyle and Lind~\cite{BL}]
\label{BoyleLind}
Let $X$ be an infinite compact metric space with a continuous $\Z^d$-action.  For each $0\leq k<d$, there exists a $k$-dimensional subspace of $\R^d$ that is nonexpansive.
\end{theorem}

When restricting to $d=2$ and the context of 
$X=X_{\eta}$, a simple corollary 
is that $\eta$ is doubly periodic if and only if every subspace of $\R^2$ is expansive.
(As throughout the paper, we mean this with respect to the $\ZZ$-action on $X_{\eta}$ by translation.)  When there exist $n,k\in\N$ such that $P_{\eta}(n,k)\leq nk$, the connection between expansive subspaces of $\R^2$ and periodicity of $\eta$ goes deeper.  We show:

\begin{theorem}\label{singleperiodic}
Suppose $\eta\in\A^{\ZZ}$ and $X_{\eta}:=\overline{\mathcal{O}(\eta)}$.  If there exist $n,k\in\N$ such that $P_{\eta}(n,k)\leq nk$ and there is a unique nonexpansive $1$-dimensional subspace for the $\ZZ$-action (by translation) on $X_{\eta}$.  Then $\eta$ is periodic but not doubly periodic, the unique nonexpansive line $L$ is a rational line through the origin, and every period vector for $\eta$ is contained in $L$.
\end{theorem}

Thus Nivat's Conjecture reduces to: 
\begin{modnivatconj}
If $\eta\in\A^{\ZZ}$, $X_{\eta}:=\overline{\mathcal{O}(\eta)}$, and there exist $n,k\in\N$ such that $P_{\eta}(n,k)\leq nk$, there is at most one nonexpansive $1$-dimensional subspace for the $\ZZ$-action (by translation) on $X_{\eta}$.
\end{modnivatconj}

Under a stronger hypothesis, on the complexity, we show that this holds:
\begin{theorem}\label{twoormore}
If $\eta\in\A^{\ZZ}$ and there exist $n,k\in\N$ such that $P_{\eta}(n,k)\leq\frac{nk}{2}$, then there is at most one nonexpansive $1$-dimensional subspace for the $\ZZ$-action (by translation) on $X_{\eta}$.
\end{theorem}

Theorem~\ref{mainthm} follows immediately by combining Theorems~\ref{BoyleLind},~\ref{singleperiodic}, and~\ref{twoormore}.

Summarizing, if $\eta\in\A^{\ZZ}$ and there exist $n,k\in\N$ for which $P_{\eta}(n,k)\leq nk$, we 
prove that there is a trichotomy for the $\ZZ$-action by translation on $X_{\eta}$:
\begin{enumerate}
\item{\bf No nonexpansive $1$-dimensional subspaces.}  In this case, 
Theorem~\ref{BoyleLind} implies that $\eta$ is doubly periodic.
\item{\bf A unique nonexpansive $1$-dimensional subspace.}  In this case, 
Theorem~\ref{singleperiodic} implies that $\eta$ is periodic, but not doubly periodic.
\item{\bf At least two nonexpansive $1$-dimensional subspaces.}  If one could show that this case can not hold, Nivat's Conjecture follows.  In Theorem~\ref{twoormore},  
we show that this case is impossible if we strengthen the hypothesis on $\eta$ to the existence 
of $n,k\in\N$ such that $P_{\eta}(n,k)\leq\frac{nk}{2}$.
\end{enumerate}

To fully take advantage of the complexity assumption, we require a finer
notion than expansiveness studied by Boyle and Lind. Namely, we define a
one-sided version of expansiveness (see Section~\ref{expansivesection} for the precise
definition) which corresponds to a subspace determining a pattern in
only one direction. Similar to the use of expansiveness by Boyle and Lind, the more relevant notion is
that of one-sided nonexpansiveness. While expansive implies
one-sided expansive and one-sided nonexpansive implies nonexpansive,
the converse statements do not necessarily hold. A similar notion was
studied in~\cite{BlM} and~\cite{BoM}.

\subsection{Another reformulation of the conjecture} 
The proofs of Theorems~\ref{singleperiodic} and~\ref{twoormore} ultimately rely on the fact that if there exist $n,k\in\N$ with $P_{\eta}(n,k)\leq nk$, then the value of $\eta_{\vec n}$ can be deduced from information about the value of $\eta_{\vec m}$ for $\vec m\in\ZZ$ 
that are nearby in an appropriate sense.  
Following Sander and Tijdeman~\cite{ST2}, we make the following definition:
\begin{definition}
For $\S\subseteq\ZZ$, let $\mathcal{W}(\S,\eta):=\{(T^{\vec u}\eta)\rst{\S}\colon\vec u\in\ZZ\}$ be the set of distinct {\em $\eta$-colorings of $\S$} (or {\em $\S$-words}) and define the {\em $\eta$-complexity function} to be
$$
P_{\eta}(\S):=\left|\mathcal{W}(\S,\eta)\right|.
$$
\end{definition}

Note that this generalizes the definition of the complexity function $P_\eta(n,k)$ for rectangles.  
It is immediate that if $\T\subset\S$, then $P_{\eta}(\T)\leq P_{\eta}(\S)$.  We also note that if $\S\subset\ZZ$ is a fixed, finite set and $\eta\in\A^{\ZZ}$, then $P_{\eta}(\S)=P_{T^{\vec u}\eta}(\S)$ for any $\vec u\in\ZZ$.  Moreover if $\alpha\in X_{\eta} = \overline{\mathcal{O}(\eta)}$, then $P_{\alpha}(\S)\leq P_{\eta}(\S)$.  This is particularly useful in our constructions, since any complexity bound on $\eta$ implies a (possibly stronger) complexity bound on any element of $X_{\eta}$.

If $\T\subset\S$ and $\alpha\in\mathcal{W}(\T,\eta)$, we say that $\beta\in\mathcal{W}(\S,\eta)$ is an {\em extension} of $\alpha$ if $\beta\rst{\T}=\alpha$.  Moreover, we say that $\alpha$ {\em extends uniquely to an $\eta$-coloring of $\S$} if there exists a unique $\beta\in\mathcal{W}(\S,\eta)$ that is an extension of $\alpha$.

We define a {\em discrepancy function} that measures the difference between the complexity of a set and its size:
\begin{definition}
For $\eta\colon\ZZ\to\A$, the {\em $\eta$-discrepancy function} $D_\eta(\S)$, or just the {\em discrepancy function} when $\eta$ is clear from context, is defined on the set of all nonempty, finite subsets of $\ZZ$ by
$$
D_{\eta}(\S):=P_{\eta}(\S)-\left|\S\right|.
$$
\end{definition}
The discrepancy function has the useful property (Lemma~\ref{generated}) that if $\S\subset\ZZ$ contains at least two elements and $x\in\S$, then either $D_{\eta}(\S\setminus\{x\})\leq D_{\eta}(\S)$ or every $\eta$-coloring of $\S\setminus\{x\}$ extends uniquely to an $\eta$-coloring of $\S$ .
The discrepancy of any one element subset of $\ZZ$ is $\left|\A\right|-1>0$ and the hypothesis of Nivat's conjecture is that there exists a rectangular subset of $\ZZ$ whose discrepancy is non-positive.  This implies (Corollary~\ref{generatingset}) the existence of a set $\S\subset\ZZ$ with the property that for many $x\in\S$,  every  $\eta$-coloring of $\S\setminus\{x\}$ extends uniquely to an $\eta$-coloring of $\S$.

In this terminology, the Modified Nivat Conjecture becomes:
\begin{modnivatconj2}
If $\eta\in\A^{\ZZ}$, $X_{\eta}:=\overline{\mathcal{O}(\eta)}$, and there exists an 
$n$ by $k$ rectangular subset $R$ of $\ZZ$ satisfying $D_\eta(R) \leq 0$, 
then there is at most one nonexpansive $1$-dimensional subspace for the $\ZZ$-action (by translation) on $X_{\eta}$.
\end{modnivatconj2}

Theorem~\ref{singleperiodic} uses the set $\S$ to show that for all $\vec u\in\ZZ$, the value of $T^{\vec u}\eta$ along a strip depends only on its restriction to a particular finite set.  Theorem~\ref{twoormore} is more subtle.  The stronger hypothesis on the complexity $P_{\eta}(n,k)$ allows us to show that one-sided nonexpansiveness 
gives rise to periodicity along strips (a more precise statement is contained in 
Proposition~\ref{prop:extendedambiguousperiod}).  When there are multiple one-sided nonexpansive subspaces, this forces $\eta$ to be doubly periodic on large, finite subsets of $\ZZ$.  We then 
complete the argument with an elaborate proof by contradiction by analyzing $\eta$ on the boundary of these subsets.

\subsection{A guide to the paper}

Sections~\ref{sec:background} and~\ref{sec:periodicitytheorems} 
develop tools for analyzing the nonexpansive subspaces of $X_{\eta}$.
In Section~\ref{sec:background}, we define {\em $\eta$-generating sets}, 
which allow us to extend colorings to large regions,
and prove a number of elementary lemmas establishing their existence and properties.
In Section~\ref{sec:periodicitytheorems}, we use this machinery to prove Theorem~\ref{singleperiodic}.  Along the way, we provide a new proof of Boyle and Lind's Theorem 
(Theorem~\ref{BoyleLind}) adapted to our setting.
The remainder of the paper is devoted to the proof of 
Theorem~\ref{twoormore}.  
In Section~\ref{nk/2}, we show how the stronger complexity bound assumed 
in Theorem~\ref{twoormore} can be used to obtain additional control over the set of one-sided nonexpansive directions for $X_{\eta}$ 
and in Section~\ref{maintheorem}, 
we use the machinery from Section~\ref{nk/2} to complete the 
proof of  Theorem~\ref{twoormore}.

\subsection*{Acknowledgment}  
We thank Mike Boyle, Alejandro Maass, and Anthony Quas for helpful conversations and we particularly thank the referee for numerous helpful suggestions and corrections.

\section{Periodicity and generating sets}
\label{sec:background}

\subsection{One dimension: the Morse-Hedlund Theorem}

For studying periodicity, we ultimately rely on the Morse-Hedlund Theorem.  Although 
this result is classical, for completeness we include the statement and proof for the various versions we use.  We start by defining the complexities for each of the possible settings: the integers, the natural numbers and a finite interval.  

\begin{definition}
If $f\colon\Z\to\A$, define $Tf\colon\Z\to\A$ by $(Tf)(n):=f(n+1)$.  Define $P_f(n)$ to be the number of distinct functions of the form $(T^mf)\rst{\{0,1,\dots,n-1\}}$, where $m$ ranges over $\Z$.  

Similarly, if $f\colon\N\to\A$, define $Tf\colon\N\to\A$ and $P_f(n)$ to be the analogous quantities, but with $m$ ranging over $\N$.  

If $a\in\Z$ and $f\colon\{a,a+1,a+2,\dots,a+i-1\}\to\A$, define $Tf\colon\{a-1,a,\dots,a+i-2\}\to\A$ by $(Tf)(n):=f(n+1)$ and define $P_f(n)$ to be the number of distinct functions of the form $(T^mf)\rst{\{a,a+1,\dots,a+n-1\}}$, where $0\leq m\leq i-n$ and $0\leq n\leq i$.
\end{definition}

\begin{theorem}[Morse-Hedlund~\cite{MH}]\label{MorseHedlundTheorem}
Suppose $f\colon U\to\A$, where $U\subseteq\Z$ is one of $\Z$, $\N$, or an interval of the form $\{a,a+1,\dots,a+i-1\}$ for some $a\in\Z$, and suppose there exists $n_0\in\N$ such that $P_f(n_0)\leq n_0$.
\begin{enumerate}
\item If $U=\{a,a+1,\dots,a+i-1\}$ and $i>3n_0$, then the restriction of $f$ to the set $\{a+n_0,a+n_0+1,\dots,a+i-n_0\}$ is periodic of period at most $n_0$;
\item If $U=\N$, then $f\rst{\{x>n_0\}}$ is periodic of period at most $n_0$;
\item If $U=\Z$, then $f$ is periodic of period at most $n_0$.
\end{enumerate}
\end{theorem}
\begin{proof}
%Without loss of generality, suppose $n_0$ is chosen to be minimal with the property that $P_f(n_0)\leq n_0$.  Since $P_f(1)=|\A|>1$, we have that $n_0\geq2$.  Since $P_f(n_0-1)\leq P_f(n_0)$, we have that$P_f(n_0-1)=P_f(n_0)=n_0$.
%
Let $n_0\in\N$ be such that $P_f(n_0)\leq n_0$.

First suppose $U$ is a finite interval.  Let $a\in\Z$, $i\in\N$, $U=\{a,a+1,\dots,a+i-1\}$, and $i>3n_0$.  For $n\leq n_0$, 
 let $P_f(n_0,n)$ denote the number of distinct functions of the form $(T^mf)\rst{\{a,a+1,\dots,a+n-1\}}$, where $0\leq m\leq i-n_0$.  Note that $P_f(n_0,n_0)=P_f(n_0)$.  For $0\leq i_1<i_2\leq n_0$, we have that $P_f(n_0,i_1)\leq P_f(n_0,i_2)$.  If $P_f(n_0,1)=1$, then the restriction of $f$ to the set $\{a,a+1,\dots,a+i-n_0\}$ is constant, and hence periodic on the claimed interval, with period $1$.  Otherwise $P_f(n_0,1)>1$ and so there exists minimal $\alpha\leq n_0$ such that $P_f(n_0,\alpha)\leq\alpha$.  
It follows that $P_f(n_0,\alpha-1)=P_f(n_0,\alpha)$.  Then for any $0\leq m\leq i-n_0$, the restriction of $T^mf$ to the set $\{a,a+1,\dots,a+\alpha-2\}$ uniquely determines its restriction to the set $\{a,a+1,\dots,a+\alpha-1\}$ (note that this relies on the definition of $P_f(n_0, n)$ and the analogous statement for $P_f(n)$ does not suffice).  
%(note that this property would {\em not} follow if we defined $\alpha$ to be minimal such that $P_f(\alpha)\leq\alpha$, since not every pattern counted in $P_f(m-1)$ occurs as the ``left subword'' of a pattern counted in $P_f(m)$).  
By the Pigeonhole Principle, there exist $0\leq m_1<m_2\leq\alpha$ such that $(T^{m_1}f)\rst{\{a,a+1,\dots,a+\alpha-2\}}=(T^{m_2}f)\rst{\{a,a+1,\dots,a+\alpha-2\}}$.  It follows by induction that $T^{m_1}f$ and $T^{m_2}f$ agree on the set $\{a,a+1,\dots,a+i-n_0+\alpha-i_2\}$.  Therefore, $f$ is periodic with period $i_2-i_1\leq n_0$ on the set $\{a+i_1,a+i_1-1,\dots,a+i-n_0+\alpha\}$ and, in particular, on the claimed interval.

For $U=\N$, by taking $i=\infty$, the result follows using the same argument.  For $U=\Z$, for any $a,m\in\Z$, the restriction of $T^mf$ to the set $\{a,a+1,\dots,a+\alpha-2\}$ extends uniquely to the sets $\{a-1,a,a+1,\dots,a+\alpha-2\}$ and $\{a,a+1,\dots,a+\alpha-1\}$.  Hence there are at most $P_f(n_0,\alpha)\leq P_f(n_0)\leq n_0$ many functions of the form $T^mf$ and the result 
follows as before.  
\end{proof}

\subsection{Geometric notation and terminology}

If $R\subset\R^2$, we denote the convex hull of $R$ by $\conv(R)$.  A subset $\S\subseteq\ZZ$ is called {\em convex} if $\conv(\S)$ is closed and $\S=\conv(\S)\cap\ZZ$.
We view the boundary of a convex subset of $\ZZ$ as a (possibly infinite) convex polygon.  
The elements of $\S$ are  the integer points contained in and enclosed by this polygon.  
 \footnote{  
The assumption that the convex sets are closed avoids pathological behavior on the boundary, 
as for example  the convex hull of $\{(x,y)\in\ZZ\colon x<0\}\cup\{(0,0)\}$ is the set $\{(x,y)\in\R^2\colon x<0\}\cup\{(0,0)\}$.  Such behavior is avoided when we assume that the convex hull of a convex set is closed.}

We let $\partial\S$ denote the boundary of $\conv(\S)$.
An {\em extreme point} of a convex set $\S\subseteq\ZZ$ is a point in $\partial\S\cap\ZZ$ which is a vertex of the convex polygon $\partial\S$, 
and a {\em boundary edge} of $\S$ an edge of $\partial\S$.  We use $V(\S)$ to 
denote the set of
extreme points of $\S$ and $E(\S)$ to denote the set of boundary edges of $\S$. 

If $\S\subset\ZZ$ is convex and $\conv(\S)$ has positive area, our standard convention is that the boundary of $\S$ is positively oriented.  When $\left|\S\right|<\infty$, this orientation endows each $w\in E(\S)$ with a well-defined {\em successor} edge, denoted $\suc(w)\in E(\S)$ and a {\em predecessor} edge, 
denoted $\pred(w)\in E(\S)$.  In the case that $\left|\S\right|=\infty$, the definitions of successor and predecessor extend in the natural way, noting that there may be two edges without predecessors and two without successors (for example, for a strip).   
We extend the functions $\suc(\cdot)$ and $\pred(\cdot)$ to infinite convex sets in the natural way (leaving $\pred(w_{\alpha})$ and $\suc(w_{\omega})$ undefined).  With these conventions, each $w\in E(\S)$ inherits an orientation from the boundary of $\S$, and so we make a slight abuse of the notation by viewing $w\in E(\S)$ as both a set and an oriented line segment.  Thus we can refer to an oriented line in $\R^2$ as being {\em parallel} or {\em antiparallel} (or neither) to a given element of $E(\S)$.

We make use of two notions of size:
\begin{enumerate}
\item If $\S\subseteq\ZZ$, then $\left|\S\right|$ denotes the cardinality of $\S$.
\item If $w\subset\R^2$ is a line segment, then $\|w\|$ denotes the length of $w$.
\end{enumerate}
In particular, if $\S\subset\ZZ$ is a finite convex set and $w\in E(\S)$, then $\|w\|$ is the length of $w$, while $\left|w\cap\S\right|$ is the number of integer points on it.

We denote the $n$ by $k$ rectangle based at the origin by 
$$
R_{n,k}:=\left\{(x,y)\in\ZZ:0\leq x<n\text{, }0\leq y<k\right\}.
$$

\subsection{The discrepancy function and $\eta$-generating sets} 

We use the discrepancy function to derive a number of useful properties of functions $\eta\in\A^{\ZZ}$ satisfying $P_{\eta}(R_{n,k})\leq nk$ for some $n,k\in\N$.

\begin{lemma}\label{generated}
Suppose $\S\subset\ZZ$ is finite and $\left|\S\right|\geq2$.  If $x\in\S$, then either every $\eta$-coloring of $\S\setminus\{x\}$ extends uniquely to an $\eta$-coloring of $\S$ or $D_{\eta}(\S\setminus\{x\})\leq D_{\eta}(\S)$.
\end{lemma}
\begin{proof}
If there is some $\eta$-coloring of $\S\setminus\{x\}$ that extends non-uniquely to an $\eta$-coloring of $\S$, then $P_{\eta}(\S\setminus\{x\})<P_{\eta}(\S)$.  Thus
\begin{equation*}
D_{\eta}(\S\setminus\{x\})=P_{\eta}(\S\setminus\{x\})-\left|\S\setminus\{x\}\right|\leq(P_{\eta}(\S)-1)-(\left|\S\right|-1)=D_{\eta}(\S). \qedhere
\end{equation*}
\end{proof}

Motivated by Lemma~\ref{generated}, we make the following definition:

\begin{definition}
If $\S\subset\ZZ$ is a finite set and $x\in\S$, we  say that $x$ is {\em $\eta$-generated} by $\S$ if every $\eta$-coloring of $\S\setminus\{x\}$ extends uniquely to an $\eta$-coloring of $\S$.    A finite, nonempty, convex subset of $\ZZ$ for which every
extreme point is generated is called a {\em weak $\eta$-generating set}.
\end{definition}

\begin{lemma}\label{generated2}
Suppose $\S\subset\ZZ$ is a finite, convex set and $D_{\eta}(\S)\leq0$.  Let $\T$ be a minimal set (with respect to partial ordering by inclusion) among all nonempty convex subsets of $\S$ with discrepancy at most $D_{\eta}(\S)$.  Then  $\T$ is a weak $\eta$-generating set, and if $y\in V(\T)$, $D_{\eta}(\T\setminus\{y\})=D_{\eta}(\T)+1$.
\end{lemma}
\begin{proof}
We proceed by contradiction.  Let $x\in\T$ be a
extreme point that is not generated.  Since any one 
element set has discrepancy $\left|\mathcal{A}\right|-1>0$, $\T$ must contain at least two elements; in particular $\T\setminus\{x\}$ is nonempty.  
Furthermore, $\T\setminus\{x\}$ is convex and 
by Lemma~\ref{generated},  $D_{\eta}(\T\setminus\{x\})\leq D_{\eta}(\T)$,  
a contradiction of the minimality of $\T$.

For the change in discrepancy, note that any
extreme point $y\in V(\T)$ is generated and so $P_{\eta}(\T)=P_{\eta}(\T\setminus\{x\})$.  The statement follows since $\left|\T\right|=\left|\T\setminus\{x\}\right|+1$.
\end{proof}

\begin{corollary}\label{generatingset}
If $\S\subset\ZZ$ is a finite, convex set with $\eta$-discrepancy $d\leq0$, then $\S$ contains a (strictly decreasing) nested family of weak $\eta$-generating subsets
$$
\S_1\supset\ldots\supset\S_{\left|d\right|+1}.
$$
\end{corollary}
\begin{proof}
Let $\S_1$ be a nonempty, convex subset of $\S$ which is minimal (with respect to inclusion) among all convex subsets having discrepancy at most $d$; such a set must exist because a one element subset of $\S$ has positive discrepancy.  By Lemma~\ref{generated2},  $\S_1$ is weak $\eta$-generating  and contains at least two elements (because it has nonpositive $\eta$-discrepancy).

Suppose that for some $i<\left|d\right|+1$, we have constructed weak $\eta$-generating sets $\S_1\supset\S_2\supset\ldots\supset\S_i$ such that $D_{\eta}(\S_j)=D_{\eta}(\S)+j-1$ for all $1\leq j\leq i$.  Let $x_i\in V(\S_i)$ and set $\tilde{\S}_i:=\S_i\setminus\{x_i\}$.  Then 
by Lemma~\ref{generated2}, $D_{\eta}(\tilde{S}_i)=D_{\eta}(\S_i)+1=D_{\eta}(\S)+i\leq0$.  We can then pass to a subset $\S_{i+1}$ of $\tilde{S}_i$ which is minimal among all convex subsets of $\tilde{\S}_i$ that have discrepancy at most $D_{\eta}(\tilde{S}_i)$, 
and note that  by Lemma~\ref{generated2}, 
$\S_{i+1}$ is  weak $\eta$-generating.  This process continues for at least $\left|d\right|+1$ steps.
\end{proof}

\begin{corollary}\label{generated3}
Suppose that $\S\subset\ZZ$ is a finite, convex set with $\eta$-discrepancy $d\leq0$.  For any $i\in\N$ and any $x_1,\dots,x_i\in\S$ such that $\S\setminus\{x_1,\dots,x_i\}$ is convex and nonempty, we have that $D_{\eta}(\S\setminus\{x_1,\dots,x_i\})\leq D_{\eta}(\S)+i$.
\end{corollary}
\begin{proof}
Since $P_{\eta}(\S\setminus\{x_1,\dots,x_i\})\leq P_{\eta}(\S)$, we have that 
\begin{align*}
D_{\eta}(\S\setminus\{x_1,\dots,x_i\})&=P_{\eta}(\S\setminus\{x_1,\dots,x_i\})-|\S\setminus\{x_1,\dots,x_i\}| \\
&\leq P_{\eta}(\S)-|\S|+i =D_{\eta}(\S)+i.\qedhere
\end{align*}
\end{proof}

We remark that the inequality in the corollary is strict unless every $\eta$-coloring of $\S\setminus\{x_1,\dots,x_i\}$ extends uniquely to an $\eta$-coloring of $\S$.

This corollary becomes relevant in our constructions: if we know that $D_{\eta}(\S)<0$, we are free to remove any $\left|D_{\eta}(\S)\right|$ elements from $\S$ and are guaranteed that the resulting set contains a weak $\eta$-generating subset (provided it is convex).

A key fact about the discrepancy function which is crucial in Sections~\ref{nk/2} and~\ref{maintheorem} is the following:

\begin{lemma}\label{nonuniqueextension1}
Suppose that $\S\subset\ZZ$ is a convex weak $\eta$-generating set for which every nonempty proper convex subset has strictly larger $\eta$-discrepancy and let $w\in E(\S)$.  If $\S\setminus w\neq\emptyset$, then there are at most $\left|w\cap\S\right|-1$ $\eta$-colorings of $\S\setminus w$ that extend non-uniquely to $\eta$-colorings of $\S$.
\end{lemma}
\begin{proof}
We have $D_{\eta}(\S\setminus w)>D_{\eta}(\S)$ (by assumption) and $\left|\S\setminus w\right|=\left|\S\right|-\left|w\cap\S\right|$.  So, $P_{\eta}(\S\setminus w)>P_{\eta}(\S)-\left|w\cap\S\right|$.  On the other hand, there are no more than $P_{\eta}(\S)-P_{\eta}(\S\setminus w)$ distinct $\eta$-colorings of $\S\setminus w$ that extend non-uniquely to an $\eta$-coloring of $\S$.
\end{proof}

This lemma leads to the key definition: 
\begin{definition}
A weak $\eta$-generating set $\S\subset\Z^2$ is an {\em $\eta$-generating set} if 
every nonempty proper convex subset has strictly large $\eta$-discrepancy. 
\end{definition}

If $\S\subset\Z^2$ is weak $\eta$-generating with $D_\eta(\S)\leq 0$, then if 
$\T$ is minimal among convex subsets of $\S$ with $D_\eta(\T)\leq D_\eta(\S)$, 
we have that $\T$ is $\eta$-generating.  In particular, if $D_\eta(\S)\leq 0$, then $\S$ 
contains an $\eta$-generating set.

\begin{remark}\label{rem:standard}
Many of our constructions assume that we have fixed a function $\eta\colon\ZZ\to\A$, an $\eta$-generating set $\S$, and an edge $w\in E(\S)$.  It is often convenient to assume that the oriented line segment $w$ points vertically downward. Such 
an assumption is not restrictive: if $A\in SL_2(\Z)$, then $A^{-1}(\S)$ is convex and
$$
D_{\eta\circ A}(A^{-1}(\S))=D_{\eta}(\S).
$$
Therefore $\S$ is an $\eta$-generating set if and only if $A^{-1}(\S)$ is an $(\eta\circ A)$-generating set, and we have no change in the discrepancy.  Since $SL_2(\Z)$ takes positively oriented polygons to positively oriented polygons and acts transitively on directed rational lines through the origin in $\R^2$, in constructions that only rely on $\S$ being $\eta$-generating, we can always make a change of coordinates such that a given edge $w$ is vertical with downward orientation.  
This is useful to simplify the notation and the pictures used to explain various definitions and arguments.
\end{remark}

\begin{lemma}\label{heightlemma}
Suppose that $\S\subset\ZZ$ is a finite convex set and there are two edges $w_1, w_2\in E(\S)$ that are antiparallel.  Then any line parallel to $w_1$ that has nonempty intersection with $\S$ contains at least $\min_{i=1,2}\{\left|w_i\cap\S\right|-1\}$ integer points.
\end{lemma}
\begin{proof}
Without loss of generality, suppose that $\left|w_1\cap\S\right|\leq\left|w_2\cap\S\right|$.  If $\ell$ is a line parallel to $w_1$ that has nonempty intersection with $\S$, then by convexity $\|\ell\cap \conv(\S)\|\geq\|w_1\|$ (recall that $\|\cdot\|$ denotes the length of a line segment in $\R^2$).  The distance between any two consecutive integer points on $\ell$ is the same as the distance between two consecutive integer points on the line determined by $w_1$, since the two sets differ only by a translation taking the integer points on one to the integer points on the other.  Since the distance between $\left|w_1\cap\S\right|$ integer points on a line parallel to $\ell$ is exactly $\|w_1\|$, any interval in $\ell$ of length at least $\|w_1\|$ must contain at least $\left|w_1\cap\S\right|-1$ integer points.  In particular,  $\ell\cap \conv(\S)$ does.
\end{proof}

We finish this subsection with two quick applications of generating sets.
The first is a relation that we use in the sequel to eliminate irrational nonexpansive directions.  We define:

\begin{definition}
A convex set $H\subset\ZZ$ is called a {\em half plane} if $conv(H)$ has positive area and $E(H)$ contains only a single edge.  In this case, the unique boundary edge is a line in 
$\R^2$.  Given $\vec v\in\R^2\setminus\{\vec0\}$, a $\vec v$-half plane is a half plane whose (positively oriented) boundary edge is parallel to $\vec v$.  

If $\S\subset\ZZ$ is convex and $\vec v\in\R^2\setminus\{\vec0\}$, then the intersection of all $\vec v$-half planes containing $\S$ is a $\vec v$-half plane whose boundary $\ell(\vec v,\S)$ has nonempty intersection with $\partial\S$.  We call $\ell(\vec v,\S)$ the {\em support line} of $\S$ determined by $\vec v$.
\end{definition}

By definition, a half plane is closed and contains the integer points on its boundary edge.
  Note that $\ell(\vec v,\S)\cap\conv(\S)$ is either a boundary edge or 
an extreme point of $\S$.  When the intersection is
an extreme point of $\S$, $\ell(\vec v,\S)\cap\S\neq\emptyset$.  
If $\S$ is finite and $\ell(\vec v,\S)\cap\conv(\S)$ is a boundary edge of $\S$, then $\ell(\vec v,\S)\cap\S\neq\emptyset$.

\begin{lemma}
\label{lemma:no-irrational}
If $\eta\colon\ZZ\to\A$ and there exist $n,k\in\N$  such that $P_\eta(n,k)\leq nk$ and $L\subset\R^2$ 
is an irrational line through the origin, then $L$ is expansive on $X_\eta$.  
\end{lemma}

\begin{proof}
Let $\S$ be an $\eta$-generating set.  Choose an expansiveness radius $r> 0$ such 
that $\S$ is contained in the set 
$$
U := \{\vec u\colon d(\vec u, L) < r\}.
$$
Choose an orientation on $L$ and set $\ell:=\ell(L,\S)$, and $\vec w = \S\cap \ell\in V(\S)$.  
Define 
$$
c := \inf_{\vec y\in \S\setminus\{\vec w\}}d(\vec y, \ell).
$$
Since $\S$ is finite, $c> 0$. 

We claim that $\eta\rst U$ determines all of $\eta$.   If not, set
$$d := \inf\{ d(\vec y, L)\colon \eta\rst U \text{ does not determine } \eta(\vec y)\}.
$$
Then $d$ is finite (or we are already finished) and $d \geq r$. 
Defining $U_R = \{\vec u \colon d(\vec u, L ) < R\}$, then 
$\eta\rst U$ determines $\eta\rst{U_{d'}}$, where $d'$ is either $d-c/2$, when this 
is positive, or $d/2$, otherwise.  Choose $\vec y\in \ZZ$ such that $d(\vec y, L)\leq d+c/4$ and 
such that $\eta(\vec y)$ is not determined by $\eta\rst U$.  
Translating $\S$, we can assume that $\vec w = \vec y$.  
Then there are two possibilities.  
The first is that  $\S\setminus\{\vec y\} \subset U_{d'}$, and  then since $\S$ is $\eta$-generating we 
have a contradiction.  Otherwise, $\S\setminus\{\vec w\}\cap U_{d'}=\emptyset$, and 
then replacing $L$ in the proof by its opposite orientation, the same argument leads 
to a contradiction.  
\end{proof}

The second application relates to entropy: 
\begin{definition}
Suppose $\eta\colon\ZZ\to\A$ and $\S\subset\ZZ$ is finite.  Define
$$
X_{\S}(\eta):=\left\{f\colon\ZZ\to\A\text{ such that } \mathcal{W}(\S,f)\subseteq\mathcal{W}(\S,\eta)\right\}
$$
to be the {\em $\ZZ$-subshift of finite type generated by the $\S$-words of $\eta$}.  

An {\em $(\S,\eta)$-coloring} of a set $\T\subseteq\ZZ$ is any function of the form $\{f\rst{\T}\colon f\in X_{\S}(\eta)\}$.
\end{definition}

(In more common terminology, if $F_{\S}:=\A^{\S}\setminus\mathcal{W}(\S,\eta)$ is the set of all $\S$ words {\em not} occurring in $\eta$, then $X_{\S}(\eta)$ is the $\ZZ$-subshift of finite type whose set of forbidden words is $F_{\S}$.)

\begin{lemma}\label{entropy}
If $\eta\colon\ZZ\to\A$ and there is an $\eta$-generating set $\S\subseteq\ZZ$, then for any finite $\S^{\prime}\supseteq\S$ the $\ZZ$-dynamical system $(X_{\S^{\prime}}(\eta),\{T^{\vec u}\}_{\vec u\in\ZZ})$ has topological entropy zero.
\end{lemma}
\begin{proof}
Choose $n,k\in\N$ such that $\S^{\prime}\subseteq R_{n,k}$.  We claim that for any $n^{\prime}>2n$ and $k^{\prime}>2k$, the function $\eta\rst{R_{n^{\prime},k^{\prime}}}$ is determined by is restriction to the (square annular) set
$$
[0,n^{\prime}-1]\times[0,k^{\prime}-1]\setminus[n,n^{\prime}-n]\times[k,k^{\prime}-k].
$$
Since any translation of an $\eta$-generating set is $\eta$-generating, we can assume without loss that $\S\subseteq R_{n,k}$ and $\S\cap(\{n-1\}\times[0,k-1])\neq\emptyset$.  Then $\S+(1,0)$ lies entirely inside the square annular region.  It follows that there exists $j\in\N$ and a vertex $v\in V(\S)$ such that $V\setminus\{v\}+(1,j)$ lies inside the square annular region, and $v+(1,j)$ is the point $(n,k)$.  Since $\S+(1,j)$ is $\eta$-generating, it follows that any $\eta$-coloring of the square annular region extends uniquely to the point $(n,k)$.  By induction it follows that, for all $0\leq j\leq k^{\prime}-k$, any $\eta$-coloring of the square annular region extends uniquely to an $\eta$-coloring of $\{n\}\times[k,k+j]$.  A similar induction shows that it extends uniquely to an $\eta$-coloring of $[n,n+i]\times[k,k+j]$ for any $0\leq j\leq k^{\prime}-k$ and any $0\leq i\leq n^{\prime}-n$.  The claim follows.

The claim gives $P_{\eta}(R_{n^{\prime},k^{\prime}})\leq\left|\A\right|^{2nk^{\prime}+2kn^{\prime}-4nk}$, and so
\[
\lim_{n^{\prime}\to\infty}\frac{1}{(n^{\prime})^2}\log P_{\eta}(R_{n^{\prime},n^{\prime}})=0.\hfill
 \qedhere
\]
\end{proof}

\begin{remark}
In dimension one, the analog of Lemma~\ref{entropy} holds and leads to a proof of the Morse-Hedlund Theorem: a one-dimensional subshift of finite type either has positive entropy or every element is periodic.  In dimension two, there are zero entropy $\ZZ$-subshifts of finite type that do not contain any periodic elements.  Thus Lemma~\ref{entropy} serves as an indication that generating sets are dynamically interesting, but does not seem to provide an approach to Nivat's conjecture.
\end{remark}

\subsection{Ambiguous half planes and periodicity}
In this section, we develop a relationship between the notions of nonexpansivity and periodicity.  The main results are Lemma~\ref{ambiguousperiod} and Corollary~\ref{ambiguouscorollary}.  Unfortunately, to state and prove them we need a significant amount of terminology.  This is introduced in the following four definitions (with Figures~\ref{figure5} and~\ref{borders} accompanying Definitions~\ref{envelopeddef2} and~\ref{border-def}).  Their complicated nature is necessitated by the need to bound the periods appearing in Lemma~\ref{ambiguousperiod}.

\begin{definition}\label{ambiguousdef}
If $\S\subset\ZZ$, $\T_1\subset\T_2\subseteq\ZZ$, then a coloring $f\in X_{\S}(\eta)$ is {\em $(\S,\T_1,\T_2,\eta)$-ambiguous} if there exist $g_1, g_2\in X_{\S}(\eta)$ such that $g_1\rst{\T_1}=g_2\rst{\T_1}=f$ but $g_1\rst{\T_2}\neq g_2\rst{\T_2}$.
\end{definition}

Note that $g_1, g_2$ do not necessarily lie in $X_{\eta}$, but only in $X_{\S}(\eta)$, the set of $\ZZ$-colorings whose $\S$-words coincide with the $\S$-words of $\eta$.  
Ambiguity becomes especially interesting when $\T_2\supset\T_1$ is produced in some 
way by $\T_1$, and this is captured in Definition~\ref{envelopeddef2}. 

\begin{definition}[Enveloping set]\label{envelopeddef}
If $\vec v_1,\dots,\vec v_n\in\ZZ\setminus\{\vec 0\}$ is a collection of vectors, we say that a convex set $\T\subseteq\ZZ$ is {\em $\{\vec v_1,\dots,\vec v_n\}$-enveloped} if for every $w\in E(\T)$, there exists $i\in\{1, \ldots, n\}$ such that $w$ is parallel to $\vec v_i$.  An {\em enveloping set} for $\T$ is a set of vectors that envelops it.  A {\em minimal enveloping set} for $\T$ is a collection of vectors that envelops $\T$ and such that no proper subset suffices.
\end{definition}

Given a convex region, we define an extension over an edge of the region (it may help to refer to Figure~\ref{figure5} while reading this definition).  The definition splits into 
several cases depending on the type of edge (recall that a rational line in $\R^2$ is a line of rational slope that contains a rational point):
\begin{definition}[See Figure~\ref{figure5}]\label{envelopeddef2}
Suppose $\T\subseteq\ZZ$ is convex, $\conv(\T)$ has positive area, and each $w\in E(\T)$ determines a rational line in $\R^2$. 
\begin{enumerate}
\item Suppose $w$ points vertically downward.  Without loss of generality, assume it 
is a subset of the $y$-axis.
	\begin{enumerate}
	\item If $w$ has both a successor edge and a predecessor edge in $E(\T)$, choose $a,b,c,d\in\Q$ such that $\pred(w)\subseteq\{(x,y)\in\R^2\colon y=ax+b\}$ and $\suc(w)\subseteq\{(x,y)\in\R^2\colon y=cx+d\}$.  If there is an integer $\Delta<0$ such that $c\Delta+d\leq a\Delta+b$ and $c\Delta+d, a\Delta+b\in\Z$, then for maximal such $\Delta$ (i.e., with minimal absolute value), we define the {\em $w$-extension $\Ext_w(\T)$ of $\T$} to be the set
	$$
	\Ext_w(\T):=\T\cup\left\{(x,y)\in\ZZ\colon cx+d\leq y\leq ax+b\text{, }\Delta\leq x<0\right\}.
	$$
	If no such $\Delta$ exists, we define $\Ext_w(\T):=\T$.
	\item If $w$ has a successor edge but does not have a predecessor edge, choose $c,d\in\Q$ such that $\suc(w)\subseteq\{(x,y)\in\R^2\colon y=cx+d\}$.  Choose maximal $\Delta<0$ such that $c\Delta+d\in\Z$.  Then we define the {\em $w$-extension $\Ext_w(\T)$ of $\T$} to be
	$$
	\Ext_w(\T):=\T\cup\left\{(x,y)\in\ZZ\colon cx+d\leq y\text{, }\Delta\leq x<0\right\}.
	$$
	\item If $w$ has a predecessor edge but does not have a successor edge, choose $a,b\in\Q$ such that $\pred(w)\subseteq\{(x,y)\in\R^2\colon y=ax+b\}$.  Choose maximal $\Delta<0$ such that $a\Delta+b\in\Z$.  Then we define the {\em $w$-extension $\Ext_w(\T)$ of $\T$} to be the set
	$$
	\Ext_w(\T):=\T\cup\left\{(x,y)\in\ZZ\colon y\leq ax+b\text{, }\Delta\leq x<0\right\}.
	$$
	\item If $w$ has neither a predecessor edge nor a successor edge (i.e. if $\T$ is a half-plane whose boundary has rational slope), then $\Ext_w(\T)$ is the smallest half plane in $\ZZ$ that strictly contains $\T$.  It is immediate that the boundary of $\Ext_w(\T)$ is parallel to $w$.
	\end{enumerate}
\item If $w$ does not point vertically downward, let $A\in SL_2(\Z)$ be such that $Aw$ points vertically downward.  We define the $w$-extension of $\T$ to be $A^{-1}(\Ext_{Aw}(A\T))$. (Note that this set does not depend on the choice of $A$.)
\end{enumerate}

\begin{figure}[ht]
     \centering
  \def\svgwidth{\columnwidth}
        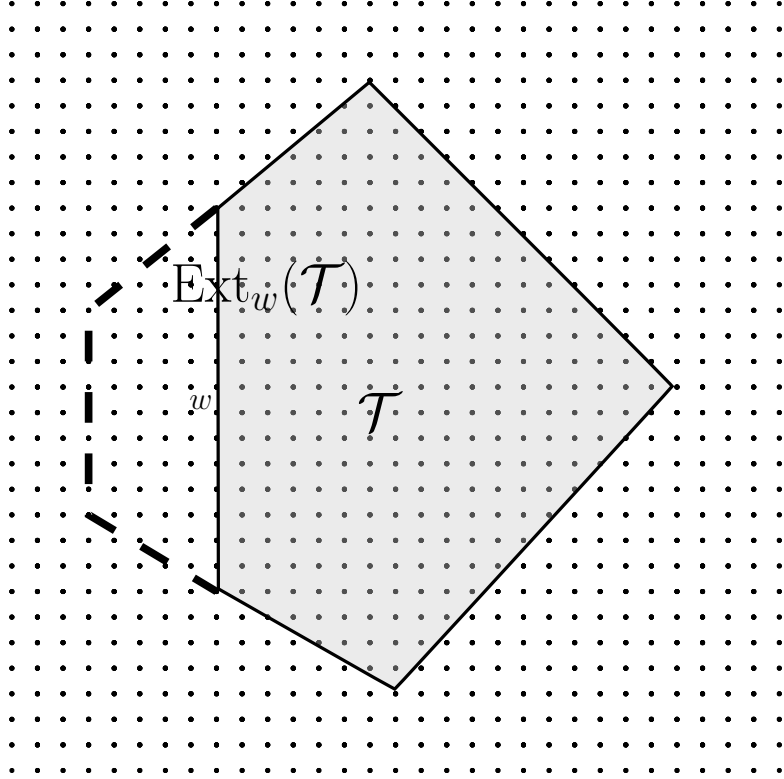
        \setlength{\abovecaptionskip}{-65mm}
     \caption{$\T$ is the set of integer points enclosed by the solid 
lines.  If $w$ points vertically downward, then $\Ext_w(\T)$ is enclosed 
by the dashed lines and the nonvertical solid lines.  The set 
$\Ext_w(\T)\setminus\T$ decomposes into five vertically aligned sets 
which determine the subextensions.  The depth of the extension is five.}
     \label{figure5}
\end{figure}

It follows that $\Ext_w(\T)$ is a convex, $E(\T)$-enveloped set containing $\T$, and may be $\T$ itself.  If $\Ext_w(\T)$ strictly contains $\T$, then there is a finite collection of lines $\ell_1,\dots,\ell_m$ such that
\begin{itemize}
\item $\ell_i$ is parallel to $w$ for all $i$;
\item $\ell_i\cap(\Ext_w(\T)\setminus\T)$ is nonempty and convex for all $i$;
\item we can decompose $\Ext_w(\T)\setminus\T$ into the disjoint union:
$$
\Ext_w(\T)\setminus\T=\bigsqcup_{i=1}^m\left(\ell_i\cap \Ext_w(\T)\right).
$$
\end{itemize}
In this case, $m$ is called the {\em depth} of the extension.  For $1\leq j\leq m$,
$$
\T\cup\bigcup_{i=1}^j\left(\ell_i\cap \Ext_w(\T)\right)
$$
is the {\em $(w,j)$-subextension} of $\T$ (note that the $(w,j)$-subextension may not be $E(\T)$-enveloped).  The $(w,0)$-subextension of $\T$ is defined to be $\T$ itself.
\end{definition}

\begin{remark}
The assumption in Definition~\ref{envelopeddef2} that the edges of $\T$ determine rational lines may seem restrictive.  However, observe that whenever $\T$ is finite its convex hull only has edges of this type.  When $\T$ is infinite, the convex sets that arise in our constructions also turn out to be bounded by edges that determine rational lines.
\end{remark}

\begin{definition}\label{border-def}
Suppose $\S,\T\subset\ZZ$ are convex and $w\in E(\T)$.  If $\Ext_w(\T)\neq\T$ and if $f\in X_{\S}(\eta)$, we say that $f\rst{\T}$ is {\em $(\S,w,\eta)$-ambiguous} if it is $(\S,\T,\Ext_w(\T),\eta)$-ambiguous.

Suppose $H$ is a half plane and $w\in E(H)$ is its unique boundary edge.  If $w$ is a rational line and $f\in X_{\S}(\eta)$, we say that $f\rst{H}$ is {\em $(\S,\eta)$-ambiguous} if it is $(\S,w,\eta)$-ambiguous.
\end{definition}

\begin{remark}
We remark that ambiguity is a property of a specific coloring of $\ZZ$, rather than a dynamical property of $X_{\mathcal{\S}}(\eta)$ like expansivity, nonexpansivity or the one-sided versions that are introduced in Section~\ref{expansivesection}.  An $(S,\eta)$-ambiguous coloring of a half plane can be thought of as a ``witness'' to the nonexpansiveness of the direction that borders the half plane.  An $(\S,\T_1,\T_2,\eta)$-ambiguous coloring of a finite set $\T_1$ similarly witnesses a finite analog of nonexpansiveness. 
\end{remark}

\begin{figure}[ht]
     \centering
  \def\svgwidth{\columnwidth}
        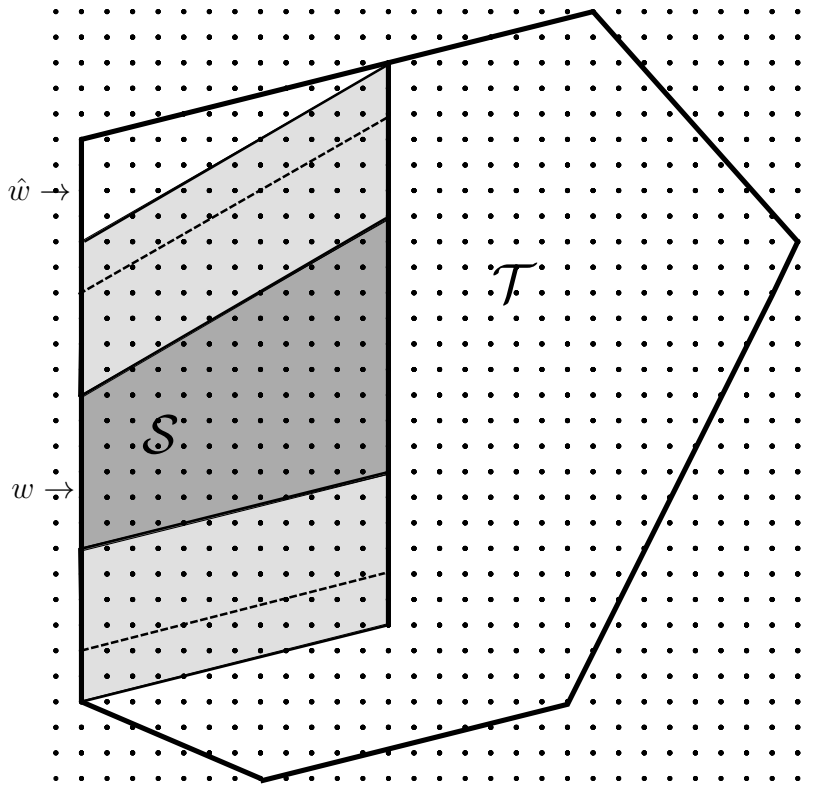
                \setlength{\abovecaptionskip}{-65mm}
	\caption{Suppose $\T$ is the set of integer points contained in the largest convex region, $\S$ is the set of integer points in the dark grey region, $w\in E(\S)$ is the downward oriented vertical line segment on the left side of $\S$, and $\hat{w}\in E(\T)$ is the downward oriented vertical line segment on the left side of $\T$.  Then the $(\S,w)$-border of $\T$ is the set of integer points in the union of the light grey regions and the dark grey region.  The $2$-interior of the $(S,w)$-border of $\T$ is the set of integer points in the region enclosed between the dashed lines and the vertical line segments connecting them.}
\label{borders}\end{figure}

\begin{definition}[See Figure~\ref{borders}]
If $\T$ is a convex set, $\S\subset\T$ is convex, and $w\in E(\S)$ is parallel to an edge $w^*\in E(\T)$, let $V_{\S,\T,w}$ be the (possibly empty) set of translations that take $\S$ to a subset of $\T$ such that the translated edge of $\S$ is contained in the edge of $\T$ parallel\footnote{We note that since the boundaries of $\S$ and $\T$ are endowed with orientations,  the edge of $\T$ parallel to $w$ uniquely identifies a single edge.} to $w$:
$$
V_{\S,\T,w}=\left\{\vec v\in\ZZ\colon(\S+\vec v)\subseteq\T\text{, }(w+\vec v)\subseteq w^*\right\}.
$$
When $V_{\S,\T,w}\neq\emptyset$, there exist vectors $\vec a_{\S,\T,w}, \vec b_{\S,\T,w}\in\ZZ$ such that 
$$
V_{\S,\T,w}=\{\vec a_{\S,\T,w}+\lambda\vec b_{\S,\T,w}\colon\lambda\in I\},
$$
where
$$
I=\begin{cases}
\{0,1,\dots,\left|V_{\S,\T,w}\right|-1\} & \text{ if } \| w^*\|<\infty; \\
\N\cup\{0\} & \text { if } w^* \text{ is a semi-infinite line}; \\
\Z & \text{ if } w^* \text{ is a line}.
\end{cases}
$$
Let $I_{\min}, I_{\max}$ be the minimum and maximum elements of $I$, respectively (allowing $I_{\min}=-\infty$ and $I_{\max}=+\infty$ if necessary).  Then the {\em $(\S,w)$-border} of $\T$ is the set
$$
\bigcup_{\vec v\in V_{\S,\T,w}}(\S+\vec v).
$$

For integers $0\leq g<\frac{1}{2}(I_{\max}-I_{\min})$, the {\em $g$-interior} of the $(\S,w)$-border is the set
$$
\bigcup_{\lambda=I_{\min}+g}^{I_{\max}-g-1}(\S+\vec a_{\S,\T,w}+\lambda\vec b_{\S,\T,w}).
$$
\end{definition}

We now show that ambiguity is a source of periodicity:
\begin{lemma}
\label{ambiguousperiod}
Suppose $\eta\colon\ZZ\to\A$, $\S\subset\ZZ$ is an $\eta$-generating set 
and there exist antiparallel $w_1, w_2\in E(\S)$.    Suppose $\left|w_1\right|\leq\left|w_2\right|$, $H$ is a $w_1$-half plane, and the restriction of $f\in X_{\S}(\eta)$ to $H$ is $(\S,\eta)$-ambiguous.  Then the $(\S\setminus w_1,w_1)$-border of $H$ is periodic with period vector parallel to $w_1$.  Its period is at most $\left|w_1\cap\S\right|-1$.
\end{lemma}
\begin{proof}
Without loss of generality (see Remark~\ref{rem:standard}), we assume that $w_1$ and $w_2$ are vertical, and $w_1$ points downward.  Let $h:=\left|w_1\cap\S\right|-1$ and for each vertical line $\ell$ with nonempty intersection with $\S$, let $\vec x_{\ell}$ denote the bottom-most element of $\ell\cap\S$.  By Lemma~\ref{heightlemma},
$$
\W:=\bigcup_{i=0}^{h-1}\left\{\vec x_{\ell}+(0,i)\in\ZZ\colon\ell\text{ vertical, }\ell\cap\S\neq\emptyset\right\}\subseteq\S.
$$
Define $\tilde{S}:=\S\setminus w_1$, $\tilde{\W}:=\W\setminus w_1$, and fix a vector $\vec u\in\ZZ$ such that $\tilde{\W}+\vec u$ is contained in the $(\tilde{\S},w_1)$-border of $H$.

We claim that for any $\lambda\in\Z$, the $\eta$-coloring of $\tilde{S}$ given by $f\rst{\tilde{\S}+\vec u+(0,\lambda)}$ has at least two extensions to an $\eta$-coloring of $\S$.  Instead, suppose not.  Then the coloring $f\rst{H}$ uniquely determines the $\eta$-coloring of $H\cup\{\S+\vec u+(0,\lambda)\}$, which in particular determines the $\eta$-coloring of all but one of the elements of the set $\S+\vec u+(0,\lambda+1)$.  Since $\S$ is $\eta$-generating, this uniquely determines the $\eta$-coloring of $H\cup(\S+\vec u+(0,\lambda+1))$.  Now for $i\geq0$, suppose the $\eta$-coloring of 
$$
H\cup\{\S+\vec u+(0,\lambda+j)\colon 0\leq j\leq i\}
$$
has been determined.  Then the $\eta$-coloring of all but one of the elements of the set $\S+\vec u+(0,\lambda+i+1)$ is determined.  Since $\S$ is $\eta$-generating, this determines the $\eta$-coloring of $H\cup\{\S+\vec u+(0,\lambda+i+1)\}$.  By induction, this holds for all $i\geq0$.  Similarly for all $i\leq0$.  But this contradicts the ambiguity of the $\eta$-coloring of $H$.

Recall that since $\S$ is $\eta$-generating, $D_{\eta}(\tilde{S})>D_{\eta}(\S)$ and so $P_{\eta}(\tilde{S})>P_{\eta}(S)-\left|w_1\cap\S\right|$.  Therefore, the number of $\eta$-colorings of $\tilde{S}$  that do not  uniquely extend to $\eta$-colorings of $\S$ is at most $h=|w_1\cap S|-1$.  In particular, there are at most $h$ $\eta$-colorings of $\tilde{\W}$ that arise as the restriction of $f$ to a set of the form $\tilde{\W}+\vec u+(0,i)$, where $i\in\Z$.

Define a color set $\tilde{\A}$ whose colors are $\eta$-colorings of $\{\vec x_{\ell}\colon\ell\text{ is vertical, }\ell\cap\tilde{\S}\neq\emptyset\}$ occurring as the restriction of $f$ to a set of the form
$$
B_i:=\{\vec x_{\ell}\colon\ell\text{ is vertical, }\ell\cap\tilde{\S}\neq\emptyset\}+\vec u+(0,i),
$$
where $i\in\Z$.  Define $g\colon \Z\to\tilde{\A}$ by $g(i)=f\rst{B_i}$.  Then the (one-dimensional) block complexity $P_g(h)$ is the number of $\eta$-colorings of $\tilde{\W}$ occurring as the restriction of $f$ to a set of the form $\tilde{\W}+\vec u+(0,i)$.  But we have shown that this is at most $h$, so in particular $P_g(h)\leq h$.  By the Morse-Hedlund Theorem, $g$ is periodic with period at most $h$.  The result now follows from the definition of $g$.  
\end{proof}

The proof of Lemma~\ref{ambiguousperiod} holds in a slightly more general setting, and we make use of this in Section~\ref{nk/2}.

\begin{corollary}\label{ambiguouscorollary}
Suppose $\eta\colon\ZZ\to\A$ and there exists a finite, convex set $\S\subset\ZZ$ such that
\begin{enumerate}
\item There exists $w\in E(\S)$ such that for any line $\ell$ parallel to $w$ that has nonempty intersection with $\S$, we have $\left|\ell\cap\S\right|\geq\left|w\cap\S\right|-1$;
\item The two endpoints of $w$ are $\eta$-generated by $\S$;
\item $D_{\eta}(\S\setminus w)>D_{\eta}(\S)$.
\end{enumerate}
Suppose further that $\T\subset\ZZ$ is a convex set and there is an edge $w^*\in E(\T)$ parallel to $w$ and such that $Ext_{w^*}(\T)\neq\T$.  Finally suppose that the $(\S\setminus w,w)$-border of $Ext_{w^*}(\T)$ has nonempty $(\left|w\cap\S\right|-1)$-interior.

If $f\in X_{\S}(\eta)$ and $f\rst{\T}$ is $(\T,w^*,\eta)$-ambiguous, then there is some $j$ between $0$ and the depth of the $w^*$-extension of $\T$ such that the restriction of $f$ to the $(\left|w\cap\S\right|-1)$-interior of the $(\S\setminus w,w)$-border of the $(w,j)$-subextension of $\T$ is periodic with period vector parallel to $w$.  Moreover, the period is at most $\left|w\cap\S\right|-1$.
\end{corollary}
\begin{proof}
Choose $j$ to be the largest integer such that $f\rst{\T}$ extends uniquely to the $(w,j)$-subextension of $\T$.  Thereafter, the proof is identical to that of Lemma~\ref{ambiguousperiod}, except that the application of the Morse-Hedlund Theorem is for a finite (or semi-infinite) interval in $\Z$, rather than to $\Z$. 
\end{proof}

\section{One-sided expansiveness and periodicity}
\label{sec:periodicitytheorems}

\subsection{Nonexpansive directions}
\label{expansivesection}

Suppose that $\eta\colon\ZZ\to\A$ satisfies $D_{\eta}(R_{n,k})\leq0$ for some $n, k\in\N$.  By Corollary~\ref{generatingset}, there is an $\eta$-generating subset $\S\subseteq R_{n,k}$.  A half plane cannot have an $(\S,\eta)$-ambiguous coloring unless the unique boundary edge is parallel to an edge of $\S$, as otherwise we can use $\S$ to extend the coloring uniquely to a larger half-plane.  However the generating set $\S$ may not be uniquely determined (and by Corollary~\ref{generatingset} it is not unique if the discrepancy is strictly negative) and so a $\vec v$-half plane cannot have an $(\S,\eta)$-ambiguous coloring unless {\em every} $\eta$-generating subset of $R_{n,k}$ has an edge parallel to $\vec v$.  This motivates the following definition:
\begin{definition}
Suppose that  $\eta\colon\ZZ\to\A$, $\ell\subset\R^2$ is an oriented rational line through the origin and $A\in SL_2(\Z)$ maps the (downward oriented) $y$-axis to $\ell$.  For $a,b\in\N$, we say that $\ell$ is {\em one-sided $\eta$-expansive with parameters $(a,b,A)$} if every $(\eta\circ A)$-coloring of $[0,a]\times[-b,b]$ extends uniquely to an $(\eta\circ A)$-coloring of $[0,a]\times[-b,b]\cup\{(-1,0)\}$.  An oriented rational line is {\em one-sided $\eta$-expansive} if there exist $a,b\in\N$ and $A\in SL_2(\Z)$ such that it is one-sided $\eta$-expansive with parameters $(a,b,A)$, or just {\em one-sided expansive} when $\eta$ is clear from the context.

An oriented rational line through the origin which is not one-sided expansive is called {\em one-sided $\eta$-nonexpansive}, or just {\em one-sided nonexpansive} when $\eta$ is clear from the context.

We can also refer to a vector $\vec v\in\ZZ\setminus\{\vec0\}$ as being {\em one-sided expansive} or {\em one-sided nonexpansive}, meaning that the span of $\vec v$ (with the orientation inherited from $\vec v$) is one-sided expansive or one-sided nonexpansive, respectively. 
\end{definition}

Although the $y$-axis seems to play a distinguished role in the definitions of 
one-sided expansiveness and one-sided nonexpansiveness, the choice of this direction is arbitrary (see Remark~\ref{rem:standard}).  The insistence on dealing with {\em directed} rational lines instead of just rational lines is to allow the possibility that a line can be one-sided expansive with one orientation and one-sided nonexpansive with the other.  These definitions are one-sided generalizations of the definitions of expansive and nonexpansive subspaces used in Boyle and Lind~\cite{BL}.

In a similar vein, the parameters $(a,b,A)$ used in the definition of one-sided expansiveness are merely the choice of a convenient coordinate system.  With respect to a different choice of $A$, the line would still be one-sided expansive but with different choice of $a$ and $b$.  

\begin{lemma}\label{ambiguousinorbit}
Suppose $\eta\colon\ZZ\to\A$ and $D_{\eta}(R_{n,k})\leq0$ for some $n, k\in\N$.  If $\S\subseteq R_{n,k}$ is an $\eta$-generating set, then a directed rational line  $\ell$ through the origin is one-sided $\eta$-nonexpansive if and only if there exists an $(\S,\eta)$-ambiguous $\ell$-half plane $P$.  

Moreover, given a one-sided $\eta$-nonexpansive line $\ell$,
there exist $f,g\in X_{\eta}$ such that the restrictions of $f$ and $g$ to the half plane $P$ coincide, but they differ on its $\ell$-extension.
\end{lemma}
\begin{proof}
If there is an $(\S,\eta)$-ambiguous $\ell$-half plane, then $\ell$ must be one-sided nonexpansive; otherwise one-sided $\ell$-expansiveness contradicts ambiguity of the half plane.  Conversely, if $\ell$ is one-sided nonexpansive and $A\in SL_2(\Z)$ maps the (downward oriented) $y$-axis to $\ell$, then for every $a,b\in\N$ there is an $(\eta\circ A)$-coloring of $[0,a]\times[-b,b]$ that has two extensions to an $(\eta\circ A)$-coloring of $[0,a]\times[-b,b]\cup\{(-1,0)\}$.  Let $f_a, g_a\in\overline{\mathcal{O}(\eta\circ A)}$ be two such extensions of $[0,a]\times[-a,a]$.  By compactness, there exist accumulation points $f,g\in\overline{\mathcal{O}(\eta\circ A)}$ for the sequences
$(f_a)_{a\in\N}$ and $(g_a)_{a\in\N}$, respectively.
Then $f\rst{H_0}=g\rst{H_0}$, but $f\rst{H_{-1}}\neq g\rst{H_{-1}}$, where $H_i:=\{(x,y)\in\ZZ\colon x\geq i\}$.  Applying $A^{-1}$ to this half plane gives the result.
\end{proof}

The following simple lemma is used to limit possible directions of periodicity:
\begin{lemma}\label{possibledirections}
Suppose $\eta\colon\ZZ\to\A$ and there exist $n, k\in\N$ such that $P_{\eta}(R_{n, k})\leq nk$.  If $\S$ is an $\eta$-generating set, then for any one-sided $\eta$-nonexpansive direction $\ell$, there is a boundary edge $w_{\ell}\in E(\S)$ parallel to $\ell$.

In particular, $\ell$ can be translated such that it intersects $R_{n,k}$ in at least two places.
\end{lemma}
\begin{proof}
Suppose $\S\subseteq R_{n,k}$ is an $\eta$-generating set but any translation of $\ell$ intersects $\S$ in at most one place.  Choose a translation of $\ell$ which intersects $\S$ at a vertex, and without loss of generality assume this translation of $\ell$ intersects $\S$ at the origin.  Let $A\in SL_2(\Z)$ be a map taking the $y$-axis to $\ell$.  Choose an $(A^{-1}(\S),\eta\circ A)$-ambiguous coloring of $H_0:=\{(x,y)\in\ZZ\colon x\geq0\}$.  Notice that 
$$A^{-1}(\S)\cap\{(0,y)\colon y\in\Z\}=(0,0).$$  
Since $(0,0)\in A^{-1}(\S)$ is $\eta\circ A$-generated, there is a unique extension of any $\eta$-coloring of $H_0$ to an $\eta$-coloring of $\{(x,y)\in\ZZ\colon x\geq-1\}$.  This contradicts the $(A^{-1}(\S),\eta\circ A)$-ambiguity of the coloring of $H_0$.
\end{proof}

Combining this lemma with Lemma~\ref{lemma:no-irrational}, we have: 
\begin{corollary}
\label{cor:no-irrational}
Suppose $\eta\colon\ZZ\to\A$ and there exist $n, k\in\N$ such that $P_{\eta}(R_{n, k})\leq nk$.
If $\ell$ is a nonexpansive line for $X_\eta$, then there exists a translation of $\ell$ 
that intersects $R_{n,k}\cap\ZZ$ in at least two points.  
\end{corollary}

\subsection{A characterization of double periodicity}

\begin{lemma}\label{perimeter}
Suppose $\vec v_1,\dots,\vec v_m\in\ZZ\setminus\{\vec0\}$.  Given $n\in\N$, there exists $A=A({n,\vec v_1,\dots,\vec v_m})\in\N$ such that any finite, convex $\S\subset\ZZ$ containing at least $A$ integer points and such that $\partial\S$ is $(\vec v_1,\dots,\vec v_m)$-enveloped, has a boundary edge that contains at least $n$ integer points.
\end{lemma}

\begin{proof}
For each $i=1,2,\dots,m$ choose a length $L_i\in\R$ such that any rational line parallel to $\vec v_i$ of length at least $L_i$ contains at least $n$ integer points.  Define
$$
A:=\left\lceil\frac{(L_1+\cdots+L_m)^2}{4\pi}\right\rceil+mn.
$$
By Pick's Theorem, the area of $\conv(\S)$ is given by
$$
\text{(\# of integer points inside $\conv(\S)$)}+\frac{\text{(\# of integer points on $\partial\S$)}}{2}-1.
$$
Since $\S$ contains at least $A$ integer points, either the number of integer points on $\partial\S$ is at least $mn$ or the area of $\conv(\S)\geq\frac{1}{4\pi}(L_1+\cdots+L_m)^2$.  In the former case, at least one of the edges of $\partial\S$ contains $n$ integer points.  In the latter case, the isoperimetric inequality implies that the length of $\partial\S$ is at least $L_1+\cdots+L_m$, and so at least one of the edges contains $n$ integer points.
\end{proof}

We strengthen the notion of an enveloping set, further assuming that the boundary
consists of of a finite collection of sufficiently long edges taken in order from the enveloping set: 

\begin{definition}
Suppose $\S\subseteq\ZZ$ is finite, convex set whose boundary edges are enumerated as $w_1,\dots,w_n$ where $w_{i+1}=\suc(w_i)$ and indices are taken modulo $n$. 
Let $\T\subseteq\ZZ$ be a convex superset of $\S$ that is $E(\S)$-enveloped.  
Enumerate the edges of $\T$ as $v_1,\dots,v_k$, where $v_{i+1}=\suc(v_i)$ and indices are taken modulo $k$ only when $v_k$ has a successor edge.  We say that $\T$ is {\em strongly $E(\S)$-enveloped} if there exists $j\in\{1,\dots,n-k\}$ such that for all $j^{\prime}\in\{0,1,\dots,k-1\}$, $v_{j+j^{\prime}}$ is parallel to $w_{j^{\prime}}$ and $\|v_{j+j^{\prime}}\|\geq\|w_{j^{\prime}}\|$.
\end{definition}

\begin{lemma}\label{parameterperimeter}
Let $\eta\colon\ZZ\to\A$ and suppose there exist $n,k\in\N$ such that $P_{\eta}(R_{n,k})\leq nk$.  Let $\S\subseteq R_{n,k}$ be an $\eta$-generating set and let $M\in\N$.

There exists $C=C(\S,M)\in\N$ such that for any strongly $E(\S)$-enveloped set $\T$ that contains at least $C$ integer points, there exists $v\in E(\T)$ such that $\Ext_v(\T)\neq\T$.  Furthermore, $\Ext_v(\T)$ is  strongly $E(\S)$-enveloped and the edge of $\Ext_v(\T)$ parallel to $v$ contains at least $M$ integer points.  

Moreover, for fixed $\S$ and $w\in E(\S)$, there exists $C^{\prime}=C^{\prime}(\S,M,w)\in\N$ such that for any strongly $E(\S)$-enveloped set $\T$ that has nonempty intersection with at least $C^{\prime}$ distinct lines parallel to $w$, there exist $v^{\prime}, v^{\prime\prime}\in E(\T)$ which are neither parallel nor antiparallel to $w$ such that $\Ext_{v^{\prime}}(\T)\neq\T$, $\Ext_{v^{\prime\prime}}(\T)\neq\T$, $\Ext_{v^{\prime\prime}}(\Ext_{v^{\prime}}(\T))$ is also strongly $E(\S)$-enveloped, and the edges of $\Ext_{v^{\prime\prime}}(\Ext_{v^{\prime}}(\T))$ parallel to $v^{\prime}$ and $v^{\prime\prime}$ both contain at least $M$ integer points.  Furthermore, with respect to the local ordering on directed lines induced by the positive orientation on the unit circle, 
 $v^{\prime}$ and $v^{\prime\prime}$ can be chosen such that $v^{\prime}$ is between $w$ and the direction antiparallel to $w$ 
 and $v^{\prime\prime}$ is between the direction antiparallel to $w$ and $w$.
\end{lemma}

\begin{proof}
Recall that the depth of an extension (see Definition~\ref{envelopeddef2}) of a convex set depends only on three slopes: that of the edge over which the extension occurs and of its successor and predecessor edges.   
Consider an edge $w\in E(\S)$.  If the edge of $\T$ that is parallel to $w$ is sufficiently long 
such that $\Ext_w(\T)\neq\T$, then the depth of the $w$-extension is some positive integer (otherwise it is $0$).  Then the length of the edge of $\Ext_w(\T)$ that is parallel to $w$ is 
greater than the length of the edge of $\T$ that is parallel to $w$.  
%For each $w\in E(\S)$, there exist $d(w), e(w)\in\N$ such that for any strongly $E(\S)$-enveloped set $\T$, the depth of the $w$-extension of $\T$ is either zero or $d(w)$ (depending on whether the edge of $\T$ parallel to $w$ is sufficiently long that $Ext_w(\T)\neq\T$), and it is $d(w)$ whenever the edge of $\T$ parallel to $w$ is at least $e(w)$.  When the length of the edge of $\T$ parallel to $w$ is at least $e(w)$, there exists $f(w)\in\R$ such that the length of the edge of $Ext_w(\T)$ parallel to $w$ is exactly the length of the edge of $\T$ parallel to $w$, minus $f(w)$.  
Thus there exists $N(w,M)\in\N$ such that if $\T$ is strongly $E(\S)$-enveloped and the edge of $\T$ parallel to $w$ is at least $N(w,M)$, then $\Ext_w(\T)$ is also strongly $E(\S)$-enveloped and the edge of $\Ext_w(\T)$ parallel to $w$ contains at least $M$ integer points.

The remainder of the proof is similar to that of Lemma~\ref{perimeter}.  For the first claim in the lemma, we find $C\in\N$ such that any strongly $E(\S)$-enveloped set $\T$ that contains at least $C$ integer points has a boundary edge whose length is at least $N(w_1,M)+N(w_2,M)+\cdots+N(w_{|E(\S)|},M)$.  This boundary edge has the desired property.  The proof of the second claim follows similarly.  We fix $w\in E(\S)$ and find $C^{\prime}\in\N$ such that any strongly $E(\S)$-enveloped set $\T$ for which there are at least $C^{\prime}$ distinct lines parallel to $w$ that have nonempty intersection with $\T$ has a boundary edge that is neither parallel nor antiparallel to $w$ and contains at least $N(w_1,M)+\cdots+N(w_{|E(\S)|},M)$ integer points.
\end{proof}

\begin{lemma}\label{extend}
Suppose $\eta\colon\ZZ\to\A$ and $\S\subset\ZZ$ is a finite, convex set whose boundary edges are labeled $w_1,\dots,w_n$ where $w_{i+1}=\suc(w_i)$ for all $i=1,\dots,n$ (indices are taken modulo $n$).  Assume that there exist $a,b,c,d\in\mathbb{Q}$ such that $b\geq d$ and
\begin{itemize}
\item $w_1$ points vertically downward and $-a+b+c-d\geq\left|w_1\cap\S\right|-1$;
\item $w_n$ is parallel to the line $y=ax+b$;
\item $w_2$ is parallel to the line $y=cx+d$;
\item there exists $\vec v\in\ZZ$ such that 
$$
(\S\setminus w_1)+\vec v\subseteq\{(x,y)\in\ZZ\colon cx+d\leq y\leq ax+b\text{, }0\leq x\leq L-1\},
$$
where $L\in\N$ is the number of distinct vertical lines that have nonempty intersection with $\S\setminus w_1$.
\end{itemize}
If the two endpoints of $w_1$ are $\eta$-generated by $\S$, then any $(\S, \eta)$-coloring of the region
$$
\mathcal{R}:=\{(x,y)\in\ZZ\colon cx+d\leq y\leq ax+b\text{, }0\leq x\leq L-1\}
$$
and of any $\left|w_1\cap\S\right|-1$ consecutive integer points of the line segment $\{(-1,y)\in\ZZ\colon d-c\leq y\leq b-a\}$ extends uniquely to an $(\S,\eta)$-coloring of
$$
\{(x,y)\in\ZZ\colon cx+d\leq y\leq ax+b\text{, }-1\leq x\leq L-1\}.
$$
\end{lemma}

It is important to note we make no assumption that the lines $y=ax+b$ or $y=cx+d$ intersect the line $x=-1$ at an integer point.

\begin{proof}
Let $h:=\left|w_1\cap\S\right|-1$.  Suppose we know the $\eta$-coloring of the points $(-1,y-h+2),\dots,(-1,y)$ for some $d-c+h-2\leq y\leq b-a-1$.  Let $\vec v\in\ZZ$ be the translation that takes the topmost element of $w_1$ to the point $(-1,y+1)$.  Since $y+1\leq b-a$, the line through $(-1,y+1)$ parallel to $w_n$ is in the region $\{(\alpha, \beta)\in\R^2\colon\beta\leq a\alpha+b\}$.  Since $\S$ is convex and $w_n$ is parallel to $y=ax+b$, we have that $\S+\vec v$ lies in the region $\mathcal{R}\cup\{(-1,y-h+2),\dots,(-1,y+1)\}$ (see Figure~\ref{figurea}).

\begin{figure}[ht]
     \centering
  \def\svgwidth{\columnwidth}
        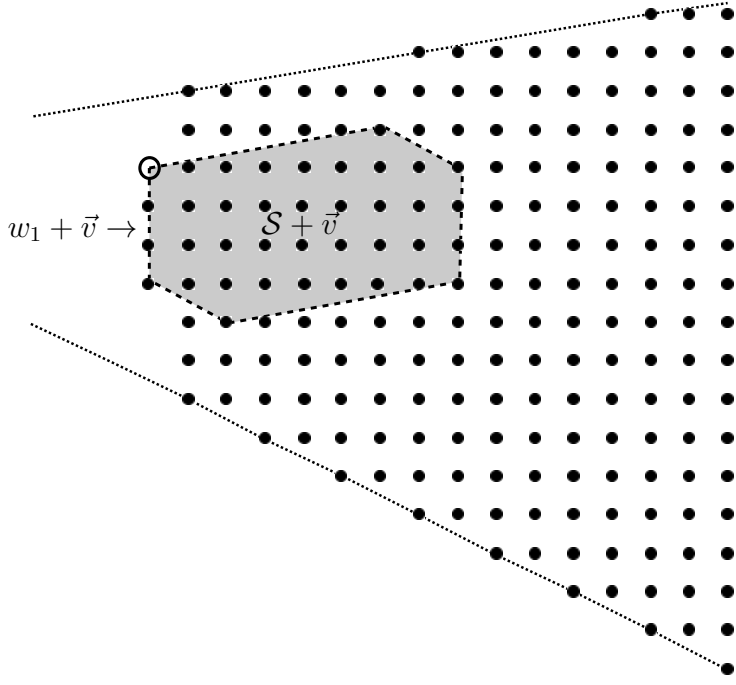
                \setlength{\abovecaptionskip}{-65mm}
	\caption{The dotted points in $\ZZ$ denote the region on which the coloring is known.  The
	color of the topmost element of $w_1+\vec v$ (denoted by the open circle) can be deduced from the coloring of
	the rest of $\S+\vec v$.}
	\label{figurea}
\end{figure}

Since the endpoints of $w_1$ are $\eta$-generated by $\S$, the color at $\eta(-1,y+1)$ can be determined by the restriction of $\eta$ to the rest of the elements of $\S+\vec v$.  Continuing inductively, the $(\S,\eta)$-coloring of the region $\mathcal{R}\cup\{(-1,y-h+2),\dots,(-1,y+1)\}$ extends uniquely to an $(\S,\eta)$-coloring of $\mathcal{R}\cup\{(-1,y-h+2),\dots,(-1,\lfloor b-a\rfloor)\}$.

A similar argument shows we can uniquely extend the $(\S,\eta)$-coloring of the region $\mathcal{R}\cup\{(-1,y-h+2),\dots,(-1,\lfloor b-a\rfloor)\}$ to an $(\S,\eta)$-coloring of the region $\mathcal{R}\cup\{(-1,\lceil d-c\rceil),\dots,(-1,\lfloor b-a\rfloor)\}$.
\end{proof}

\begin{corollary}\label{uniqueextension}
Assume that $\eta\colon\ZZ\to\A$ and $\S\subset\ZZ$ is a finite, convex set whose boundary edges are labeled $w_1,\dots,w_n$ where $w_{i+1}=\suc(w_i)$ for $i=1,\dots,n$ (indices are taken modulo $n$) and suppose the two endpoints of $w_1$ are $\eta$-generated by $\S$.  Let $\T$ be a convex, strongly $E(\S)$-enveloped set that has an edge parallel to $w_1$ which is sufficiently long so that its $w_1$-extension is also strongly $E(\S)$-enveloped.  Then any $(\eta,\S)$-coloring of $\T$ and of any $\left|w_1\cap\S\right|-1$ consecutive integer points in $\Ext_{w_1}(\T)$ on each line parallel to $w_1$ that has non-empty intersection with $\Ext_{w_1}(\T)\setminus\T$ extends uniquely to an $\eta$-coloring of $\Ext_{w_1}(\T)$.
\end{corollary}
\begin{proof}
After a linear change of coordinates mapping $w_1$ to the vertical direction, 
this follows by repeated applications of Lemma~\ref{extend}. 
\end{proof}

This leads us to necessary and sufficient conditions for double periodicity, a result that 
can be derived from Boyle and Lind's Theorem (Theorem~\ref{BoyleLind}).  
We include a complete proof, as we need further information that can be 
derived from the finer notion of one-sided expansiveness, as opposed to expansiveness.  In particular, 
techniques of the proof are also used to understand the case of a unique direction of 
expansivity (Theorem~\ref{singleperiodic}). 

\begin{theorem}\label{doubleperiodicity}
The coloring $\eta\colon\ZZ\to\A$ is doubly-periodic if and only if there exist $n,k\in\N$ such that $P_{\eta}(R_{n,k})\leq nk$ and $\eta$ has no nonexpansive one-dimensional subspaces.
\end{theorem}

\begin{remark}\label{helpfulremark}
We begin by observing that if there are no nonexpansive subspaces for $X_\eta$, then every subspace (and any orientation on it) is one-sided expansive.
\end{remark}

\begin{proof}
Assume that  $\eta$ is doubly periodic and assume that it has vertical period $n$ and horizontal period $k$.  Then for every $a,b\in\N$, $P_{\eta}(R_{a,b})\leq nk$.  In particular, $D_{\eta}(R_{1,nk})\leq0$ and $D_{\eta}(R_{nk,1})\leq0$.  By Corollary~\ref{generatingset},
$R_{1,nk}$ contains an $\eta$-generating set $\S$ and $R_{nk,1}$ contains 
an $\eta$-generating set $\T$.  If $\ell$ is any rational line, then at least one of $\S$ and $\T$ is not parallel to $\ell$.  Since $\S$ and $\T$ are generating, $\ell$ is one-sided expansive.

Conversely, we show:
\begin{claim}
\label{claim:double}
If $D_{\eta}(R_{n,k})\leq0$ and $\eta$ has no one-sided nonexpansive directions, then there exists a finite set $F\subseteq\ZZ$ such that every $\eta$-coloring of $F$ extends uniquely to an $\eta$-coloring of $\ZZ$. 
\end{claim}
Given the claim, it follows that $\eta$ is doubly periodic: since there are only finitely many $\eta$-colorings of $F$, there exist vectors $\vec a, \vec b, \vec c\in\ZZ$ such that $\vec c-\vec a$ and $\vec c-\vec b$ are not collinear and such that $(T^{\vec a}\eta)\rst{F}=(T^{\vec b}\eta)\rst{F}=(T^{\vec c}\eta)\rst{F}$.  Since each coloring of $F$ extends uniquely to an $\eta$-coloring of $\ZZ$, it follows that $T^{\vec a}\eta=T^{\vec b}\eta=T^{\vec c}\eta$.  Thus $\vec c-\vec a$ and $\vec c-\vec b$ are two linearly independent period vectors for $\eta$, proving the statement.  The remainder of the proof is devoted to establishing the claim.

Fix an $\eta$-generating set $\S\subseteq R_{n,k}$ and enumerate the edges of $\partial\S$ as $w_1,\dots,w_m$ where $w_{i+1}=\suc(w_i)$ for all $i$ (indices are taken modulo $m$).  None of the lines determined by $w_1,\dots, w_m$ are one-sided nonexpansive by Remark~\ref{helpfulremark}.  Thus we can choose parameters $a_1,\dots,a_m,b_1,\dots,b_m\in\N$ and $A_1,\dots,A_m\in SL_2(\Z)$ such that the line determined by $w_i$ is one-sided $(\S,\eta)$-expansive with parameters $(a_i,b_i,A_i)$.   By definition of one-sided expansiveness, for fixed $r\in\N$, any $(\eta\circ A_i)$-coloring of the rectangular set $[0,a_i]\times[-b_i,b_i+r]$ extends uniquely to an $(\eta\circ A_i$)-coloring of
$$
([0,a_i]\times[-b_i,b_i+r])\cup\{(-1,0),(-1,1),(-1,2),\dots,(-1,r-1)\}.
$$
Equivalently, any $\eta$-coloring of the set
$$
B_i^r:=A_i([0,a_i]\times[-b_i,b_i+r])
$$
extends uniquely to an $\eta$-coloring of the set
\begin{equation}\label{subextensionestimate}
\widetilde{B_i^r}:=A_i\left(([0,a_i]\times[-b_i,b_i+r])\cup\{(-1,0),(-1,1),\dots,(-1,r-1)\}\right).
\end{equation}

For $i,j$ such that $w_i$ and $w_j$ are neither parallel nor antiparallel, let $\theta_{i,j}\in\left(-\pi,\pi\right)\setminus\{0\}$ be the angle between $w_i$ and $w_j$ and let $c_i$ be the length of the orthogonal projection of $A_i(1,0)$ onto the direction determined by $w_i$.  Let
$$
N=\sum_{i=1}^m(2b_i+a_ic_i)+\max_{i,j}\left\{\frac{a_1+\cdots+a_m+m}{\left|\tan\theta_{i,j}\right|}\right\}+\sum_{i=1}^m\left|w_i\cap\S\right|.
$$

Let $C^{\prime}(\S,N+1,w_i)$ be the parameter appearing in Lemma~\ref{parameterperimeter} and let 
$$
c:=\max_{w\in E(\S)}C^{\prime}(\S,N+1,w).
$$
If $\T\subseteq\ZZ$ is a strongly $E(\S)$-enveloped set satisfying
\begin{enumerate}
\item $\T$ has edges parallel to each of the edge of $\S$;\label{conditionA}
\item for each $w\in E(\T)$ there are at least $c$ distinct lines parallel to $w$ that have nonempty intersection with $\T$,\label{conditionB}
\end{enumerate}
then there exist integers $1\leq i_1<i_2\leq m$ 
 such that the edges of $\T$ parallel to $w_{i_1}$ and $w_{i_2}$ both contain at least $N+1$ integer points, where $w_{i_1}$ is between $w_1$ and the direction antiparallel to $w_1$ and $w_{i_2}$ is between the direction antiparallel to $w_1$ and $w_1$ (with respect to the local ordering by the positive orientation).   It also follows that there exist vectors $\vec u_1, \vec u_2\in\ZZ$ such that $(B_{i_1}^{|w_{i_1}\cap\S|-1}+\vec u)\subseteq\T$, but $\widetilde{B_{i_1}}^{|w_{i_1}\cap\S|-1}$ is not.  The analogous 
 statement holds for $i_2$.  Moreover, by Corollary~\ref{uniqueextension}, the edge of $\Ext_{w_{i_1}}(\T)$ parallel to $w_{i_1}$ contains at least $N+1$ integer points and the edge of $\Ext_{w_{i_2}}(\Ext_{w_{i_1}}(\T))$ parallel to $w_{i_2}$ contains at least $N+1$ integer points.  Recall that there exists an integer $d>0$, the depth of the extension, and lines $l_1,\dots,l_d$ parallel to $w_i$, such that
$$
\Ext_{w_{i_1}}(\T)\setminus\T=\bigsqcup_{j=1}^d(l_j\cap \Ext_{w_{i_1}}(\T)).
$$
Also recall that for $r=1,\dots,d$, the set
$$
\T\cup\bigcup_{j=1}^r(l_j\cap \Ext_{w_{i_1}}(\T))
$$
is the $(w_{i_1},r)$-subextension of $\Ext_{w_{i_1}}(\T)$.  By convexity, 
since the edges of $\T$ and $\Ext_{w_{i_1}}(\T)$ parallel to $w_{i_1}$ both contain at least $N+1$ integer points, 
each of the sets $l_j\cap \Ext_{w_{i_1}}(\T)$ contains at least $N$ integer points.  Although the $(w_{i_1},r)$-subextension of $\T$ may not be strongly $E(\S)$-enveloped (recall that $d$ is the smallest integer for which the extension is enveloped by the same set as $\T$), we still have 
that for each $r=1,\dots,d$ there exists $\vec u_r\in\ZZ$ such that $B_{i_1}^{|w_{i_1}\cap\S|-1}+\vec u_r$ is contained in the $(w_{i_1},r)$-subextension of $\Ext_{w_{i_1}}(\T)$, but $\widetilde{B_{i_1}}^{|w_{i_1}\cap\S|-1}$ is not.  By~\eqref{subextensionestimate} it follows that any $\eta$-coloring of the $(w_{i_1},r)$-subextension of $\T$ uniquely determines the color of at least $|w_{i_1}\cap\S|-1$ consecutive integer points on $l_{r+1}\cap \Ext_{w_{i_1}}(\T)$.  By Corollary~\ref{uniqueextension} it follows that any $\eta$-coloring of the $(w_{i_1},r)$-subextension of $\T$ extends uniquely to an $\eta$-coloring of the $(w_{i_1},r+1)$-subextension.  Inductively it follows that any $\eta$-coloring of $\T$ extends uniquely to an $\eta$-coloring of $\Ext_{w_{i_1}}(\T)$.  Similarly any $\eta$-coloring of $\Ext_{w_{i_1}}(\T)$ extends uniquely to an $\eta$-coloring of $\Ext_{w_{i_2}}(\Ext_{w_{i_1}}(\T))$.

We construct a nested sequence of finite sets 
\begin{equation}\label{finitesetconstruction}
\T_0\subset\T_1\subset\T_2\subset\cdots\subset\ZZ
\end{equation}
such that, for each $i$, any $\eta$-coloring of $\T_i$ extends uniquely to an $\eta$-coloring of $\T_{i+1}$.  We show that these sets have the property that any $\eta$-coloring of $\bigcup_i\T_i$ extends uniquely to an $\eta$-coloring of $\ZZ$.  Hence any $\eta$-coloring of $\T_0$  extends uniquely to an $\eta$-coloring of $\ZZ$.  Taking $F:=\T_0$ establishes Claim~\ref{claim:double}. 

 Let $\T_0$ be a strongly $E(\S)$-enveloped set that satisfies conditions~\eqref{conditionA} and~\eqref{conditionB}.  Let $B\subset\R^2$ be the thinnest strip with edges parallel and antiparallel to $w_1$ that contains $\T_0$.  By the above argument, find $w_{i_1}, w_{i_2}\in E(\T_0)$ that are neither parallel nor antiparallel to $w_1$ and such that any $\eta$-coloring of $\T_0$ extends uniquely to an $\eta$-coloring of $\Ext_{w_{i_2}}(\Ext_{w_{i_1}}(\T_0))$ and such that this extension satisfies conditions~\eqref{conditionA} and~\eqref{conditionB}.  Moreover $w_{i_1}$ and $w_{i_2}$ can be chosen such that $w_{i_1}$ is between $w_1$ and the direction antiparallel to $w_1$ and $w_{i_2}$ is between the direction antiparallel to $w_1$ and $w_1$.  Let $\T_1:=\Ext_w(\T_0)$.  Inductively, suppose we have constructed a sequence of strongly $E(\S)$-enveloped finite sets $\T_0\subset\T_1\subset\cdots\subset\T_i$ which all have edges parallel to each of the edges of $\S$ and are such that for any $0\leq j\leq i$ and any $w\in E(\T_j)$ there are at least $c$ lines parallel to $w$ that have nonempty intersection with $\T_j$.  Furthermore suppose that any $\eta$-coloring of $\T_0$ extends uniquely to an $\eta$-coloring of $\T_i$.  By the previous argument, there exist $w_{i_1}, w_{i_2}\in E(\T_i)$ which are neither parallel nor antiparallel to $w_1$ and such that $\Ext_{w_{i_2}}(\Ext_{w_{i_1}}(\T_i))$ satisfies the same conditions as $\T_i$ and every $\eta$-coloring of $\T_i$ extends uniquely to an $\eta$-coloring of  it.  Define $\T_{i+1}:=\Ext_{w_{i_2}}(\Ext_{w_{i_1}}(\T_i))$.  By induction, any $\eta$-coloring of $\T_0$ extends uniquely to an $\eta$-coloring of $\bigcup_i\T_i$.

Since $\T_{i+1}$ is an extension of $\T_i$ over edges that are neither parallel nor antiparallel to $w_1\in E(\S)$, it follows that $\bigcup_i\T_i$ is an infinite, $E(\S)$-enveloped subset of $B\cap\ZZ$.   Since $\T_{i+1}$ is obtained by first extending $\T_i$ over an edge whose direction is between $w_1$ and the direction antiparallel to $w_1$ and then extending the extension over an edge that is between the direction antiparallel to $w_1$ and $w_1$, $\bigcup_i\T_i=B\cap\ZZ$.  Since $w_1$ does not determine an expansive direction for $\eta$ and $\T_0$ contains the set $A_1([0,a_1]\times[-b_1,b_1])$, there is a unique extension of any $\eta$-coloring of $\bigcup_i\T_i$ to an $\eta$-coloring of $\ZZ$.  In this case,  set $F:=\T_0$, completing the proof of Claim~\ref{claim:double} and the theorem.  
\end{proof}

\subsection{Single periodicity}
We have now developed the tools to prove Theorem~\ref{singleperiodic}.  We recall 
the statement for convenience: 
\begin{theorem*}[Theorem~\ref{singleperiodic}]
Suppose $\eta\in\A^{\ZZ}$ and $X_{\eta}:=\overline{\mathcal{O}(\eta)}$.  If there exist $n,k\in\N$ such that $P_{\eta}(n,k)\leq nk$ and there is a unique nonexpansive $1$-dimensional subspace for the $\ZZ$-action (by translation) on $X_{\eta}$.  Then $\eta$ is periodic, but not doubly periodic, the unique nonexpansive line $L$ is a rational line through the origin, and every period vector for $\eta$ is contained in $L$.
\end{theorem*}

We make the following definition.
\begin{definition}\label{diameterdefinition}
If $\S\subset\ZZ$ and $(a,b)\in\ZZ$, the {\em $(a,b)$-diameter of $\S$} is the number of distinct rational lines parallel to $(a,b)$ that have nonempty intersection with $\S$.  We denote this by $\diam_{(a,b)}(\S)$.
\end{definition}

\begin{proof}[Proof of Theorem~\ref{singleperiodic}]
We adopt the same notation for $\S$, $w_1,\dots,w_n\in E(\S)$, and $a,N\in\N$ used in the proof of Theorem~\ref{doubleperiodicity}.

By Theorem~\ref{doubleperiodicity}, $\eta$ is not doubly periodic and  
by Corollary~\ref{cor:no-irrational}, the unique nonexpansive line $\ell$ is 
a rational line through the origin.  Without loss of generality, we can assume that
$\ell$ is vertical.  

We claim that any $\eta$-coloring of $[1,a]\times[1,a]$ extends uniquely to an $\eta$-coloring of $[1,a]\times\Z$.  Without loss assume that $w_1$ is the nonexpansive direction and let $\hat{w}_1\in E(\S)$ be edge of $\S$ antiparallel to $w_1$.  By Lemma~\ref{parameterperimeter}, as in the proof of Claim~\ref{claim:double}, we can construct a sequence of finite sets
$$
\T_0\subset\T_1\subset\T_2\subset\cdots\subset\ZZ
$$
such that, for each $i$, any $\eta$-coloring of $\T_i$ extends uniquely to an $\eta$-coloring of $\T_{i+1}$ and $\bigcup_i\T_i$ is a strip with edges parallel and antiparallel to $w_1$.  Choosing $\T_0$ sufficiently large to contain $[1,a]\times[1,a]$ gives the result.

By the claim, the restriction of $\eta$ to any vertical strip of width $a$ is vertically periodic of period at most $P_{\eta}(\T)$, where $\T$ is the smallest $\{w_1,\dots,w_n\}$-enveloped set containing $[1,a]\times[1,a]$, thereby
completing the proof.
\end{proof}

\begin{remark}
We contrast this with a recent result of Hochman~\cite{H}, which shows that there are $\ZZ$-systems that have a unique nonexpansive $1$-dimensional subspace but are not periodic.  Theorem~\ref{singleperiodic} only applies to the special case of those $\ZZ$-subshifts that arise as the orbit closure of a function satisfying the hypothesis of Nivat's Conjecture.
\end{remark}

\begin{remark}
With the assumptions of Theorem~\ref{singleperiodic}, if follows that if there is a unique nonexpansive $1$-dimensional subspace of $\R^2$ for the translation action on $X_{\eta}$, then {\em both} orientations of the subspace are one-sided nonexpansive (since $\eta$ is singly periodic with period vectors contained in this subspace, if either orientation were one-sided expansive it would follow that $\eta$ was doubly periodic).  In general this is not the case (e.g. only one of the two orientations on the vertical direction is one-sided nonexpansive in Ledrappier's example~\cite{L}).
\end{remark}

\section{A stronger bound on complexity}\label{nk/2}
 
In light of Theorems~\ref{doubleperiodicity} and~\ref{singleperiodic}, one strategy for proving Nivat's Conjecture is to show that if $\eta\colon \ZZ\to\A$ satisfies $P_{\eta}(R_{n,k})\leq nk$ for some $n,k\in\N$, then $\eta$ does not have two linearly independent one-sided nonexpansive directions.
Under the strengthened hypothesis that $P_{\eta}(R_{n,k})\leq\frac{nk}{2}$, 
this is the content of Theorem~\ref{mainthm}.
The additional control over $\eta$ obtained from the assumption $P_{\eta}(R_{n,k})\leq\frac{nk}{2}$ comes in three guises:
we obtain a special sort of $\eta$-generating set (Lemma~\ref{stronggeneratingset}), 
we prove the existence of sets that contain many points in a given
rational direction (Lemma~\ref{balancedlemma} and Proposition~\ref{prop:extendedambiguousperiod}), and 
we obtain control on the periods in Section~\ref{sec:thin-generating}.

\subsection{Strong generating sets}

\begin{lemma}\label{stronggeneratingset}
Suppose $\eta\colon\ZZ\to\A$ is aperiodic and there exist $n,k\in\N$ such that $P_{\eta}(R_{n,k})\leq\frac{nk}{2}$.  Then there exists an $\eta$-generating set $\S\subset R_{n,k}$ such that 
\begin{enumerate}
\item
\label{cond:1}
$D_{\eta}(\S)\leq-\frac{\left|\S\right|}{2}$;
\item 
\label{cond:2}
For any $w\in E(\S)$, the discrepancy function satisfies
$$
D_{\eta}(\S\setminus w)\geq D_{\eta}(\S)+\left\lceil\frac{\left|w\cap\S\right|}{2}\right\rceil.
$$
\item
\label{cond:3}
If $\T\subset\S$ is convex and nonempty, then
$$
D_{\eta}(\T)>D_{\eta}(\S).
$$
\end{enumerate}
\end{lemma}

We give a name to a set satisfying the conclusion of this lemma:
\begin{definition}\label{stronggenerator}
If $\eta\colon\ZZ\to\A$ is aperiodic and satisfies $P_{\eta}(R_{n,k})\leq\frac{nk}{2}$ for some $n,k\in\N$, then an $\eta$-generating set $\S\subseteq R_{n,k}$ is a {\em strong $\eta$-generating set} if it satisfies conditions~\eqref{cond:1},~\eqref{cond:2}, and~\eqref{cond:3} of Lemma~\ref{stronggeneratingset}.  
\end{definition}

We note that the existence of such an $\eta$-generating is the first use of the stronger hypothesis 
on the complexity.  
\begin{proof}[Proof of Lemma~\ref{stronggeneratingset}]
We construct the set $\S$  by an iterative process.  
By assumption we have $D_{\eta}(R_{n,k})\leq-\frac{\left|R_{n,k}\right|}{2} = -\frac{nk}{2}$.  Let $\S_1\subseteq R_{n,k}$ be a convex set which is minimal (with respect to inclusion) among all convex subsets of $R_{n,k}$ that have discrepancy at most $D_{\eta}(R_{n,k})$.  
Minimality of $\S_1$ implies that $\S_1$ is $\eta$-generating.
By construction, $\S_1$ satisfies $D_{\eta}(\S_1)\leq D_{\eta}(R_{n,k})\leq-\frac{\left|R_{n,k}\right|}{2}\leq-\frac{\left|\S_1\right|}{2}$. 
If for every $w\in E(\S_1)$, the 
discrepancy satisfies $D_{\eta}(\S_1\setminus w)\geq D_{\eta}(\S_1)+\left\lceil\frac{\left|w\cap\S_1\right|}{2}\right\rceil$, 
then the set $\S:=\S_1$ satisfies the conclusions and we are finished. 

Otherwise, suppose that we have inductively constructed a nested sequence of sets
$$
R_{n,k}\supseteq\S_1\supset\cdots\supset\S_m
$$
such that for $i=1,\dots,m$:
\begin{enumerate}
\item
\label{item:1}
$\S_i$ is convex and nonempty;
\item 
\label{item:2}
$\S_i$ is $\eta$-generating;
\item
\label{item:3} 
$\S_i$ satisfies $D_{\eta}(\S_i)\leq-\frac{\left|\S_i\right|}{2}$;
\item 
\label{item:4}
$\S_i$ is not the intersection of a line segment with $\ZZ$;
\item 
\label{item:5}
There exists $w_i\in E(\S_i)$ such that $D_{\eta}(\S_i\setminus w_i)<D_{\eta}(\S_i)+\left\lceil\frac{\left|w_i\cap\S_i\right|}{2}\right\rceil$.
\end{enumerate}
Choose $w_m\in E(\S_m)$ such that $D_{\eta}(\S_m\setminus w_m)<D_{\eta}(\S_m)+\left\lceil\frac{\left|w_m\cap\S_m\right|}{2}\right\rceil$.  Since the left hand side of the inequality is an integer, 
$$
D_{\eta}(\S_m\setminus w_m)<D_{\eta}(\S_m)+\frac{\left|w_m\cap\S_m\right|}{2}.
$$
Let $\S_{m+1}\subset\S_m\setminus w_m$ be a convex set which is minimal (with respect to inclusion) among all convex subsets of $\S_m\setminus w_m$ of discrepancy at most $D_{\eta}(\S_m\setminus w_m)$.  Then $\left|\S_{m+1}\right|\leq\left|\S_m\right|-\left|w_m\cap\S_m\right|$, and so
\begin{align*}
D_{\eta}(\S_{m+1})\leq D_{\eta}(\S_m\setminus w_m) & <  D_{\eta}(\S_m)+\frac{\left|w_m\cap\S_m\right|}{2}\\
& \leq- \frac{\left|\S_m\right|}{2}+\frac{\left|w_m\cap\S_m\right|}{2}\leq-\frac{\left|\S_{m+1}\right|}{2}.
\end{align*}
By minimality, $\S_{m+1}$ is $\eta$-generating, and contains at least two elements (since its $\eta$-discrepancy is negative).  Thus we have satisfied conditions~\eqref{item:1},~\eqref{item:2}, and~\eqref{item:3}.
If $\S_{m+1}$ is the intersection of a line segment with $\ZZ$, then the Morse-Hedlund Theorem implies that the restriction of $\eta$ to any line parallel to $\S_{m+1}$ is periodic with period at most $\left|\S_{m+1}\right|$.  But then $\eta$ is periodic, a contradiction, 
and so condition~\eqref{item:4} is satisfied.  If for every $w\in E(\S_{m+1})$, 
$$D_{\eta}(\S_{m+1}\setminus w)\geq D_{\eta}(\S_{m+1})+\left\lceil\frac{\left|w\cap\S\right|}{2}\right\rceil,$$
 then the set $\S:=\S_{m+1}$ satisfies the conclusions of the lemma.  Otherwise $\S_{m+1}$ satisfies all of the induction hypotheses and the construction continues.  
 In both cases, condition~\eqref{item:5} is satisfied.  

Each $\S_i$ is contained in $R_{n,k}$, so the construction terminates after 
finitely many steps.
\end{proof}

\begin{lemma}\label{nonuniqueextension2}
Suppose $\eta\colon\ZZ\to\A$ is aperiodic and there exist $n,k\in\N$ such that $P_{\eta}(R_{n,k})\leq\frac{nk}{2}$.  Let $\S$ be a strong $\eta$-generating set.  If $w\in E(\S)$, then there are at most $\left\lfloor\frac{\left|w\cap\S\right|}{2}\right\rfloor$ distinct $\eta$-colorings of $\S\setminus w$ that extend non-uniquely to $\eta$-colorings of $\S$.
\end{lemma}
\begin{proof}
The proof is identical to that of Lemma~\ref{nonuniqueextension1} with the stronger bound on $P_{\eta}(\S)-P_{\eta}(\S\setminus w)$ implied by assumption that $\S$ is 
strong generating.
\end{proof}

\begin{lemma}\label{ambiguousperiod2}
Suppose $\eta\colon\ZZ\to\A$, $\S\subset\ZZ$ is a strong $\eta$-generating set and there are antiparallel $w_1, w_2\in E(\S)$.    Suppose $\left|w_1\right|\leq\left|w_2\right|$, $H$ is a $w_1$-half plane, and the restriction of $f\in X_{\S}(\eta)$ to $H$ is $(\S,\eta)$-ambiguous.  Then the $(\S\setminus w_1,w_1)$-border of $H$ is periodic with period vector parallel to $w_1$.  Its period is at most
$\left\lfloor\frac{\left|w_1\cap\S\right|}{2}\right\rfloor$.
\end{lemma}
\begin{proof}
Again, the proof is identical to that of Lemma~\ref{ambiguousperiod} with the stronger bound on $P_{\eta}(\S)-P_{\eta}(\S\setminus w_1)$ implied by the assumption on $\S$ and its use in Lemma~\ref{nonuniqueextension2}.
\end{proof}

\subsection{Balanced sets}
We now give a definition motivated by the technical conditions appearing in Corollary~\ref{ambiguouscorollary}.  Intuitively, an $\ell$-balanced set is useful for the same reason as a generating set (i.e. if a coloring is known on some region, a balanced set often allows us to deduce the coloring on larger regions) except that any line parallel to $\ell$ that has nonempty intersection with it, intersects it in ``many'' of places.  This intersection property comes 
at the expense that not all of the vertices of the set are $\eta$-generated by the set.

\begin{definition}\label{balanced}
Suppose that $\eta\colon\ZZ\to\A$ and $\S\subset\ZZ$ is finite and convex.  Suppose $\ell$ is an oriented rational line and let $\ell(\S)\subseteq E(\S)\cup V(\S)$ be the intersection of $\conv(\S)$ with the support line to $\S$ parallel to $\ell$.  We say that $\S$ is {\em $\ell$-balanced for $\eta$} (or simply {\em $\ell$-balanced}) if all of the following conditions hold:
\begin{enumerate}
\item Every rational line parallel to $\ell$ that has nonempty intersection with $\S$ contains at least $\left|\ell(\S)\cap\S\right|-1$ integer points;
\item The endpoints of $\ell(\S)\cap\S$ are $\eta$-generated by $\S$;
\item $D_{\eta}(\S\setminus \ell(\S))>D_{\eta}(\S)$.
\end{enumerate}
\end{definition}

\begin{definition}
For $\vec v\in\ZZ\setminus\{\vec0\}$, a {\em $\vec v$-strip} is a convex subset of $\ZZ$ whose boundary contains precisely two edges, one of which is parallel to $\vec v$ and the other is antiparallel to $\vec v$ (we also include the degenerate case, calling the intersection of $\ZZ$ with a line parallel to $\vec v$, a $\vec v$-strip).  The {\em $\vec v$-width} of a $\vec v$-strip is the number of distinct lines parallel to $\vec v$ that have nonempty intersection with it (in the degenerate case, the width is $1$).
\end{definition}

Showing the existence of an $\ell$-balanced set for $\eta$ is the second use of the stronger hypothesis on complexity.  It is used in the proof of Theorem~\ref{twoormore} in Section~\ref{constructionofK}.
\begin{lemma}\label{balancedlemma}
If $\eta\colon\ZZ\to\A$ and $P_{\eta}(R_{n,k})\leq\frac{nk}{2}$ for some $n,k\in\N$, then for any rational line $\ell$, there exists an $\ell$-balanced set for $\eta$.
\end{lemma}
\begin{proof}
If $\ell$ is a horizontal line (without loss of generality, assume it points west),   
choose the minimal $k^{\prime}\leq k$ such that $P_{\eta}(R_{n,k^{\prime}})\leq\frac{nk^{\prime}}{2}$.  Let $w\in E(R_{n,k^{\prime}})$ be the edge parallel to $\ell$.  By minimality, $D_{\eta}(R_{n,k^{\prime}}\setminus\ell)>D_{\eta}(R_{n,k^{\prime}})$.  Choose 
a minimal convex $\S$ satisfying
$$
R_{n,k^{\prime}}\setminus\ell\subset\S\subseteq R_{n,k^{\prime}}
$$
for which $D_{\eta}(\S)=D_{\eta}(R_{n,k^{\prime}})$.  By minimality of $\S$, the endpoints of the support line of $\S$ parallel to $\ell$ are generated.  
Therefore $\S$ satisfies the definition of an $\ell$-balanced set.  The case that $\ell$ is vertical is similar.

If $\ell$ is neither vertical nor horizontal, assume that $n\geq k$ (the other case is similar).  Let $\vec v=(v_1, v_2)\in\ZZ$ be the shortest integer vector parallel to $\ell$.  We assume that $v_1, v_2<0$ (the other cases are similar).  If every translation of $\ell$ intersects $R_{n,k}$ in at most one integer point, then any $\eta$-generating set contained in $R_{n,k}$ is automatically $\ell$-balanced (here the first condition in the definition of an $\ell$-balanced set is trivial).  Thus we assume that $(v_1, v_2)$ connects two integer points in $R_{n,k}$.  Choose $\vec u\in\ZZ$ such that $\ell+\vec u$ passes through the
\begin{itemize}
\item northeast corner of $R_{n,k}$  if $v_2/v_1>k/n$;
\item southwest corner of $R_{n,k}$  if $v_2/v_1\leq k/n$.
\end{itemize}
Assume that $v_2/v_1>k/n$ (the other case is similar).  By choice of $\vec u$, $(\ell+\vec u)$ intersects both the top and bottom of the rectangle $R_{n,k}$, so
$$
\|\conv(R_{n,k})\cap(\ell+\vec u)\|=\max_{\vec v\in\ZZ}\|\conv(R_{n,k})\cap(\ell+\vec v)\|.
$$
Moreover, one of the endpoints of the line segment $\conv(R_{n,k})\cap(\ell+\vec u)$ is an integer point and so
\begin{equation}\label{calc1}
\left|R_{n,k}\cap(\ell+\vec u)\right|=\max_{\vec v\in\ZZ}\left|R_{n,k}\cap(\ell+\vec v)\right|.
\end{equation}
There is some $i\in\R$ such that $\ell+\vec u-(i,0)$ passes through the southwest corner of $R_{n,k}$ and, by symmetry, the number of integer points in $R_{n,k}$ to the left of $\ell+\vec u-(i,0)$ is the same as the number of integer points in $R_{n,k}$ to the right of $\ell+\vec u$ (see Figure~\ref{figureb}).

\begin{figure}[ht]
     \centering
  \def\svgwidth{\columnwidth}
        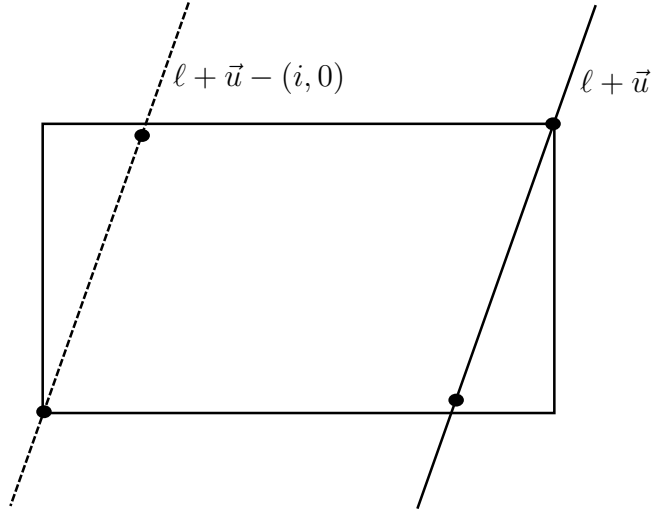
                \setlength{\abovecaptionskip}{-100mm}
	\caption{$R_{n,k}$ with $\ell+\vec u$ (solid line) and $\ell+\vec u-(i,0)$ (dashed line) shown.  The integer points in $R_{n,k}$ are preserved under the rotation by $\pi$ about the center of $R_{n,k}$.  The two points marked on $\ell+\vec u$ are the topmost and bottom most integer points of $(\ell+\vec u)\cap R_{n,k}$.}
	\label{figureb}
\end{figure}

Let $\S_1\subseteq R_{n,k}$ be the (convex) set of all $\vec x\in R_{n,k}$ that are either on $\ell+\vec u$ or to the left of it.  Then $\left|R_{n,k}\setminus\S_1\right|\leq\frac{1}{2}\left|R_{n,k}\right|$ and so by Corollary~\ref{generated3}, $D_{\eta}(\S_1)\leq0$.  Let $a,b\in\ZZ$ be the two extremal elements of $\S_1\cap(\ell+\vec u)$ (the dotted points in Figure~\ref{figureb}).  Let $\S_2\subseteq\S_1$ be minimal (with respect to inclusion) among all convex subsets of $\S_1$ that contain $a$ and $b$ and have $\eta$-discrepancy no larger than $D_{\eta}(\S_1)$.  Then either $\S_2=\S_1\cap(\ell+\vec u)$ or $\S_2$ contains $\S_1\cap(\ell+\vec u)$ and $\conv(\S_2)$ has positive area.  The case that $\conv(\S_2)$ has positive area is illustrated in Figure~\ref{figurec}.

\begin{figure}[ht]
     \centering
  \def\svgwidth{\columnwidth}
        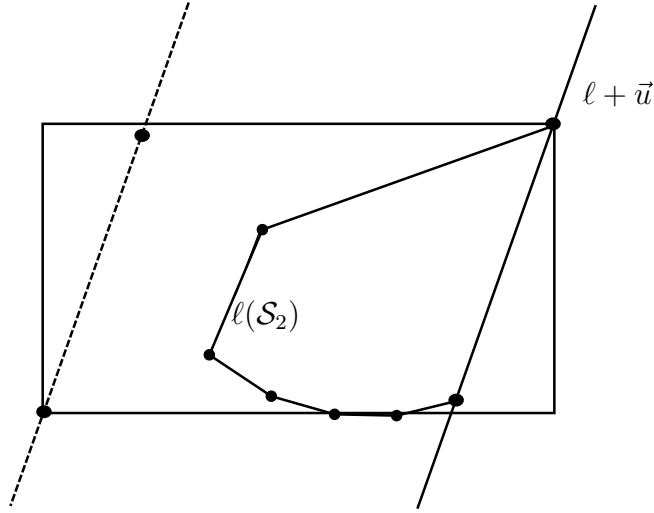
                \setlength{\abovecaptionskip}{-100mm}
	\caption{All extremal points in $\S_2$ except the endpoints of $(\ell+\vec u)\cap R_{n,k}$ are $\eta$-generated, where $\ell$ points southwest and $\S_2$ is $\ell$-balanced.  In this case, $\ell(\S_2)$ is the edge on the left side of $\S_2$, parallel to $\ell$.}
	\label{figurec}
\end{figure}

If the area of $\conv(\S_2)$ is zero, let $\S_3\subseteq\S_2$ be minimal among all convex subsets of $\S_2$ with $\eta$-discrepancy at most $D_{\eta}(\S_2)$.  Then $\S_3$ is an $\eta$-generating set contained entirely in $\ell+\vec u$ and so $\S_3$ is $\ell$-balanced.

In the second case, by minimality of $\S_2$ and Lemma~\ref{generated}, any extremal point of $\S_2$ other than $a$ and $b$ must be $\eta$-generated by $\S_2$.  If $\ell(\S_2)\subseteq V(\S_2)$, then $\ell(\S_2)$ is $\eta$-generated by $\S_2$  and so $\S_2$ is $\ell$-balanced.  Otherwise $\ell(\S_2)\in E(\S_2)$ and both of the extremal elements of $\ell(\S_2)$ are $\eta$-generated by $\S_2$.  Then $E(\S_2)$ has edges parallel and antiparallel to $\ell$ (the edge antiparallel to $\ell$ is the line segment $(\ell+\vec u)\cap\conv(R_{n,k})$).  By~\eqref{calc1}, the number of integer points on the edge parallel to $\ell$ is no larger than the number of integer points on the edge antiparallel to $\ell$.  By Lemma~\ref{heightlemma}, $\S_2$ is $\ell$-balanced.
\end{proof}

We now use balanced sets to show that ambiguity gives rise to periodicity.  %In the following 
%proposition, the assumption of the existence of an $\ell$-balanced set $\S$ is 
%redundant (but convenient for the statement), as by applying Lemma~\ref{balancedlemma}, its existence follows 
%directly from the assumption on complexity:

\begin{proposition}\label{prop:extendedambiguousperiod}
Let $\eta\colon\ZZ\to\A$ and suppose there exist $n,k\in\N$ such that $P_{\eta}(R_{n,k})\leq\frac{nk}{2}$.  Let $\ell$ be a one-sided nonexpansive direction for $\eta$ and 
and let $H$ be an $\ell$-half plane.  Then there exists an $\ell$-balanced set $\S$ such that: 
\begin{enumerate}
\item Any $f\in X_{\S}(\eta)$ whose restriction to $H$ is $(\S,\eta)$-ambiguous is periodic with period vector parallel to $\ell$.
\item If $w\in E(\S)$ is parallel to $\ell$ and $\tilde{\S}=\S\setminus w$, then the restriction of any $(\S,\eta)$-ambiguous $f$ to the $(\tilde{\S},w)$-border of $H$ has period at most $\left|w\cap\S\right|-1$ and the restriction of $f$ to any $\ell$-strip of width $\diam_{w}(\tilde{S})$ has period at most $2\left|w\cap\S\right|-2$.
\end{enumerate}
\end{proposition}

\begin{proof}
By Lemma~\ref{balancedlemma}, there exists an $\ell$-balanced set $\S$ and an $\hat{\ell}$-balanced set $\S_1$, where $\hat{\ell}$ is the direction antiparallel to $\ell$.  Without loss of generality, we can assume (see Remark~\ref{rem:standard}) that $w$ points vertically downward and $H=\{(x,y)\in\ZZ\colon x\geq0\}$.  
Define $\tilde{\S}:=\S\setminus w$ and for all $K\in\Z$,
set
$$
B_K:=\left\{(x,y)\in\ZZ\colon K\leq x<K+\diam_w(\tilde{\S})\right\}.
$$
By translating if necessary, we can assume that $\tilde{\S}\subset B_0$.  Let $h:=\left|w\cap\S\right|-1$.

We claim that $f\rst{B_K}$ is periodic of period at most $2h$ for all $K\in\Z$, which establishes the 
proposition.  We prove this using induction in several steps.  We start by setting up 
the base case of the induction via two cases, depending on $f\rst{B_K}$ 
extending uniquely or not.  
\subsubsection{Assuming $f\rst{B_K}$ does not extend uniquely}
\label{subsec:base1}
If $f\rst{B_K}$ does not extend uniquely to an $(\S,\eta)$-coloring of $B_K\cup B_{K-1}$, 
we claim that $f\rst{B_K}$ is periodic with period at most $h$ and $f\rst{B_{K-1}}$ is periodic of period at most $2h$.

To prove this, suppose $K\in\Z$ and the coloring of $B_K$ given by $f\rst{B_K}$ does not extend uniquely to an $(\S,\eta)$-coloring of $B_K\cup B_{K-1}$.  By 
Corollary~\ref{ambiguouscorollary}, $f\rst{B_K}$ is vertically periodic of period at most $h$.
For the set $\S$, write $\S(i,j)$ for the translation $\S+(i,j)$, and we use the analogous 
notation for $\tilde{\S}$.  
The two endpoints of $w$ are $\eta$-generated by $\S$ (since $\S$ is $\ell$-balanced), so the coloring of $\tilde{\S}$ given by $f\rst{\tilde{\S}(K,i)}$ does not extend uniquely to an $(\S,\eta)$-coloring of $\S$ for any $i\in\Z$ (otherwise we could use this information to deduce the coloring of $B_{K-1}$).  This means that
\begin{eqnarray*}
\#\left\{f\rst{\tilde{\S}(K,j)}\colon j\in\Z\right\}&\leq&\#\left\{\begin{tabular}{l}$\eta$-colorings of $\tilde{\S}$ that do not extend \\ uniquely to $\eta$-colorings of $\S$\end{tabular}\right\} \\
&\leq&P_{\eta}(\S)-P_{\eta}(\tilde{\S}) \\
&\leq&h\text{ (by Definition~\ref{balanced}(iii) and $\S$ being $\ell$-balanced)}
\end{eqnarray*}
Furthermore, we claim the number of $\eta$-colorings of $\S$ whose restriction to $\tilde{\S}$ is $(\tilde{\S},\S,\eta)$-ambiguous is at most $2h$.  The number of such colorings is $P_{\eta}(\S)$ minus the number of $\eta$-colorings of $\S$ whose restriction is not $(\tilde{\S},\S,\eta)$-ambiguous.  This is the same as $P_{\eta}(\S)$ minus the number of $\eta$-colorings of $\tilde{\S}$ that are not $(\tilde{\S},\S,\eta)$-ambiguous.  The number of $\eta$-colorings of $\tilde{\S}$ that are not $(\tilde{\S},\S,\eta)$-ambiguous is $P_{\eta}(\tilde{\S})$ minus the number of of $\eta$-colorings of $\tilde{\S}$ that are $(\tilde{\S},\S,\eta)$-ambiguous.  The number of $\eta$-colorings of $\tilde{\S}$ that are $(\tilde{\S},\S,\eta)$-ambiguous is at most $h$ (from above).  Putting everything together we get
$$
\#\left\{\begin{tabular}{l}$\eta$-colorings of $\S$ whose restriction \\ to $\tilde{\S}$ is $(\tilde{\S},\S,\eta)$-ambiguous\end{tabular}\right\}\leq P_{\eta}(\S)-\left(P_{\eta}(\tilde{\S})-h\right)\leq2h.
$$
%\begin{eqnarray*}
%\#\left\{\begin{tabular}{l}$\eta$-colorings of $\S$ whose restriction \\ to $\tilde{\S}$ is $(\tilde{\S},\S,\eta)$-ambiguous\end{tabular}\right\} \\
%&&\hspace{-1.2 in}=P_{\eta}(\S)-\#\left\{\begin{tabular}{l}$\eta$-colorings of $\S$ whose restriction \\ to $\tilde{\S}$ is not $(\tilde{\S},\S,\eta)$-ambiguous\end{tabular}\right\} \\
%&&\hspace{-1.2 in}=P_{\eta}(\S)-\#\left\{\begin{tabular}{l}$\eta$-colorings of $\tilde{\S}$ that are \\ not $(\tilde{\S},\S,\eta)$-ambiguous\end{tabular}\right\} \\
%&&\hspace{-1.2 in}=P_{\eta}(\S)-\left(P_{\eta}(\tilde{\S})-\#\left\{\begin{tabular}{l}$(\tilde{\S},\S,\eta)$-ambiguous \\ colorings of $\tilde{\S}$\end{tabular}\right\}\right) \\
%&&\hspace{-1.2 in}\leq P_{\eta}(\S)-P_{\eta}(\tilde{\S})+h \\
%&&\hspace{-1.2 in}\leq2h.
%\end{eqnarray*}

By the Pigeonhole Principle, there exist $0\leq i<j\leq2h$ such that the $\eta$-colorings of $\S$ given by $f\rst{\S(K,i)}$ and $f\rst{\S(K,j)}$ coincide.  Recall that, since $\S$ is $\ell$-balanced, any vertical line $\tilde{\ell}$ that has nonempty intersection with $\S$ satisfies $\left|\tilde{\ell}\cap\S\right|\geq h$ (Definition~\ref{balanced}).  Moreover 
$$
B_K=\ZZ\cap\bigcup_{x=K}^{K+\diam_w(\tilde{\S})}\{(x,y)\colon y\in\R\}
$$
is the intersection of $\ZZ$ with the union of $\diam_w(\tilde{\S})$ many vertical lines, so in particular each of them intersects both $\tilde{\S}(K,i)$ and $\tilde{\S}(K,j)$ in at least $h$ places.  But the minimal vertical period of $f\rst{B_K}$ is at most $h$, so the only way that $f\rst{\tilde{\S}(K,j)}$ could coincide with $f\rst{\tilde{\S}(K,i)}$ is if $j-i$ is a period for $f\rst{B_K}$.  Since the endpoints of $w$ are $\eta$-generated by $\S$, an easy induction argument shows that $f\rst{\S(K,i+k)}$ and $f\rst{\S(K,j+k)}$ coincide for all $k$.  Thus $j-i\leq 2h$ is a period for $f\rst{B_K\cup B_{K-1}}$ and the claim is proven. 

\subsubsection{Assuming $f\rst{B_K}$ extends uniquely}
\label{subsec:base2}
If $f\rst{B_K}$ is periodic and extends uniquely to an $(\S,\eta)$-coloring of $B_K\cup B_{K-1}$, we claim that $f\rst{B_{K-1}}$ is periodic with period dividing that of $f\rst{B_K}$.

To see this, suppose $K\in\Z$.  We already know that
\begin{enumerate}
\item $f\rst{B_K}$ is vertically periodic with period $p$;
\item The coloring of $B_K$ given by $f\rst{B_K}$ extends uniquely to an $\eta$-coloring of $B_K\cup\{(K-1,y)\colon y\in\Z\}$.
\end{enumerate}
We claim that $f\rst{B_{K-1}}$ is vertically periodic of period dividing $p$.  If not, 
define $g\colon \ZZ\to\A$ by
$$
g(x,y)=\left\{\begin{tabular}{ll}
$f(x,y)$ & if $x\geq K$; \\
$f(x,y+p)$ & if $x<K$.
\end{tabular}\right.
$$
Since $f\in X_{\S}(\eta)$, $f\rst{B_K}$ is periodic, and $\diam_w(B_K)=\diam_w(\tilde{\S})$, we are guaranteed that $g\in X_{\S}(\eta)$.  Since the restriction of $f$ to $\{(K-1,y)\colon y\in\Z\}$ is not periodic of period dividing $p$, $f\rst{B_K}=g\rst{B_K}$ but their restrictions to $B_K\cup\{(K-1,y)\colon y\in\Z\}$ do not agree, contradicting the fact that $f\rst{B_K}$ had only one extension to an $(\S,\eta)$-coloring of $B_{K-1}$.  This completes the proof of the 
claim.  

\subsubsection{Periodicity of $f\rst{B_K}$ for $K<0$}
\label{subsec:vertical}
We now start the main induction, carried out in three steps.  
Starting with the ambiguity of $f\rst H$, we
show that the proposition holds for the restriction of $f$ to $\ZZ\setminus H$.
We claim that $f\rst{B_K}$ is vertically periodic of period at most $2h$ for all $K\leq0$.  
Using induction to prove this claim, by Lemma~\ref{ambiguousperiod}, $f\rst{B_0}$ is vertically periodic of period at most $h$.  Suppose that $K<0$ and for $i=0,-1,\dots,K$, $f\rst{B_i}$ is vertically periodic of period at most $2h$.  One of the hypotheses of the two claims 
(in~\ref{subsec:base1} or~\ref{subsec:base2}) applies to $f\rst{B_K}$, and so $f\rst{B_{K-1}}$ is periodic of period at most $2h$.  By induction, this holds for all $K<0$.

\subsubsection{Periodicity of $f\rst{B_K}$ for $K>0$}
\label{subsec:vertical2}
To extend the result for $K>0$, 
recall that $\S_1$ is a set balanced in the direction antiparallel to $\ell$.  Suppose 
that $\hat{w}\in E(\S_1)$ is antiparallel to $\ell$ and let $\tilde{\S}_1:=\S_1\setminus\{\hat{w}\cap\S_1\}$.  Define
$$
\hat{B}_K:=\left\{(x,y)\in\ZZ\colon K-\diam_{\hat{w}}(\S_1)+1\leq x\leq K\right\}
$$
and assume, without loss of generality, that $\hat{w}\subset\{(0,y)\colon y\in\Z\}$.  Then, by the result of~\ref{subsec:vertical}, $f\rst{\ZZ\setminus H}$ is periodic and 
so $f\rst{\hat{B}_0}$ is vertically periodic.  By an induction argument analogous to that given in Sections~\ref{subsec:base1} and~\ref{subsec:base2} (except now using $\S_1$ in place of $\S$), $f\rst{\hat{B}_K}$ is periodic for all $K>0$, and its period is at most the maximum of $(2\left|w\cap\S\right|-2)!$ (an upper bound for the period of $f\rst{\hat{B}_0}$) and $2\left|\hat{w}\cap\S_1\right|-2$.  This implies that there is some constant $C>0$ such that for all $K\in\Z$, $f\rst{B_K}$ is vertically periodic of period at most $C$.  This establishes the first conclusion of the proposition.

\subsubsection{Bounds on the period} 
To establish the second part of the proposition, we need an improvement on the bound of the vertical period of $f\rst{B_K}$ when $K>0$.  For any $K_0\in\Z$ such that $f\rst{B_{K_0}}$ is vertically periodic of period at most $2h$, the induction argument from~\ref{subsec:vertical}, but with the base case changed from $B_0$ to $B_{K_0}$), shows that $f\rst{B_K}$ is vertically periodic of period at most $2h$, for all $K\leq K_0$.  Therefore, it suffices to find a sequence $0<i_1<i_2<\cdots$ such that $f\rst{B_{i_j}}$ is vertically periodic with period at most $2h$ for all $j\in\N$.
 
 Assume instead that no such sequence exists.  Then there exists $I\in\N$ such that for all $i>I$, the coloring of $B_0$ given by $(T^{(i,0)}f)\rst{B_0}$ is either vertically aperiodic or periodic of period larger than $2h$.  Since $f\rst{B_0}$ is periodic of period at most $2h$, we have $I\geq0$ and $(T^{(i,0)}f)\rst{B_0}$ is vertically periodic of period at most $2h$.  We can further assume that $I$ is minimal with this property.  
For $i>I$, the fact that $f\rst{B_i}$ does not satisfy the conclusion of Corollary~\ref{ambiguouscorollary} (specifically the bounds on its period) implies that $(T^{(i,0)}f)\rst{B_0}$ extends uniquely to an $(\S,\eta)$-coloring of $B_0\cup B_{-1}$.  Since $f$ is vertically periodic, there are only finitely many colorings of $B_0$ that occur as $(T^{(K,0)}f)\rst{B_0}$, for $K\in\Z$.  Thus there exists a smallest integer $J$ such that $J\geq I$ and 
there is $j>J$ satisfying $(T^{(J,0)}f)\rst{B_0}=(T^{(j,0)}f)\rst{B_0}$.
Since $(T^{(J,0)}f)\rst{B_0}$ extends uniquely to an $\eta$-coloring of $B_0\cup B_{-1}$, and since the functions $(T^{(J,0)}f)\rst{B_0\cup B_{-1}}$ and $(T^{(j,0)}f)\rst{B_0\cup B_{-1}}$ are two such colorings, they must coincide.  Then $(T^{(J-1,0)}f)\rst{B_0}$ coincides with $(T^{(j-1,0)}f)\rst{B_0}$ and so $J=I$ (by minimality of $J$).  Then $f\rst{B_I}=f\rst{B_J}=f\rst{B_j}$ is periodic of period at most $2h$.  But $f\rst{B_j}$ is either aperiodic or periodic of  period greater than $2h$, contradicting the definition of $I$.
\end{proof}

\begin{corollary}\label{periodicstripextension}
Suppose $\eta\colon\ZZ\to\A$ and there exist $n,k\in\N$ such that $P_{\eta}(R_{n,k})\leq\frac{nk}{2}$.  Suppose $\ell$ is a rational line, $\S$ is an $\ell$-balanced set, $w\in E(\S)$ is parallel to $\ell$, and $B$ is an $\ell$-strip of width $\diam_w(\S)-1$.  If $f\in X_{\S}(\eta)$ and $f\rst{B}$ is periodic (with period vector parallel to $\ell$), then $f$ is periodic with period vector parallel to $\ell$.
\end{corollary}
\begin{proof}
We proceed as in the proof of the first part of Proposition~\ref{prop:extendedambiguousperiod}.
The assumption of the corollary replaces the base case (\ref{subsec:base1} and~\ref{subsec:base2}) and the induction steps of~\ref{subsec:vertical} and~\ref{subsec:vertical2} are identical.
\end{proof}

\begin{lemma}\label{nonexpansivepairs}
Suppose $\eta\colon \ZZ\to\A$ and there exist $n,k\in\N$ such 
that $P_{\eta}(R_{n,k})\leq\frac{nk}{2}$.  If the oriented rational line $\ell$ is a one-sided nonexpansive
direction for $\eta$, then the direction antiparallel to $\ell$ is also one-sided nonexpansive.  In particular, any $\eta$-generating set has boundary edges parallel and antiparallel to $\ell$.
\end{lemma}
\begin{proof}
Let $\S$ be an $\eta$-generating set, $w\in E(\S)$ be parallel to $\ell$, and 
without loss of generality, we can assume that $\ell$ points vertically downward.
Let $\hat{\ell}$ be the direction antiparallel to $\ell$  
and suppose for contradiction that $\hat{\ell}$ is not a one-sided nonexpansive direction for $\eta$.   

Set $H:=\{(x,y)\in\ZZ\colon x\geq0\}$ (a half-plane whose boundary is parallel to $\ell$).  Since $\ell$ is a one-sided nonexpansive direction, we can choose $f_1, f_2\in X_{\S}(\eta)$ such that $f_1\rst{H}=f_2\rst{H}$ but $f_1\neq f_2$.  
By Proposition~\ref{prop:extendedambiguousperiod},  $f_1$ and $f_2$ are both vertically periodic.  Since $f_1\rst{H}=f_2\rst{H}$, at most one of $f_1$ and $f_2$ is doubly periodic.  Without loss of generality, assume that $f_1$ is not doubly periodic.

Since $\hat{\ell}$ is not one-sided nonexpansive, by definition there exist $a_1, a_2\in\N$ and $A\in SL_2(\Z)$ such that $\hat{\ell}$ is one-sided $\eta$-expansive with parameters $(a_1,a_2,A)$.  Then every $\eta$-coloring of a vertical strip of width at least $\|a_2\cdot A^{-1}(1,0)\|$ extends uniquely to an $\eta$-coloring of its $\hat{\ell}$-extension.  In particular, the restriction of $f_1$ to any vertical strip of this width extends uniquely to its $\hat{\ell}$-extension (to the ``right'' in the coordinate system we have chosen).  The vertical periodicity of $f_1$ implies that there are only finitely many distinct patterns that arise from restricting $f_1$ to different vertical strips of this width, each of which extends uniquely to its $\hat{\ell}$-extension.  In this way, the restriction of $f_1$ to one such vertical strip uniquely determines the restriction of $f_1$ to the vertical strip to its right.  Since this holds for all such vertical strips (even those that do not occur in $H$), it follows that $f_1$ is also horizontally periodic, a contradiction.
\end{proof}

\begin{proposition}\label{balancedstronggeneratingset}
Suppose $\eta\colon\ZZ\to\A$ is aperiodic and there exist $n,k\in\N$ such that $P_{\eta}(R_{n,k})\leq\frac{nk}{2}$.  There exists a strong $\eta$-generating set $\S$ such that if $w\in E(\S)$ points in a one-sided nonexpansive direction, then $\left\lfloor\frac{|w\cap\S|}{2}\right\rfloor\leq|w_1\cap\S|-2$.
\end{proposition}
\begin{proof}
By Lemma~\ref{stronggeneratingset} there exists a strong $\eta$-generating set.  Let $\S$ be minimal (with respect to inclusion) among strong $\eta$-generating subsets of $R_{n,k}$.
If for all $w\in E(\S)$ that point in a one-sided nonexpansive direction, we have $|w\cap\S|>2$, then $\S$ satisfies the conclusion of the proposition and we are done.  Otherwise there exists $w\in E(\S)$ which points a one-sided nonexpansive direction and is such that $\left|w\cap\S\right|=2$.  Suppose $w\cap\S=\{(x_1,y_1), (x_2, y_2)\}$.  Then by Lemma~\ref{generated2},
since $(x_1,y_1)$ is $\eta$-generated by $\S$, we have that 
$$
D_{\eta}(\S\setminus\{(x_1,y_1)\})=D_{\eta}(\S)+1.
$$
The convex hull of $\S\setminus(x_1,y_1)$ does not have an edge parallel to $w$ and so by one-sided nonexpansiveness, the vertex $(x_2,y_2)$ is not $\eta$-generated by $\S\setminus\{(x_1,y_1)\}$.  Thus $D_{\eta}(\S\setminus w)\leq D_{\eta}(\S)+1$.  On the other hand, $\left|\S\setminus w\right|=\left|\S\right|-2$, and so $D_{\eta}(\S\setminus w)\leq D_{\eta}(\S)+1\leq-\frac{|\S|}{2}+1=-\frac{\left|\S\setminus w\right|}{2}$.  But then, as in the proof of Lemma~\ref{stronggeneratingset}, $\S\setminus w$ contains a strong $\eta$-generating subset.  This contradicts the minimality of $\S$.
\end{proof}

\subsection{Constructions with balanced sets}

We make precise what it means for a coloring to be periodic on a region:
\begin{definition}\label{convexperiod}
Suppose that $\T\subset\ZZ$ is a convex set and there exists $\vec v\in\ZZ\setminus\{\vec0\}$ such that $(\T+\vec v)\subseteq\T$.  If $f\colon \T\to\A$ is an $\eta$-coloring of $\T$, then $f$ is {\em periodic on $\T$} of period $\vec p\in\ZZ\setminus\{\vec0\}$ if $(\T+\vec p)\subseteq\T$ and $f(\vec x)=f(\vec x+\vec p)$ for all $\vec x\in\T$.

If $w\in E(\T)$, $\vec u\in\ZZ\setminus\{\vec0\}$ is the shortest integer vector parallel to $w$, and $(\T+\vec u)\subseteq\T$, then $f\rst{\T}$ is {\em $w$-eventually periodic with period $p\in\N$ and gap $g\in\N$} if $f\rst{\T+g\vec u}$ is periodic with period $p\vec u$.
\end{definition}

\begin{definition}\label{semiinfinitedefinition}
If $\S\subset\ZZ$ is a finite convex set and $w\in E(\S)$, then a {\em semi-infinite $(\S, w)$-strip} is a set of the form
$$
\vec u+\left\{\S+\lambda\vec v\colon\lambda\in\N\cup\{0\}\right\}\hspace{0.2 in}\text{ or }\hspace{0.2 in}\vec u+\left\{\S-\lambda\vec v\colon\lambda\in\N\cup\{0\}\right\},
$$
where $\vec u\in\ZZ$ and $\vec v$ is the shortest integer vector parallel to $w$.
\end{definition}

Pictorially, a semi-infinite $(\S,w)$-strip is a half-strip whose boundary edges are parallel to edges of $\S$ (and not the other natural interpretation in which the boundary has two semi-infinite edges and one more edge connecting them).

\begin{proposition}\label{semiinfiniteperiodicextension}
Suppose $\eta\colon \ZZ\to\A$ and there exist $n,k\in\N$ such that $P_{\eta}(R_{n,k})\leq\frac{nk}{2}$.  Suppose $\ell$ is a rational line, $\S$ is an $\ell$-balanced set, and $w\in E(\S)$ is parallel to $\ell$.  If $\T$ is a semi-infinite $(\S\setminus w,w)$-strip and $f\in X_{\S}(\eta)$ is such that $f\rst{\T}$ does not extend uniquely to an $\eta$-coloring of the $w$-extension of $\T$, then $f\rst{\T}$ is $w$-eventually periodic with gap at most $\left|w\cap\S\right|-1$ and period at most $\left|w\cap\S\right|-1$.

Moreover, if $\tilde{f}\in X_{\S}(\eta)$ and $\tilde{f}\rst{\T}$ is eventually periodic of period at most $2\left|w\cap\S\right|-2$, then any extension of $\tilde{f}\rst{\T}$ to an $\eta$-coloring of the $w$-extension of $\T$ is also $w$-eventually periodic with the same gap and period at most $2\left|w\cap\S\right|-2$.
\end{proposition}

\begin{proof}
The first statement follows immediately from Corollary~\ref{ambiguouscorollary}.

For the second, let $\tilde{\S}:=\S\setminus w$.  Without loss of generality, 
we can assume that $w$ points vertically downward, the topmost element of $w$ is $(0,0)$, 
and $\tilde{f}\rst{\T}$ is $(0,-1)$-eventually periodic with period $p$ and gap $g$.
Suppose further that the boundary edge of $\T$ parallel to $w$ is $\{(0,y)\in\ZZ\colon y\leq0\}$.  (The case that the boundary edge of $\T$ parallel to $w$ is unbounded from 
above, rather than below, is analogous.)  

Then $\tilde{f}\rst{\T-(0,g)}$ is periodic with period $(0,-p)$.  
Let $B$ be the $(\tilde{\S},w)$-border of $\T-(0,g)$.  There
exist $a\in\Q$ and $b\in\N$ such that the $w$-extension of $\T$ is given by
$$
\T\cup\{(x,y)\in\ZZ\colon -b\leq x\text{, }y\leq ax\}.
$$
We proceed by induction.  Let
$$
B_i:=\{(x,y)\in\ZZ\colon i\leq x<i+\diam_w(\tilde{\S})\text{, }y\leq ax\}.
$$
By assumption $\tilde{f}\rst{B_0-(0,g)}$ is vertically periodic with period at most $2\left|w\cap\S\right|-2$.  We claim that for all $i<0$, if $\tilde{f}\rst{B_i-(0,g)}$ is vertically periodic of period at most $2\left|w\cap\S\right|-2$, then $\tilde{f}\rst{B_{i-1}-(0,g)}$ is also.
To prove the claim, we consider two cases.
\subsubsection{Unique extensions}
First we show that if $\tilde{f}\rst{B_i-(0,g)}$ is periodic of period $p\leq2\left|w\cap\S\right|-2$ and there exists $j<-g$ such that the $\eta$-coloring of $\tilde{\S}$ given by $(T^{(i,j)}\tilde{f})\rst{\tilde{\S}}$ extends uniquely to an $\eta$-coloring of $\S$, then $\tilde{f}\rst{B_{i-1}}$ is periodic of period dividing $p$.

To see this, recall that the vertical period of $f\rst{B_i-(0,g)}$ is $p$.  Then $\tilde{f}\rst{\S-(i,j+p)}=\tilde{f}\rst{\S-(i,j)}$ and $\tilde{f}\rst{\tilde{\S}-(i,j+l+p)}=\tilde{f}\rst{\tilde{\S}-(i,j+l)}$ for all $l\in\N$ such that $l\leq ai-j$.  Since $\S$ is $\ell$-balanced, the top most and bottom most elements of $w$ are $\eta$-generated by $\S$.  Thus 
$\tilde{f}\rst{\S-(i,j+l)}=\tilde{f}\rst{\S-(i,j+l)}$ for all such $l$.  In particular, $\tilde{f}\rst{(B_i\cup B_{i-1})-(i,g)}$ is vertically periodic with period dividing $p$.

\subsubsection{No unique extensions}
Next  we show that if there is no $j<-g$ for which the $\eta$-coloring of $\tilde{\S}$ given by $(T^{(i,j)}\tilde{f})\rst{\tilde{\S}}$ extends uniquely to an $\eta$-coloring of $\S$, then $\tilde{f}\rst{B_i-(0,g)}$ is periodic of period at most $\left|w\cap\S\right|-1$ and $\tilde{f}\rst{B_{i-1}-(0,g)}$ is periodic of period at most $2\left|w\cap\S\right|-2$.

To prove this, if $\tilde{f}\rst{B_i-(0,g)}$ does not extend uniquely to an $\eta$-coloring of the $w$-extension of $B_{i-1}-(0,g)$, then by Corollary~\ref{ambiguouscorollary}, $\tilde{f}\rst{B_i-(0,g)}$ is eventually periodic, and by our assumptions, we have that it is periodic with period at most $\left|w\cap\S\right|-1$.  As in the proof of Proposition~\ref{prop:extendedambiguousperiod}, there are at most $2\left|w\cap\S\right|-2$ distinct $\eta$-colorings of $\S$ occurring as $\tilde{f}\rst{\S-(i,y)}$ for $y\geq g$.  By the Pigeonhole Principle, there exist $j,k\in\N$ with $g\leq j<k<g+2\left|w\cap\S\right|-2$ such that $\tilde{f}\rst{\S-(i,j)}=\tilde{f}\rst{\S-(i,k)}$.  Since $\S$ is $\ell$-balanced, every vertical line that has nonempty intersection with $\S$ intersects in at least $\left|w\cap\S\right|-1$ places.  Since $\tilde{f}\rst{B_i-(0,g)}$ is periodic of period at most $\left|w\cap\S\right|-1$, then $\tilde{f}\rst{\tilde{\S}-(i,y)}=\tilde{f}\rst{\tilde{\S}-(i,y+j-k)}$ for every $y\geq g$.  Arguing as in the previous case, $\tilde{f}\rst{B_i\cup B_{i-1}-(0,g)}$ is periodic of period at most $k-j\leq2\left|w\cap\S\right|-2$.

This proves the claim, and the result follows by induction.
\end{proof}

\begin{corollary}\label{semiinfiniteperiodicextension2}
Under the conditions of Proposition~\ref{semiinfiniteperiodicextension}, if $\S$ is an $\ell$-balanced strong $\eta$-generating set, then $f\rst{\T}$ is $w$-eventually periodic with gap at most $\left|w\cap\S\right|-1$ and period at most $\left\lfloor\frac{\left|w\cap\S\right|}{2}\right\rfloor$.
\end{corollary}
\begin{proof}
This is identical to the proof of Proposition~\ref{semiinfiniteperiodicextension}, with the stronger bound on $P_{\eta}(\S)-P_{\eta}(\S\setminus w)$ implied by the fact that $\S$ is a strong $\eta$-generating set.
\end{proof}

\begin{corollary}\label{ambiguoussuperset}
Suppose $\eta\colon \ZZ\to\A$ and there exist $n,k\in\N$ such that $P_{\eta}(R_{n,k})\leq\frac{nk}{2}$.  Let $\ell$ be a  rational line, $\S$ an $\eta$-generating set, and $\vec u\in\ZZ$ be the shortest integer vector parallel to $\ell$.  Fix a finite set $F\subset\ZZ$ and an $\ell$-strip $B$ of width at least $\diam_{\vec u}(\S)-1$ that contains $F$.  If there is some $f\in X_{\eta}$ such that for all $\lambda\in\Z$ the coloring $(T^{\lambda\vec u}f)\rst{F}$ extends uniquely to an $\eta$-coloring of $B$, then $\eta$ is periodic.
\end{corollary}
\begin{proof}
The condition on $f$ guarantees that $f\rst{B}$ is periodic with period vector parallel to $\vec u$.  Since $f\in X_{\eta}$, there is a translation $\vec v\in\ZZ$ such that $(T^{\vec v}\eta)\rst{F}=f\rst{F}$ and by uniqueness $(T^{\vec v}\eta)\rst{B}=f\rst{B}$.  By Corollary~\ref{periodicstripextension}, $T^{\vec v}\eta$ is periodic with period vector parallel to $\vec u$.  Therefore $\eta$ is also.
\end{proof}

\section{Proof of Theorem~\ref{twoormore}}\label{maintheorem}

\subsection{Starting the proof of Theorem~\ref{twoormore}}
The proof of Theorem~\ref{twoormore} is completed in this section via multiple steps; we include a short summary of what is covered at the beginning of  each section.  We proceed by contradiction, and the rough overall structure of the proof is as follows: assuming the existence of a counterexample, we produce other counterexamples with more structure (specifically with large regions on which they are periodic).  With a sufficiently well structured counterexample, we fix a generating set and count colorings of it that occur on the boundary of the region of periodicity.  We reach a contradiction by showing that a larger than possible number of distinct colorings occur.  The difficulty arises in controlling the periods sufficiently well that we can count enough $\eta$-colorings to reach this contradiction.

Throughout this section, we assume that $\eta$ is a counterexample to Theorem~\ref{twoormore}, 
meaning that $\eta:\ZZ\to\A$ has at least two linearly independent one-sided nonexpansive directions and there exist $n,k\in\N$ such that $P_{\eta}(R_{n,k})\leq\frac{nk}{2}$.  We remark that if $\eta$ were periodic, it could have at most one one-sided nonexpansive direction.  Therefore we can assume that $\eta$ is aperiodic.

\subsection{An aperiodic counterexample with doubly periodic regions}
We use the existence of $\eta$ to construct $\alpha\in X_{\eta}$ which is aperiodic, but the restriction of $\alpha$ to a large convex subset of $\ZZ$ is doubly periodic.  This is carried out in three steps, first showing the existence of $f\in X_{\eta}$
which is singly, but not doubly, periodic (Section~\ref{sec:periodic-half}) and 
then using $f$ to show that there exists an aperiodic $\alpha\in X_{\eta}$ that 
is doubly periodic on a large convex region (Sections~\ref{sec:construction-alpha} and~\ref{sec:second-ambiguous}).

\subsubsection{ A periodic half plane}  
\label{sec:periodic-half}
By Lemma~\ref{stronggeneratingset},  there exists a strong $\eta$-generating set $\S$.  
Let $\ell_1$ be a one-sided nonexpansive direction for $\eta$.  By Lemma~\ref{possibledirections}, there is some $w_1\in E(\S)$ parallel to $\ell_1$.  By Lemma~\ref{nonexpansivepairs}, the direction antiparallel to $\ell_1$ is also one-sided nonexpansive and there is some $w_2\in E(\S)$ antiparallel to $\ell_1$.  By convexity, $\S$ is either $w_1$-balanced or $w_2$-balanced.  Without loss of generality, 
\begin{equation}
\label{balanced-assumption}
\text{we assume that $\S$ is }w_1\text{-balanced}
\end{equation}
and  that (see Remark~\ref{rem:standard}) $w_1$ points vertically downward.  

Set $\tilde{\S}:=\S\setminus w_1$ and set
\begin{eqnarray*}
H_i&:=&\{(x,y)\in\ZZ\colon x\geq i\}; \\
A_i&:=&\{(x,y)\in\ZZ\colon i\leq x<i+\diam_{w_1}(\tilde{\S})\}.
\end{eqnarray*}
By Lemma~\ref{ambiguousinorbit}, there exist $f,g\in X_{\eta}$ such that $f\rst{H_1}=g\rst{H_1}$ but $f\rst{H_0}\neq g\rst{H_0}$.  At most one of $f\rst{H_0}$ and $g\rst{H_0}$ has a horizontal period vector (in the sense of Definition~\ref{convexperiod}).  Thus we can assume that $f\rst{H_0}$ is not horizontally periodic.  By Proposition~\ref{prop:extendedambiguousperiod}, $f$ is vertically periodic,   and for every $i\in\Z$, $f\rst{A_i}$ has period at most $2\left|w_1\cap\S\right|-2$.  Moreover, by Lemma~\ref{ambiguousperiod2}, the vertical period of $f\rst{A_1}$ is at most $\left\lfloor\frac{\left|w_1\cap\S\right|}{2}\right\rfloor$.  By Proposition~\ref{balancedstronggeneratingset}, we can assume that $\left\lfloor\frac{\left|w_1\cap\S\right|}{2}\right\rfloor\leq\left|w_1\cap\S\right|-2$.

We also remark that, by the bound on the vertical period of $f\rst{A_i}$ for $i\geq0$, if $G\subset H_0$ is a convex set such that 
\begin{enumerate}[(H-I)]
\item $(G+(1,0))\subset G$; \label{condition1}
\item $G$ contains at least $2\left|w_1\cap\S\right|-2$ points on the $y$-axis, \label{condition2}
\end{enumerate}
then $f\rst{G}$ does not have a horizontal period vector in the sense of Definition~\ref{convexperiod}.

We summarize the main features of this construction:
\begin{enumerate}
\item $f\in X_{\eta}$; \label{property1}
\item $f$ is vertically periodic; \label{property2}
\item $f\rst{A_1}$ is vertically periodic of period at most $\left|w_1\cap\S\right|-2$; \label{property3}
\item The restriction of $f$ to an infinite convex set $G\subset H_0$ that satisfies conditions (H-\ref{condition1}) and (H-\ref{condition2}) cannot be extended to a horizontally periodic $\eta$-coloring of $\ZZ$. \label{property4}
\end{enumerate}

\subsubsection{Construction of aperiodic $\alpha\in X_{\eta}$ which agrees with $f$ on a large region}
\label{sec:construction-alpha}
Translating $\S$ if necessary, we may assume that 
$(0,0)\in w_1\subset\{(0,y)\colon y\in\Z\}$.  Using an inductive procedure, we define a function $\alpha\in X_{\eta}$ which is aperiodic but agrees with $f$ on an infinite, convex subset of $\ZZ$ (and is, therefore, vertically periodic on this subset).

%\m

\subsubsection*{Base case}  Let $F_1:=\S$ and $G_0=(0,0)$.  By Corollary~\ref{ambiguoussuperset} and aperiodicity of $\eta$, there exists $y_1\in\Z$ such that $(T^{(0,y_1)}f)\rst{F_1}$ does not extend uniquely to an $\eta$-coloring of the region 
$$
B_1:=\{(x,y)\in\ZZ\colon 0\leq x<\diam_{w_1}(\S)\}.
$$
Let $\alpha_1\in X_{\eta}$ be such that $\alpha_1\rst{F_1}=(T^{(0,y_1)}f)\rst{F_1}$, but $\alpha_1\rst{B_1}\neq (T^{(0,y_1)}f)\rst{B_1}$.  Let $G_1$ be a maximal, strongly $E(\S)$-enveloped subset of $B_1$ that contains $F_1$ and is such that $\alpha_1\rst{G_1}=(T^{(0,y_1)}f)\rst{G_1}$.

%\m

\subsubsection*{Inductive step}  Suppose that we have constructed sequences of convex, strongly $E(\S)$-enveloped, finite sets
$$
G_0\subseteq F_1\subseteq G_1\subseteq F_2\subseteq G_2\subseteq\cdots\subseteq F_i\subseteq G_i,
$$
configurations $\alpha_1,\dots,\alpha_i\in X_{\eta}$, and integers $y_1,\dots,y_i$ such that for $1\leq j\leq i$:
\begin{enumerate}
\item ({\em $F_j$ hypothesis}) $F_j$ contains both $G_{j-1}$ and $[0,j-1]\times[-j+1,j-1]$;
\item ({\em $\alpha_j$ hypothesis}) Defining the strip $B_j$ by 
$$
B_j:=\{(x,y)\in\ZZ\colon 0\leq x<\diam_{w_1}(F_j)\},
$$
then 
	\begin{enumerate}
	\item $\alpha_j\rst{F_j}=(T^{(0,y_j)}f)\rst{F_j}$;
	\item $\alpha_j\rst{B_j}\neq(T^{(0,y_j)}f)\rst{B_j}$;
	\end{enumerate}
\item ({\em $G_j$ hypothesis}) $G_j\subset B_j$ is a maximal set among all convex, strongly $E(\S)$-enveloped subsets of $B_j$ such that
	\begin{enumerate}
	\item $F_j\subseteq G_j$;
	\item $\alpha_j\rst{G_j}=(T^{(0,y_j)}f)\rst{G_j}$.
	\end{enumerate}
\end{enumerate}

\begin{figure}[ht]
     \centering
  \def\svgwidth{\columnwidth}
        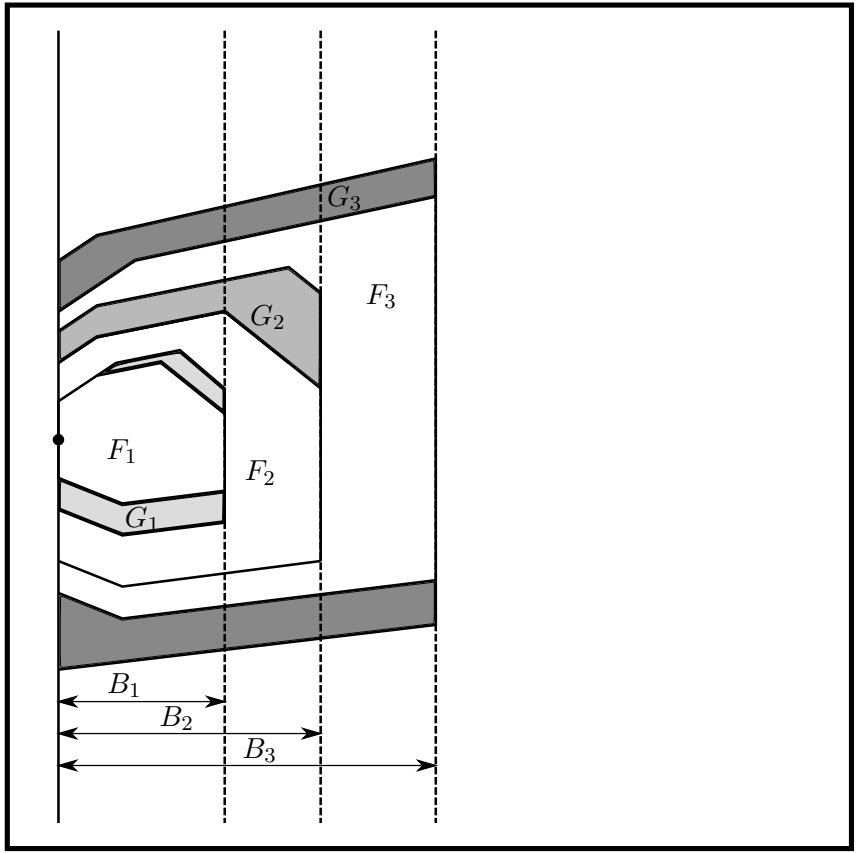
                \setlength{\abovecaptionskip}{-60mm}
	\caption{The sets $F_1\subseteq G_1\subseteq F_2\subseteq G_2\subseteq\cdots$.}
	\label{figured}
\end{figure}

Let $F_{j+1}\subset H_0$ be a finite, convex, strongly $E(\S)$-enveloped set containing both $G_j$ and $[0,j]\times[-j,j]$.  By Corollary~\ref{ambiguoussuperset} and aperiodicity of $\eta$, there exists  $y_{j+1}\in\Z$ such that $(T^{(0,y_{j+1})}f)\rst{F_{j+1}}$ does not extend uniquely to an $\eta$-coloring of the strip
$$
B_{j+1}:=\{(x,y)\in\ZZ\colon 0\leq x<\diam_{w_1}(F_{j+1})\}.
$$
Choose $\alpha_{j+1}\in X_{\eta}$ such that $\alpha_{j+1}\rst{F_{j+1}}=(T^{(0,y_{j+1})})\rst{F_{j+1}}$, but $\alpha_{j+1}\rst{B_{j+1}}\neq(T^{(0,y_{j+1})}f)\rst{B_{j+1}}$.  Let $G_{j+1}\subset B_{j+1}$ be a maximal strongly $E(\S)$-enveloped set containing $F_{j+1}$ such that $\alpha_{j+1}\rst{G_{j+1}}=(T^{(0,y_{j+1})}f)\rst{G_{j+1}}$.  By induction these functions, sets, and integers are defined for all $j$.

By vertical periodicity of $f$, we can assume that $y_j\in[0,(2\left|w_1\cap\S\right|-2)!)$ for all $j\in\Z$.  By passing to a subsequence, we can assume that the sequence $\{y_j\}_{j\in\N}$ is constant and, by replacing $f$ with $T^{(0,y_1)}f$ if necessary, we can assume that this constant is zero.

By construction, for each $j\in\N$, $E(G_j)$ has a downward oriented edge contained in the $y$-axis.  Let $(0,z_j)\in\ZZ$ be the topmost element of this edge and let
$$
\tilde{G}_j:=\{(x,y-z_j)\colon (x,y)\in G_j\}.
$$
Then $(T^{(0,z_j)}\alpha_j)\rst{\tilde{G}_j}=(T^{(0,z_j)}f)\rst{\tilde{G}_j}$ and there is no strongly $E(\S)$-enveloped convex subset of $B_j$ that strictly contains $\tilde{G}_j$ for which this is true (by maximality of $G_j$).  By vertical periodicity of $f$,
$\left\{T^{(0,z_j)}f\colon j\in\N\right\}\subseteq\{T^{(0,m)}f\colon0\leq m<(2\left|w_1\cap\S\right|-2)!\}$.  Let $z\in[0,(2\left|w_1\cap\S\right|-2)!)$ be such that $T^{(0,z_j)}f=T^{(0,z)}f$ for infinitely many $j$.  By passing to a subsequence, we can assume this holds for all $j$.  Define $\tilde{\alpha}_j:=T^{(0,z_j)}\alpha_j$ and $\tilde{f}:=T^{(0,z)}f$.  Then with this notation, $\tilde{\alpha}_j\rst{\tilde{G}_j}=\tilde{f}\rst{\tilde{G}_j}$ and there is no strongly $E(\S)$-enveloped subset of $B_j$ that strictly contains $\tilde{G}_j$ for which this holds.

Enumerate the vectors in $E(\S)$ as $u_1, u_2,\dots, u_m$ where $u_1=w_1$ and $u_{k+1}=\pred(u_k)$ for $k=1,\dots,m-1$ (recall that $\partial\S$ is positively oriented and $\pred(\cdot)$ is the predecessor edge with this orientation).  Let $K\in\N$ be the index for which $u_K=w_2$ (the edge of $\partial\S$ antiparallel to $w_1$).  
Since $\tilde{G}_j$ is $E(\S)$-enveloped for all $j$, define
$$
h(j,k)=\left\{\begin{tabular}{cl}
$\|a_{j,k}\|$ & if there is some $a_{j,k}\in E(\tilde{G}_j)$ parallel to $u_k$; \\
$0$ & otherwise.
\end{tabular}\right.
$$
Passing to a subsequence if necessary, we can assume that for each fixed $k=1,2,\dots,m$, the function $h(\cdot,k)$ is either constant or strictly increasing as a function of $j$.  By construction, $[0,j]\times[-j,j]\subseteq G_{j+1}\subset H_0$ for all $j$.  So $\bigcup_jG_j=H_0$ and there is at least one index $1<k<K$ for which $h(\cdot,k)$ is unbounded.  
\begin{equation}
\label{eq-growingedge}
\text{Let }1<k_{\min}<K\text{ be the least integer for which this holds.}
\end{equation}
Define integers $1<k_1<\dots<k_s<k_{\min}$ to be the indices in this interval for which $h(\cdot,k)$ is eventually positive.  Let $v_1,\dots,v_s\in E(G_1)$ be the edges for which $v_i$ is parallel to $u_{k_i}$.  We emphasize that by construction, $v_1,\dots,v_s\in E(G_i)$ for all $i\geq1$, meaning that not only does $G_i$ have an edge parallel to $v_1, v_2, \ldots$, but these fixed line segments are edges of $G_i$.  Set
$$
G_{\omega}:=\bigcup_{j=1}^{\infty}\tilde{G}_j.
$$
Then $G_{\omega}$ is convex and $E(\S)$-enveloped (see Figure~\ref{figurep} -- also observe that it is not strongly $E(\S)$-enveloped since it doesn't have edges parallel to each of the edges of $\S$).  Moreover $E(G_{\omega})$ is comprised of $v_1,\dots,v_s$, as well as $\{(0,y)\in\ZZ\colon y\leq0\}$, and a semi-infinite edge parallel to $u_{k_{\min}}$.

\begin{figure}[ht]
     \centering
  \def\svgwidth{\columnwidth}
        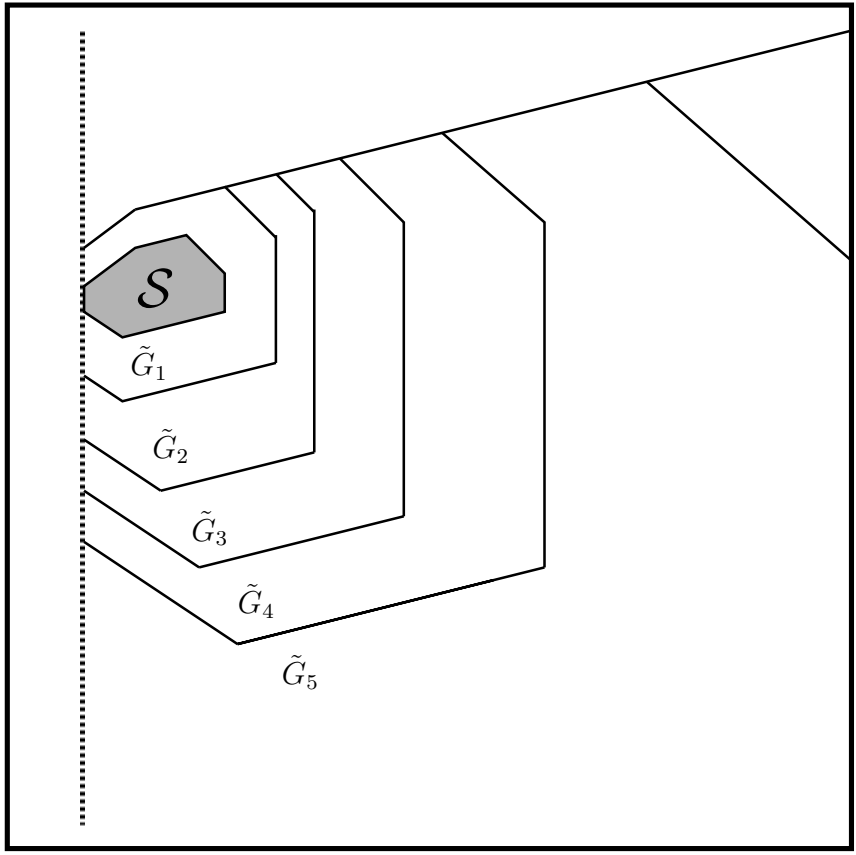
                \setlength{\abovecaptionskip}{-60mm}
	\caption{The sets $\S\subseteq\tilde{G}_1\subseteq\tilde{G}_2\subseteq\cdots\subset G_{\omega}$.}
	\label{figurep}
\end{figure}

By compactness, the sequence $\{\tilde{\alpha}_j\}_{j\in\N}$ has an accumulation point.  Let $\alpha\in X_{\eta}$ be such a point.  By passing to a subsequence, we can assume that for all $1\leq j_1<j_2$ we have $\tilde{\alpha}_{j_2}\rst{\tilde{G}_{j_1}}=\tilde{\alpha}_{j_1}\rst{\tilde{G}_{j_1}}$.  By construction, $\alpha\rst{G_{\omega}}=\tilde{f}\rst{G_{\omega}}$.  In particular, $\alpha\rst{G_{\omega}}$ is vertically periodic (in the sense of Definition~\ref{convexperiod}) and the restriction of $\alpha$ to any semi-infinite $(\tilde{\S},w_1)$-strip in $G_{\omega}$ has period at most $2\left|w_1\cap\S\right|-2$.  Moreover, the restriction of $\alpha$ to the $(\tilde{\S},w_1)$-border of $G_{\omega}$ has period at most $\left|w_1\cap\S\right|-2$ (because $\tilde{f}=T^{(0,z)}f$, the $(\tilde{\S},w_1)$-border of $G_{\omega}$ is a subset of $A_1$, and this bound on the period was shown for $f\rst{A_1}$ in Section~\ref{sec:periodic-half}).

\subsubsection{Another one-sided nonexpansive direction for $\alpha$}
\label{sec:second-ambiguous}  
Next, we show that both semi-infinite edges in $E(G_{\omega})$ are one-sided $\eta$-nonexpansive.
	
Let $\Ext_{u_{k_{\min}}}(G_{\omega})$ denote the $u_{k_{\min}}$-extension of $G_{\omega}$ (recall Definition~\ref{envelopeddef2}).  Since the boundary edge of $\Ext_{u_{k_{\min}}}(G_{\omega})$ parallel to $u_{k_{\min}}$ is semi-infinite, we have that $\Ext_{u_{k_{\min}}}(G_{\omega})\setminus G_{\omega}$ is equal to the intersection of $\ZZ$ 
with the disjoint union of finitely many semi-infinite lines $l_1,\dots,l_{r_1}$ parallel to $u_{k_{\min}}$.  We denote the subextensions by
\begin{equation*}
%\label{eq:Gomega+}
G_{\omega}^{(i)}:=G_{\omega}\cup(l_1\cap\ZZ)\cup\cdots\cup(l_i\cap\ZZ)
\end{equation*}
for $i=1,\dots,r_1$.

We now inductively define an increasing sequence of sets $\{G_{\omega}^{(i)}\}_{i\in\N}$.  Suppose we have constructed integers $r_1,\dots,r_m\in\N$ and an increasing sequence of convex sets $\{G_{\omega}^{(i)}\}_{i=1}^{r_1+\cdots+r_m}$ such that for all $j=1,\dots,m$, the sets $G_{\omega}^{(r_1+\cdots+r_j)}$ are $E(\S)$-enveloped and each has a semi-infinite edge parallel to $u_{k_{\min}}$ (the edge is not required to be the same for all of the sets).  Then $\Ext_{u_{k_{\min}}}(G_{\omega}^{(r_1+\cdots+r_m)})\setminus G_{\omega}^{(r_1+\cdots+r_m)}$ is nonempty and can be written as the intersection of $\ZZ$ with the disjoint union of $r_{m+1}$ semi-infinite lines $l_{r_1+\cdots+r_m+1},\dots,l_{r_1+\cdots+r_{m+1}}$ (for some $r_{m+1}\in\N$).  For $r_1+\cdots+r_m<i\leq r_1+\cdots+r_{m+1}$ define
$$
G_{\omega}^{(i)}:=G_{\omega}\cup(l_1\cap\ZZ)\cup\cdots\cup(l_i\cap\ZZ).
$$
This defines a sequence of integers $\{r_m\}_{m\in\N}$ and sets $\{G_{\omega}^{(i)}\}_{i\in\N}$.

Recall that for all $j$, $\tilde{G}_j$ is strongly $E(\S)$-enveloped and the length of the boundary edge parallel to $u_{k_{\min}}$ increases monotonically in $j$, by~\eqref{eq-growingedge}.  Thus for $j$ sufficiently large, $\Ext_{u_{k_{\min}}}(\tilde{G}_j)\neq\tilde{G}_j$.  Moreover the depth of the extension $\Ext_{u_{k_{\min}}}(\tilde{G}_j)$ depends only on the slopes of the lines determined by the boundary edges of $\tilde{G}_j$ (recall Definition~\ref{envelopeddef2}).  Therefore if $d_j$ is the depth of the extension $\Ext_{u_{k_{\min}}}(\tilde{G}_j)$, then $d_j$ is bounded (in $j$) and there is some $d\in\N$ such that $d_j=d$ for infinitely many $j$.  We pass to a subsequence such that this holds for all $j$.

Let $\tilde{l}(j,k)$ be the intersection of $\Ext_{u_{k_{\min}}}(\tilde{G}_j)$ and $l_k$.  Then $\Ext_{u_{k_{\min}}}(G_j)$ can be written as the disjoint union
$$
\tilde{G}_j\sqcup\bigsqcup_{k=1}^d\tilde{l}(j,k).
$$
By construction (recall the inductive hypotheses for $G_j$ at the beginning of this subsection) $f\rst{\Ext_{u_{k_{\min}}}(\tilde{G}_j)}\neq\alpha\rst{\Ext_{u_{k_{\min}}}(\tilde{G}_j)}$.  Let $1\leq a\leq d$ be the smallest integer for which $\tilde{\alpha}_j\rst{\tilde{l}(j,a)}\neq\tilde{f}\rst{\tilde{l}(j,a)}$ for infinitely many $j$.  Passing to a subsequence, we can assume that for all $j\in\N$, $\tilde{\alpha}_j\rst{\tilde{l}(j,k)}=\tilde{f}\rst{\tilde{l}(j,k)}$ for all $1\leq k<a$, but $\tilde{\alpha}_j\rst{\tilde{l}(j,a)}\neq\tilde{f}\rst{\tilde{l}(j,a)}$.  

Let  $w_3\in E(\S)$ denote the edge parallel to $u_{k_{\min}}$.
By Corollary~\ref{uniqueextension}, there are never $\left|w_3\cap\S\right|-1$ consecutive integer points on $\tilde{l}(j,a)$ where $\tilde{\alpha}_j$ and $\tilde{f}$ coincide (otherwise they would coincide everywhere on $\tilde{l}(j,a)$ since $\S$ is $\eta$-generating).  In particular, there are never $\left|w_3\cap\S\right|-1$ consecutive integer points on $l_a$ where $\alpha$ and $\tilde{f}$ coincide.  As a result, the restriction of $\alpha\rst{G_{\omega}^{(a-1)}}$ does not extend uniquely to an $\eta$-coloring of $G_{\omega}^{(a)}$.

Moreover, since $f$, and hence $\tilde{f}$, is vertically periodic and $\alpha\rst{G_{\omega}^{(a-1)}}=\tilde{f}\rst{G_{\omega}^{(a-1)}}$ but $\alpha\rst{G_{\omega}^{(a)}}\neq\tilde{f}\rst{G_{\omega}^{(a)}}$, $\alpha$ is not vertically periodic.  Moreover, because the boundary edge of $G_{\omega}^{(a-1)}$ parallel to $w_3$ is semi-infinite there is an ambiguous coloring of a $w_3$-half plane, obtained by passing to appropriate accumulation points of $\{T^{-m\cdot w_3}\alpha\}_{m=1}^{\infty}$ and $\{T^{-m\cdot w_3}\tilde{f}\}_{m=1}^{\infty}$ (viewing $w_3$ as a vector).  Thus by Lemma~\ref{ambiguousinorbit}, $w_3$ is a one-sided nonexpansive direction for $\eta$ and by Lemma~\ref{nonexpansivepairs}, there is an edge $w_4\in E(\S)$ antiparallel to $w_3$.

\subsubsection{Construction of $K$}
\label{constructionofK}
We show that there is an infinite, convex subset $K$ of $G_{\omega}$ such that $\alpha\rst{K}$ is doubly periodic.  The construction has four steps which are illustrated in Figure~\ref{construction-of-K}.

\subsubsection*{Step 1}  By Lemma~\ref{balancedlemma}, there exists a $w_3$-balanced set $\S_1$.  Let $\hat{w}_3\in E(\S_1)$ be the edge parallel to $w_3$ and let $\tilde{S}_1:=\S_1\setminus \hat{w}_3$.  Recall the integer $a\in\N$ defined in Section~\ref{sec:second-ambiguous} is such that $\tilde{f}\rst{G_{\omega}^{(a-1)}}=\alpha\rst{G_{\omega}^{(a-1)}}$ but $\tilde{f}\rst{G_{\omega}^{(a)}}\neq\alpha\rst{G_{\omega}^{(a)}}$.  By Lemma~\ref{semiinfiniteperiodicextension}, the $(\tilde{\S}_1,\hat{w}_3)$-border of $G_{\omega}^{(a-1)}$ is $w_3$-eventually periodic.  This is illustrated in Figure~\ref{test1}.

\subsubsection*{Step 2} Let $\mathcal{B}$ denote the $(\tilde{\S}_1,\hat{w}_3)$-border of $G_{\omega}^{(a-1)}$.  Since  $\tilde{f}$ is vertically periodic and $\tilde{f}\rst{G_{\omega}^{(a-1)}}=\alpha\rst{G_{\omega}^{(a-1)}}$, $\alpha\rst{G_{\omega}^{(a-1)}}$ is vertically periodic (in the sense of Definition~\ref{convexperiod}).  Let $p\in\N$ be the minimal vertical period of $\alpha\rst{G_{\omega}^{(a-1)}}$ such that $(T^{(0,p)}\alpha)\rst{(T^{(0,-p)}G_{\omega}^{(a-1)})}=\alpha\rst{(T^{(0,-p)}G_{\omega}^{(a-1)})}$.  Then the restriction of $\alpha$ to any set of the form $T^{(0,-mp)}\mathcal{B}$ is eventually $w_3$-periodic, with the same eventual period and the same gap.  This is illustrated in Figure~\ref{test2}.

\subsubsection*{Step 3} The set $T^{(0,-p)}\mathcal{B}$ is a semi-infinite $(\tilde{\S}_1,\hat{w_3})$-strip (recall Definition~\ref{semiinfinitedefinition}).  Let $\mathcal{B}_1:=\Ext_{w_3}(T^{(0,-p)}\mathcal{B})$ be the $w_3$-extension of $\mathcal{B}$.  Now inductively let $\mathcal{B}_{i+1}:=\Ext_{w_3}(\mathcal{B}_i)$ for $i\in\N$.  Then there is some $j\in\N$ such that $\mathcal{B}_j$ contains all but finitely many elements of $G_{\omega}^{(a-1)}\setminus T^{(0,-p)}G_{\omega}^{(a-1)}$.  Since $\alpha\rst{T^{(0,-p)}\mathcal{B}}$ is eventually $w_3$-periodic, Proposition~\ref{semiinfiniteperiodicextension} guarantees that $\alpha\rst{\mathcal{B}_j}$ is also eventually $w_3$-periodic with the same gap but possibly larger eventual period.  This is illustrated in Figure~\ref{test3}.

\subsubsection*{Step 4}  Since $\alpha\rst{G_{\omega}^{(a-1)}}$ is vertically periodic with period $p$,  the restriction of $\alpha$ to $\bigcup_{m=1}^{\infty}T^{(0,-mp)}(\mathcal{B}_{j+1}\cap G_{\omega}^{(a-1)})$ is eventually $w_3$-periodic.  Hence there is some $q\in\N$ such that the restriction of $\alpha$ to $T^{(q,0)}\bigcup_{m=1}^{\infty}T^{(0,-mp)}(\mathcal{B}_{j+1}\cap G_{\omega}^{(a-1)})$ is $w_3$-periodic. Since $\mathcal{B}_{j+1}$ is a semi-infinite $w_3$-strip and $\mathcal{B}_{j+1}\cap T^{(0,-p)}\mathcal{B}_{j+1}\neq\emptyset$, there is an $E(\S)$-enveloped, convex set $\tilde{K}\subseteq T^{(q,0)}\bigcup_{m=1}^{\infty}T^{(0,-mp)}(\mathcal{B}_{j+1}\cap G_{\omega}^{(a-1)})$ whose boundary has semi-infinite edges parallel to $w_3$ and $w_1$.  The restriction of $\alpha$ to $\tilde{K}$ is doubly periodic.  Let $K$ be the largest (with respect to inclusion) convex set containing $\tilde{K}$ for which $\alpha\rst{K}$ is doubly periodic.  By construction $\alpha\rst{G_{\omega}^{(a)}}$ is not doubly periodic since it differs from the vertically periodic coloring $\tilde{f}$.  Therefore $K$ has a semi-infinite edge parallel to $w_3$.  Since $\alpha\rst{G_{\omega}^{(a-1)}}=\tilde{f}\rst{G_{\omega}^{(a-1)}}$ was constructed such that it is not horizontally periodic (in the sense of Definition~\ref{convexperiod}), the set $K$ has a semi-infinite edge parallel to $w_1$.  This is illustrated in Figure~\ref{test4}.

\begin{figure}
	\begin{subfigure}[t]{0.425\textwidth}
		\centering
		  \def\svgwidth{\columnwidth}
        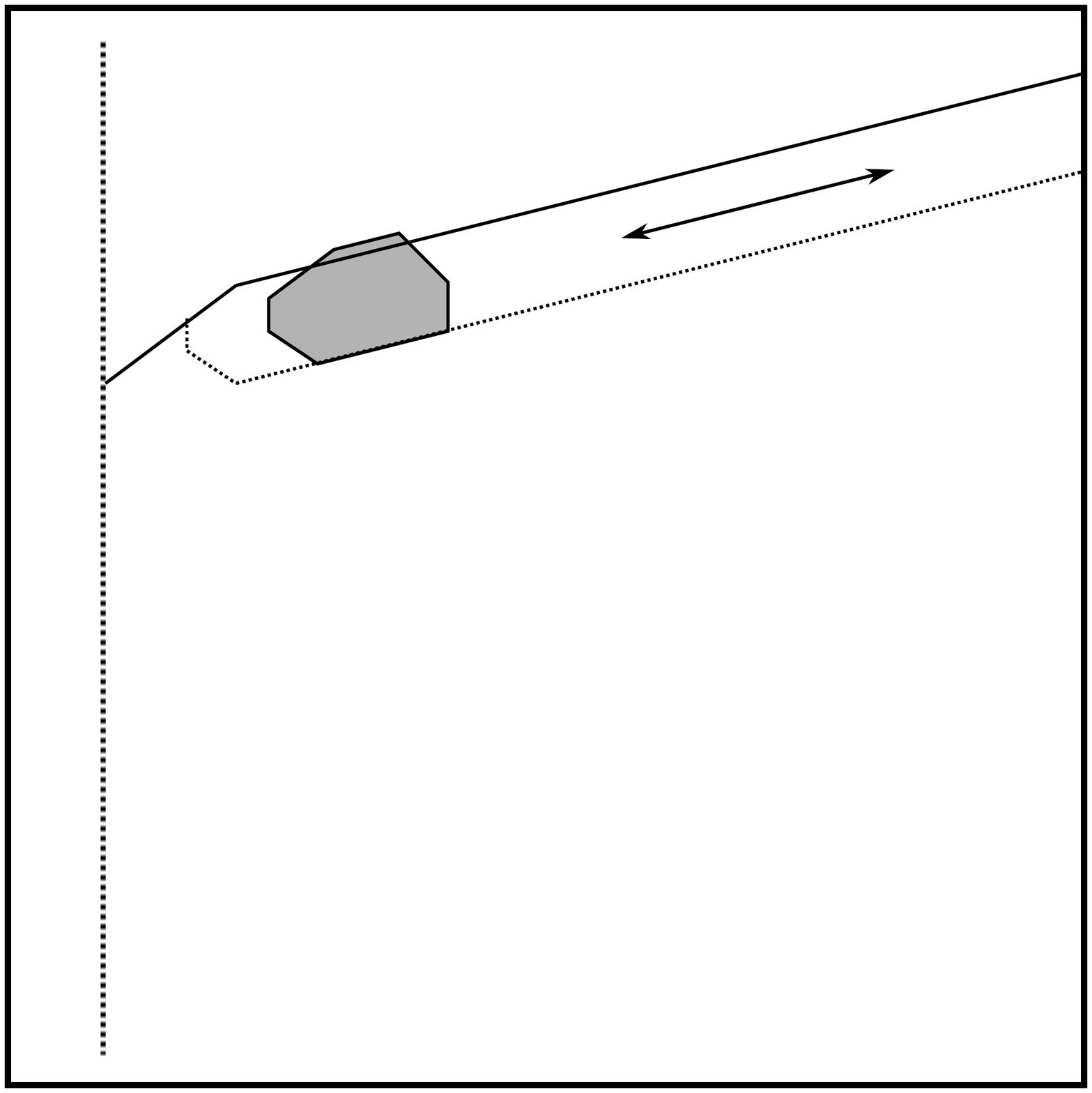
                \setlength{\abovecaptionskip}{-25mm}
		\caption{The shaded set $\S_1$ and the set $G_{\omega}^{(a-1)}$ with its $(\tilde{S}_1,w_3)$-border.  The arrow indicates the region that is eventually $w_3$-periodic.}
		\label{test1}
	\end{subfigure}
	\hspace{0.5 in}
	\begin{subfigure}[t]{0.425\textwidth}
		\centering
		  \def\svgwidth{\columnwidth}
        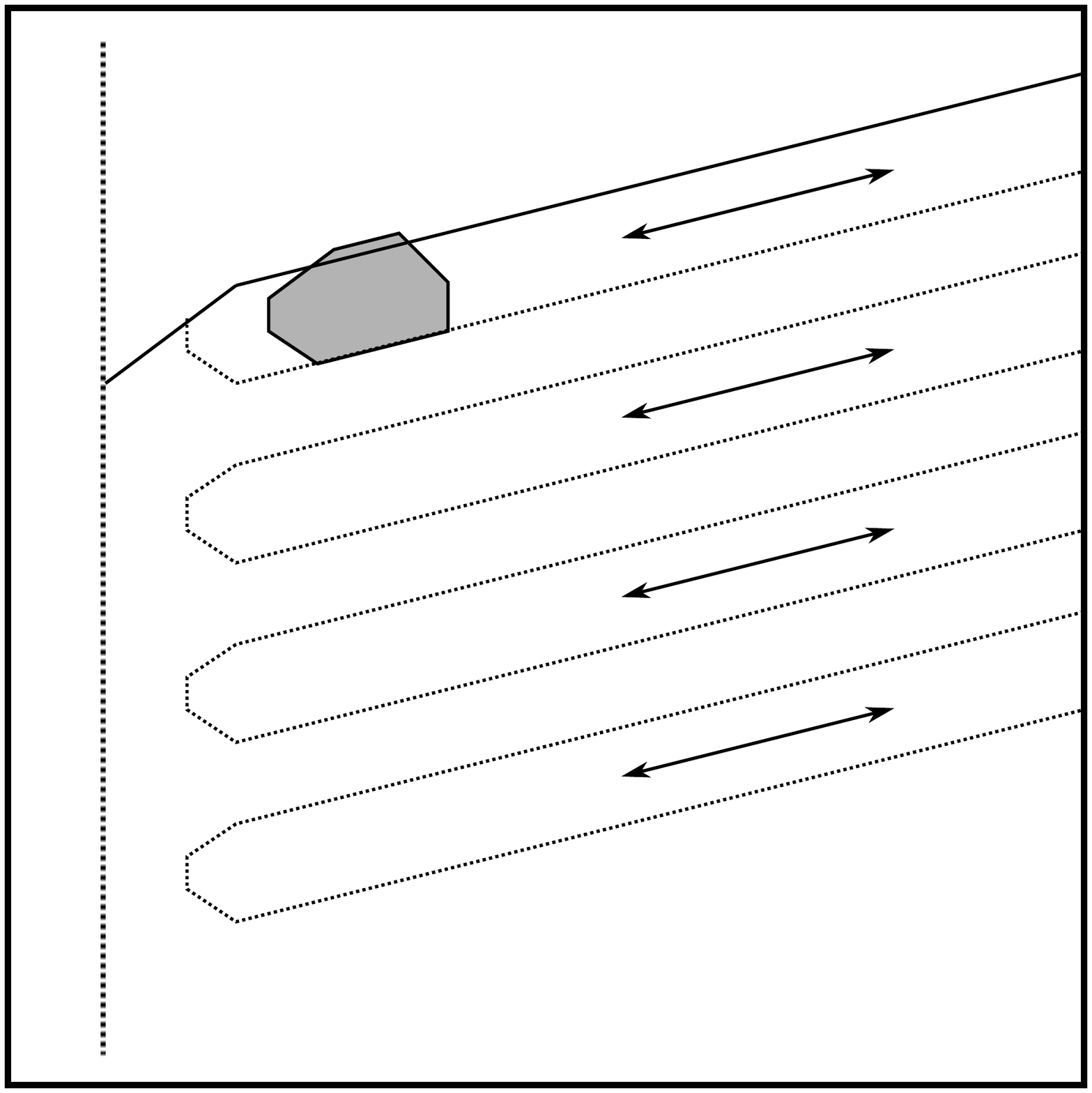
                \setlength{\abovecaptionskip}{-25mm}
		\caption{$\alpha\rst{G_{\omega}^{(a-1)}}$ is vertically periodic.  Translations of $\mathcal{B}$ are shown and  $\alpha$ is eventually periodic on each translated set.}
		\label{test2}
	\end{subfigure}
	\hspace{0.5 in}
	\begin{subfigure}[t]{0.425\textwidth}
		\centering
  \def\svgwidth{\columnwidth}
        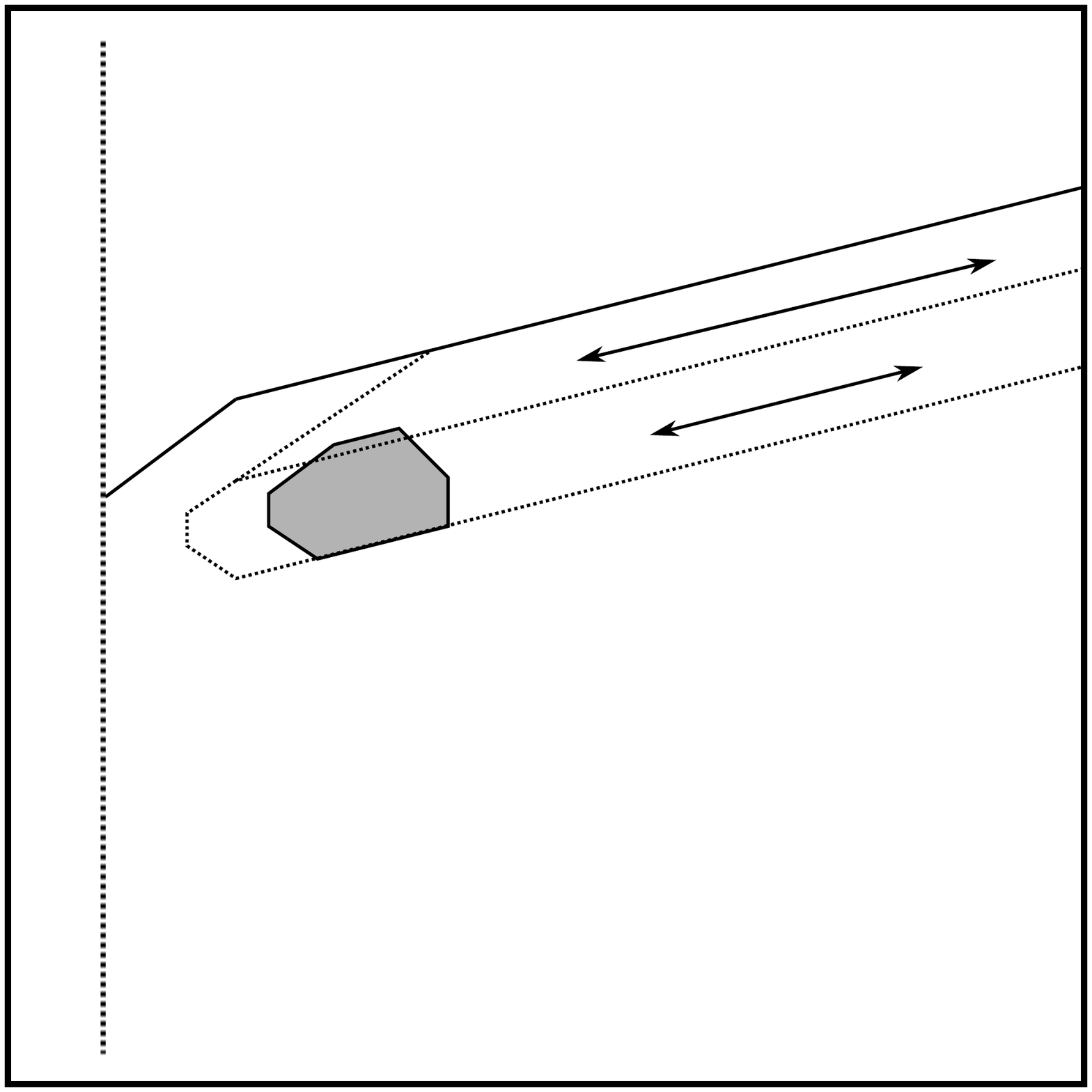
                \setlength{\abovecaptionskip}{-25mm}
                		\caption{The semi-infinite $(\tilde{\S}_1,w_3)$-strip is
		eventually $w_3$-periodic and so any $\eta$-coloring of its
		$w_3$-extension is also $w_3$-periodic
		(possibly of larger period), by
		Proposition~\ref{semiinfiniteperiodicextension}.  The arrows indicate the direction of 
		eventual periodicity and the possibly different periods.}
		\label{test3}
	\end{subfigure}
	\hspace{0.5 in}
	\begin{subfigure}[t]{0.425\textwidth}
		\centering
  \def\svgwidth{\columnwidth}
        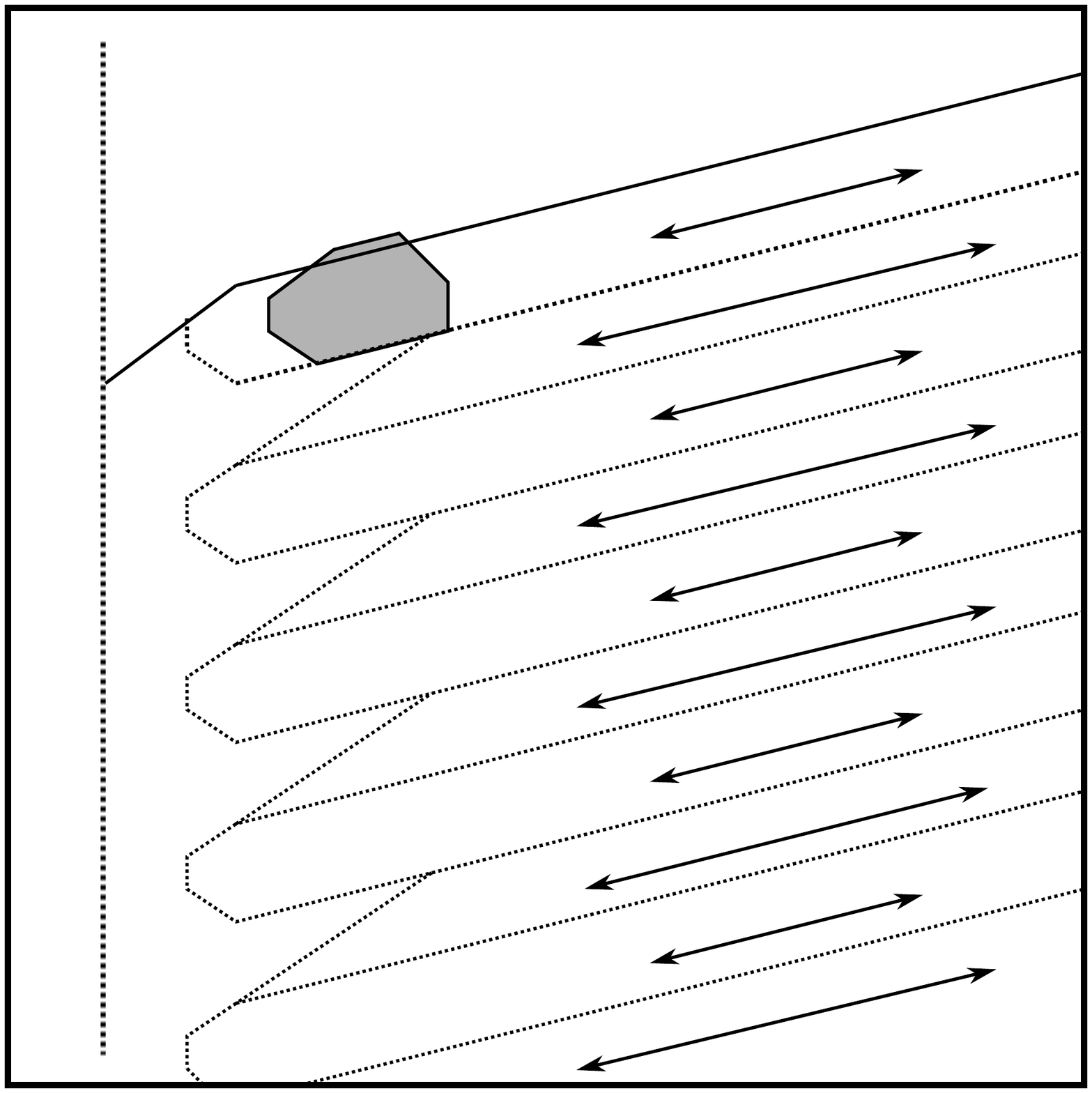
                \setlength{\abovecaptionskip}{-25mm}
		\caption{$\alpha\rst{G_{\omega}^{(a-1)}}$ is vertically periodic and so there is an infinite convex region where $\alpha$ is doubly periodic.  $K$ is the largest convex set for which this holds.}
		\label{test4}
	\end{subfigure}
	\caption{Construction of $K$ }
	\label{construction-of-K}
\end{figure}

\subsection{Bounds on the period of $\alpha$}
In this section, we show that we have strong bounds on the $w_1$- and $w_3$-periods of $\alpha\rst{K}$, 
first extending the region $K$ to a larger region where $\alpha$ is singly (but not doubly) periodic (Section~\ref{sec:periodic-extensions})
and then producing a generating set with particular properties that imply the bounds (Sections~\ref{sec:thin-generating} and~\ref{sec:bounding-periods}).  The existence of this type of generating set strongly relies on the bound of $P_\eta(n,k)\leq \frac{nk}{2}$.

\subsubsection{Periodic extensions} 
\label{sec:periodic-extensions} 
We show that the region $K$ on which $\alpha$ is doubly periodic can be extended 
to a larger region on which $\alpha$ is singly, but not doubly, periodic.

Let $K_0:=K$.  Since $K$ has a semi-infinite edge parallel to $w_3$, $\Ext_{w_3}(K)\neq K$ and there exist semi-infinite lines $\ell_1,\dots,\ell_{f_1}$ parallel to $w_3$ such that 
$$
\Ext_{w_3}(K)\setminus K=\bigsqcup_{i=1}^{f_1}(\ZZ\cap\ell_i).
$$
For $1\leq i\leq f_1$, set $K_i:=K\cup\ell_1\cup\cdots\cup\ell_i$.  We continue inductively:  having defined integers $f_1,\dots,f_j$ and sets $K_1,\dots,K_{f_1+\cdots+f_j}$ such that $K_{f_1+\cdots+f_j}$ contains a semi-infinite boundary edge parallel to $w_3$, there exist an integer $f_{j+1}$ and semi-infinite lines $\ell_{f_1+\cdots+f_j+1},\dots,\ell_{f_1+\cdots+f_{j+1}}$ such that
$$
\Ext_{w_3}(K_{f_1+\cdots+f_j})\setminus K_{f_1+\cdots+f_j}=\bigsqcup_{i=1}^{f_{j+1}}(\ZZ\cap\ell_{f_1+\cdots+f_j+i})
$$
By the second claim of Proposition~\ref{semiinfiniteperiodicextension}, the restriction of $\alpha$ to $\bigcup_{i=1}^{\infty}K_i$ is $w_3$-periodic with period at most $2\left|w_3\cap\S\right|-2$.

\subsubsection{A thin generating set}
\label{sec:thin-generating} 
We use the assumption on complexity $P_\eta(n,k)\leq\frac{nk}{2}$ to show 
that we have a generating set with a small diameter (recall Definition~\ref{diameterdefinition}).

Let $x_{\min}$ and $x_{\max}$ denote the minimal and maximal $x$-coordinates of elements of $\S$.  Let $d:=\left\lfloor\frac{x_{\max}-x_{\min}+1}{2}\right\rfloor$ and let the left subset of $\S$ be defined by
$$
\S_L:=\{(x,y)\in\S\colon x_{\min}\leq x\leq x_{\min}+d\}.
$$
If $\left|\S_L\right|\geq\frac{1}{2}\left|\S\right|$, then by Lemma~\ref{generated3}, $D_{\eta}(\S_L)\leq0$.  Let $\S_2\subset\S_L$ be an $\eta$-generating set.  Otherwise $D_{\eta}(\S\setminus\S_L)\leq0$.  In this case, let $\S_2\subset(\S\setminus\S_L)$ be an $\eta$-generating set.  In both cases, let $u\in E(\S_2)$ be the edge parallel to $w_3$ (which exists by  one-sided nonexpansiveness of $w_3$ and Lemma~\ref{possibledirections}) and let $v\in E(\S_2)$ be the edge parallel to $w_1$.  By construction
\begin{equation}\label{calculation3}
\diam_v(\S_2)\leq\left\lceil\frac{\diam_v(\S)}{2}\right\rceil.
\end{equation}
We call the set $\S_2$ a {\em thin generating set} for $\eta$.

\subsubsection{Bounding the periods of $\alpha\rst{K}$}
\label{sec:bounding-periods}  
Using the generating set $\S_2$ and the construction of $\alpha$, 
we obtain strong bounds on periods of $\alpha\rst{K}$.  A key tool we use is the classic Fine-Wilf Theorem:
\begin{theorem*}[Fine-Wilf Theorem~\cite{FW}]
Suppose $\{f_n\}_{n=0}^{\infty}$ and $\{g_n\}_{n=0}^{\infty}$ are two periodic sequences of periods $p$ and $q$, respectively.  If $f_n=g_n$ for $p+q-\gcd(p,q)$ consecutive entries, then $f_n=g_n$ for all $n$.  Moreover, $p+q-\gcd(p,q)$ is the minimum number of consecutive entries that make this property hold.
\end{theorem*}
\begin{corollary}\label{FineWilfCor}
If $\{f_n\}_{n=0}^{\infty}$ is a periodic sequence of period at most $p$, then the sequence can be reconstructed uniquely from any $2p-2$ consecutive entries.  Moreover, if the exact value of the period is unknown (other that it is no greater than $p$), then $2p-2$ is the smallest number of consecutive entries that suffices.
\end{corollary}
\noindent {\em Proof}.
If for some $n_0\in\N$ we are given the value of $f_{n_0}, f_{n_0+1},\dots,f_{n_0+2p-3}$, we call a number $q\leq p$ a {\em possible} period for $f$ if $f_k=f_{k+q}$ for all $n\leq k\leq n_0+2p-q-3$.  Let $\Phi\subseteq\{1,2,\dots,p\}$ be the set of all possible periods for the sequence $\{f_n\}_{n=0}^{\infty}$.  For each $q\in\Phi$, let $\{f^q_n\}_{n=0}^{\infty}$ be the unique $q$-periodic sequence such that $f^q_{n_0+i}=f_{n_0+i}$ for all $i=0,1,\dots,2p-3$.  By the Fine-Wilf Theorem, if $q_1, q_2\in\Phi$, then $f^{q_1}_k=f^{q_2}_k$ for all $k$ (since $q_1+q_2-\gcd(q_1,q_2)\leq2p-2$) so there is a unique sequence of period at most $p$ that agrees with the information given about $\{f_n\}_{n=0}^{\infty}$.

This bound is optimal: given $n\in\N$ let $w_n$ be the word
$$
w_n:=\underbrace{00\cdots0}_{n}1\underbrace{00\cdots0}_{n}
$$
and let $p=n+2$.  Then $2p-2>2n+1=\left|w_n\right|$ and $w_n$ can indeed be extended in two different ways to a periodic sequence of period at most $p$ (one of period $n+1$ and one of period $n+2$).  So if we were told that $\{f_n\}_{n=0}^{\infty}$ is a periodic sequence of period at most $p$ and that the first $2n+1$ entries were the word $w_n$, then the sequence could not be uniquely reconstructed.$\hfill\diamond$
\begin{definition}\label{perFineWilf}
If $w=(w_0,w_1,\dots,w_{m-1})\in\A^m$ is a word of length $m$ and $p\in\{1,\dots,m-1\}$, we say that $w$ is {\em periodic of period $p$} if $w_i=w_{i+p}$ for all $0\leq i<m-p$.  We also make the convention that every word of length $m$ is periodic of period $m$.
\end{definition}
\begin{notation}
\label{notation:T1}
There are two distinguished semi-infinite strips in $K$ and we label them:
\begin{itemize}
\item Let $\T_1$ denote the $(\S\setminus w_3,w_3)$-border of $K$.
\item Let $\T_2$ denote the $(\S\setminus w_1,w_1)$-border of $K$.
\end{itemize}
\end{notation}
In the remainder of this section, we show:
\begin{claim}
\label{claim:bounds-on-periods}
Maintaining notation as above,
\begin{enumerate}[(i)]
\item The $w_1$-period of $\alpha\rst{K}$ is at most $\left\lfloor\frac{\left|w_1\cap\S\right|}{2}\right\rfloor$;
\item The $w_3$-period of $\alpha\rst{K}$ is at most $\left|w_3\cap\S\right|-1$.
\end{enumerate}
\end{claim}

By convexity, $\S$ is either $w_3$- or $w_4$-balanced.  We prove 
the claim by considering three cases separately.

\subsubsection*{Case 1} Suppose $\S$ is $w_3$-balanced.  It is immediate that $\left|w_3\cap\ZZ\right|\leq\left|w_4\cap\ZZ\right|$.  By assumption~\eqref{balanced-assumption}, $\S$ is also $w_1$-balanced.  In this case we show the claimed bound on the $w_1$-period of $\alpha\rst{K}$ but  prove (the stronger bound) that the $w_3$-period of $\alpha\rst{K}$ is at most $\left\lfloor\frac{\left|w_3\cap\S\right|}{2}\right\rfloor$.
   
By maximality of $K$, $\alpha\rst{K}$ is $(\S,w_3,\eta)$-ambiguous.
    Since $\alpha\rst{K}$ is doubly periodic, $\alpha\rst{\T_1}$ is periodic.  Thus by
    Corollary~\ref{semiinfiniteperiodicextension2},
    \begin{equation}\label{T1assumption}
    \alpha\rst{\T_1}\text{ is periodic with period vector parallel to $w_3$}
    \end{equation}
    and period at most $\left\lfloor\frac{\left|w_3\cap\S\right|}{2}\right\rfloor$.
By vertical periodicity of $\alpha\rst{K}$, there is some $p\in\N$ such that the colorings $(T^{(0,mp)}\alpha)\rst{\T_1}$ coincide for all $m=0,1,2\dots$

Again by maximality of $K$, $\alpha\rst{K}$ is $(\S,w_1,\eta)$-ambiguous and  so $\alpha\rst{\T_2}$ is vertically periodic with period at most $\left\lfloor\frac{\left|w_1\cap\S\right|}{2}\right\rfloor$.  Then there is some $q\in\N$ such that the colorings $(T^{-mq\cdot w_3}\alpha)\rst{\T_1}$ coincide for all $m=0,1,2\dots$ (here $w_3$ is understood as a vector rather than a line segment).

Since $\T_1$ is a semi-infinite $(\S\setminus w_3,w_3)$-strip and $\T_2$ is a semi-infinite $(\S\setminus w_1,w_1)$-strip, there exist $m_1, m_2\in\N$ such that 
$$
P:=(\T_1-(0,m_1p))\cap(\T_2-m_2q\cdot w_3)\cap\ZZ
$$ is the intersection of $\ZZ$ with a parallelogram, with sides parallel to $w_1$ and $w_3$ and integer vertices.  This is illustrated in Figure~\ref{figuree}.

\begin{figure}[ht]
	\centering
  \def\svgwidth{\columnwidth}
        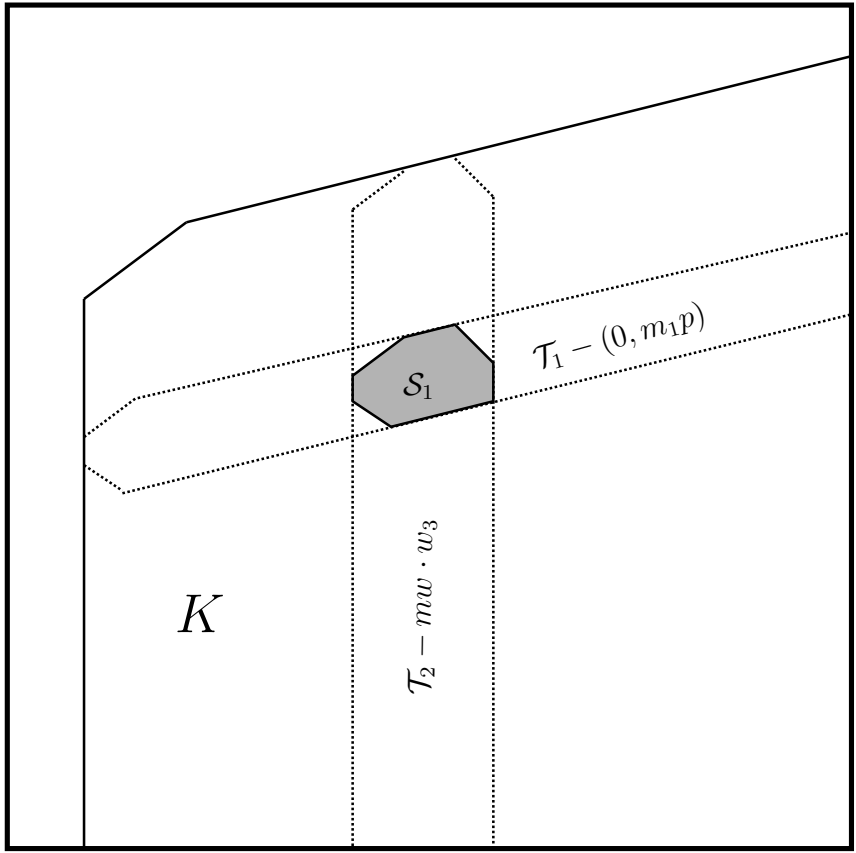
                \setlength{\abovecaptionskip}{-60mm}
                	\caption{The restriction of $\alpha$ to each of the strips is periodic in the direction determined by the strip and the period is at most half of the side length of the parallelogram.}
	\label{figuree}
\end{figure}

Since $\S$ is $w_3$-balanced, if $L$ is any line parallel to $w_3$ that has nonempty intersection with $P$, then
\begin{equation}\label{Pwidth}
\left|L\cap P\right|\geq\left|w_3\cap\S\right|-1.
\end{equation}
Therefore, we claim that if $L$ is any line parallel to $w_3$ that has nonempty intersection with $P$, there is a unique $\A$-coloring of $L\cap\ZZ$ such that
\begin{enumerate}
\item the coloring coincides with $\alpha$ on $L\cap P$;
\item the coloring is periodic of period at most $\left\lfloor\frac{\left|w_3\cap\S\right|}{2}\right\rfloor$.
\end{enumerate}
To see this, by~\eqref{T1assumption} and the definition of $p$, we have that 
$\alpha\rst{(\T_1-m_1p)}$ is $w_3$-periodic of period at most
$\left\lfloor\frac{\left|w_3\cap\S\right|}{2}\right\rfloor$, and by~\eqref{Pwidth},
the coloring of $L$ agrees with $\alpha$ on at least $\left|w_3\cap\S\right|-1$ consecutive integer points: therefore we can apply Corollary~\ref{FineWilfCor}, showing that the coloring agrees with $\alpha$ on $L\cap(\T_1-m_1p)$ (which is semi-infinite).  There is only one way to extend a semi-infinite periodic coloring to a doubly infinite periodic coloring of $L$.

Since $\S$ is also $w_1$-balanced, a similar result holds for lines parallel to $w_1$: if $L$ is any line parallel to $w_1$ that has nonempty intersection with $P$ then
$$
\left|L\cap P\right|\geq\left|w_1\cap\S\right|-1.
$$
Therefore, since $\alpha\rst{(\T_2-m_2q\cdot w_3)}$ is $w_1$-periodic of period at most $\left\lfloor\frac{\left|w_1\cap\S\right|}{2}\right\rfloor$, for such an $L$ there is a unique $\A$-coloring of $L\cap\ZZ$ that coincides with $\alpha$ on $L\cap P$ and has period at most $\left\lfloor\frac{\left|w_1\cap\S\right|}{2}\right\rfloor$.  Let $\mathcal{C}_1$ be the union of all lines parallel to $w_3$ that have nonempty intersection with $P$ and let $\mathcal{C}_2$ be the union of all lines parallel to $w_1$ that have nonempty intersection with $P$.  Let $\tilde{\beta}:\mathcal{C}_1\cup\mathcal{C}_2\to\A$ be the coloring just described.  We claim that $\tilde{\beta}$ extends uniquely to an $\A$-coloring of $\ZZ$ that is $w_1$-periodic of period at most $\left\lfloor\frac{\left|w_1\cap\S\right|}{2}\right\rfloor$ and $w_3$-periodic of period at most $\left\lfloor\frac{\left|w_3\cap\S\right|}{2}\right\rfloor$.  Indeed, if $L$ is any line parallel to $w_1$ that has nonempty intersection with $\ZZ$, then $\left|L\cap(\mathcal{C}_1\cup\mathcal{C}_2)\right|\geq\left|w_1\cap\S\right|-1$ (since this is true for $P$ and $\mathcal{C}_1$ is produced by translating $P$ along the vector $w_3$).  The coloring $\tilde{\beta}\rst{L\cap\mathcal{C}_1}$ is $w_1$-periodic of period at most $\left\lfloor\frac{\left|w_1\cap\S\right|}{2}\right\rfloor$.  As above, the coloring extends uniquely.  We set $\beta\colon\ZZ\to\A$ to be the coloring obtained from $\tilde{\beta}$ by this procedure, and the $w_3$-bound follows from $w_3$-periodicity of $\beta\rst{\T_1-(0,m_1p)}$ and vertical periodicity.

We claim that 
\begin{equation}
\label{eq:equal-on-K}
\alpha\rst{K}=\beta\rst{K},
\end{equation} which establishes the claim for the first case.
Let $\T_3$ denote the $(\S_2\setminus v,v)$-border of $(\T_2+m_2q\cdot w_3)$,
where $\S_2$ is the thin generating set of Section~\ref{sec:thin-generating} and
  $v\in E(\S)$ is the edge parallel to $w_1$.

The bounds established (by ambiguity) on the periods of $\alpha\rst{(\T_1-(0,m_1p))}$ and $\alpha\rst{(\T_2-m_2q\cdot w_3)}$ imply that the restrictions $\alpha\rst{\mathcal{C}_1\cup\mathcal{C}_2}=\beta\rst{\mathcal{C}_1\cup\mathcal{C}_2}$ (since $\beta\rst{\mathcal{C}_1}$ is the unique $\A$-coloring that coincides with $\alpha$ on $P$ and has $w_3$ period at most $\left\lfloor\frac{\left|w_3\cap\S\right|}{2}\right\rfloor$, and similarly with $\beta\rst{\mathcal{C}_2}$).  Let $\vec s\in\ZZ$ be the shortest $w_3$-period of $\beta$.  Then $(T^{\vec s}\beta)\rst{\T_3}=\beta\rst{\T_3}$.  Since $\diam_v(\S_2)\leq\left\lceil\frac{\diam_v(\S)}{2}\right\rceil$ by~\eqref{calculation3}, $\diam_v(\T_3)\leq\left\lfloor\frac{\diam_v(\S)}{2}\right\rfloor$.  Therefore $(\T_3+\vec s)\subseteq(\T_2+m_2q\cdot w_3)$, and it follows that $(T^{\vec s}\alpha)\rst{\T_3}=\alpha\rst{\T_3}$.  Recall that by construction, 
$\S_3\subseteq\S_L\subseteq\S$.  Since we know that $\alpha\rst{\mathcal{C}_2}=\beta\rst{\mathcal{C}_2}$; $\diam_{w_3}(\S_2)\leq\diam_{w_3}(\mathcal{C}_2)$; $\alpha\rst{(\T_2-m_2q\cdot w_3)}$ is $w_3$-periodic (in the sense of Definition~\ref{perFineWilf}); $\beta\rst{\mathcal{C}_2}$ is $w_3$-periodic with period $\vec s$; and $\mathcal{S}_2$ is an $\eta$-generating set, we can conclude that the coloring $\alpha\rst{K}$ can be determined from $\alpha\rst{((\T_1-(0,m_1p))\cap(\T_2-m_2q\cdot w_3))}$ and it follows by induction that $\alpha\rst{K}=\beta\rst{K}$.

\subsubsection*{Case 2} Suppose $\S$ is $w_4$-balanced and for infinitely many $m_1$, $\alpha\rst{\T_1-(0,m_1p)}$ extends non-uniquely to its $w_4$-extension.  For any such $m_1$, 
by Corollary~\ref{semiinfiniteperiodicextension2} $\alpha\rst{\T_1-(0,m_1p)}$ is periodic with period vector parallel to $w_4$ and period at most $\left\lfloor\frac{\left|w_4\cap\S\right|}{2}\right\rfloor$.  The proof now proceeds as in Case 1, with $w_4$ taking the role of $w_3$.

\subsubsection*{Case 3} Suppose $\S$ is $w_4$-balanced and for all but finitely many $m_1$, $\alpha\rst{\T_1-(0,m_1p)}$ extends uniquely to its $w_4$-extension.  We have that 
 $\alpha\rst{\bigcup_{i=1}^{\infty}K_i}$ is periodic with period vector parallel to $w_3$ and period at most $2\left|w_3\cap\S\right|-2$, but is {\em not} vertically periodic.  Thus there is some semi-infinite $(\S,w_4)$-strip in $\bigcup_{i=1}^{\infty}K_i$ to which the restriction of $\alpha$ is $(\S,w_4,\eta)$-ambiguous, as otherwise each of the finitely many $\eta$-colorings arising as the restriction of $\alpha$ to such strips extend uniquely to their $w_4$-extension, forcing vertical periodicity.  Let $\T_4$ be a semi-infinite $(\S\setminus w_4)$-strip in $\bigcup_{i=1}^{\infty}K_i$ to which the restriction of $\alpha$ is $(\S,w_4,\eta)$-ambiguous.  Without loss of generality, we can assume that for any $p>0$, $\alpha\rst{\T_4-(0,p)}$ extends uniquely to its $w_4$-extension.  By Corollary~\ref{semiinfiniteperiodicextension2}, $\alpha\rst{\T_4}$ is eventually periodic with period vector parallel to $w_4$ and period at most $\left\lfloor\frac{\left|w_4\cap\S\right|}{2}\right\rfloor$ and gap at most $\left|w_4\cap\S\right|-1$.  Again by Corollary~\ref{semiinfiniteperiodicextension2}, 
 the restriction of $\alpha$ to the $w_4$-extension of $\T_4$ is eventually periodic with the same gap and period at most $2\left\lfloor\frac{\left|w_4\cap\S\right|}{2}\right\rfloor\leq\left|w_4\cap\S\right|\leq\left|w_3\cap\S\right|-1$.  Inductively, we produce a sequence of sets
$$
\T_4=\T_4^1\subset\T_4^2\subset\cdots
$$
where $\T_4^{i+1}$ is the $w_4$-extension of $\T_4^i$.  Since $\alpha\rst{\T_4^i}$ extends uniquely to its $w_4$-extension and since $\alpha\rst{\T_4}$ is periodic with period at most $\left|w_3\cap\S\right|-1$, the restriction of $\T_4^i$ is also $w_3$-periodic with period dividing that of $\alpha\rst{\T_4}$, for all $i=1,2,\dots$  Since $\alpha\rst{K}$ is doubly periodic and
$$
K\cap\bigcup_{i=1}^{\infty}\T_4^i
$$
is an infinite, convex set whose two semi-infinite edges are non-parallel, the restriction of $\alpha$ to any $(\S\setminus w_4,w_4)$-strip is periodic with period vector parallel to $w_4$ and period dividing the period of $\alpha\rst{\T_4}$.  Since $\alpha\rst{\T_2+m_2\cdot qw_3}$ is vertically periodic with period at most $\left\lfloor\frac{\left|w_1\cap\S\right|}{2}\right\rfloor$, has $w_3$-diameter at least $\diam_{w_3}(P)$, and the $w_3$-period of $\alpha\rst{K}$ is at most $\left|w_3\cap\S\right|-1$, $\alpha\rst{K}$ is also vertically periodic of period at most $\left\lfloor\frac{\left|w_1\cap\S\right|}{2}\right\rfloor$.

This completes the proof of Claim~\ref{claim:bounds-on-periods}.

\subsection{Completing the proof of Theorem~\ref{twoormore}}
We make use of the properties of $\alpha$ to obtain a contradiction. 
Specifically, we show that for a given $\eta$-generating set $\S$, there exists a convex subset $\S^*\subset\S$ for which there are more than 	
$$
P_{\eta}(\S)-P_{\eta}(\S^*)
$$
$\eta$-colorings of $\S^*$ that extend non-uniquely to $\eta$-colorings of $\S$.  This leads to a contradiction, as if
\begin{eqnarray*}
P&:=&\{(T^{\vec u}\eta)\rst{\S}\colon\vec u\in\ZZ\}; \\
Q&:=&\{(T^{\vec u}\eta)\rst{\S^*}\colon\vec u\in\ZZ\}, 
\end{eqnarray*}
then there is a natural surjective map $R:P\to Q$ by restriction.  The number of elements of $Q$ that have more than one preimage (equivalently, the number of colorings of $\S^*$ that extend nonuniquely to colorings of $\S$) is at most $\left|P\right|-\left|Q\right|=P_{\eta}(\S)-P_{\eta}(\S^*)$.  

\subsubsection{Construction of the set $\S^*$}
\label{sec:setS}  

Given $x\in\Z$, let $\ell_x=\{(x,y)\in\ZZ\colon y\in\Z\}$ denote the vertical line passing through $x$.  For $x\in\Z$ such that $\ell_x\cap\S\neq\emptyset$, let $A_x$ denote the bottom-most $\left|w_1\cap\S\right|-2$ elements of $\ell_x\cap\S$ (recall that $\S$ is $w_1$-balanced and so each such intersection contains at least $\left|w_1\cap\S\right|-1$ integer points).  Given $d\geq1$, define
\begin{equation}
\label{eq:Bd}
B(d):=\bigcup_{i=0}^{d-1}A_{(x_{\max}-i)}, 
\end{equation}
where, as in Section~\ref{sec:thin-generating}, $x_{\max}$ denotes that maximal $x$-coordinate of any element of $\S$.  Let $\mathcal{T}(K):=\{\vec u\in\ZZ\colon \S+\vec u\subset K\}$ be the set of translations taking $\S$ to a subset of $K$.   Choose minimal $d\geq1$ such that 
\begin{equation}\label{defofd}
\text{for any $\vec u, \vec v\in\mathcal{T}(K)$, whenever $\alpha\rst{B(d)+\vec u}=\alpha\rst{B(d)+\vec v}$,}
\end{equation}
we have that $\alpha\rst{\S+\vec u}=\alpha\rst{\S+\vec v}$.  Since $\alpha\rst{K}=\beta\rst{K}$ and $\beta$ is doubly periodic, we can rephrase this condition as saying that $d$ is the minimal integer such that
\begin{equation}
\label{conditionond}
\text{every }\beta\text{-coloring of }B(d)\text{ extends uniquely to a }\beta\text{-coloring of }\S.
\end{equation}
(Note that such an integer $d$ exists because $\alpha\rst{K}$ is vertically periodic with period at most $\left\lfloor\frac{\left|w_1\cap\S\right|}{2}\right\rfloor\leq\left|w_1\cap\S\right|-2$.)  Let $\S^*\subset\S$ be the set obtained by removing the topmost element of $\ell_x\cap\S$ for all $x$.  Note that $\S^*$ is a convex, proper subset of $\S$.

Therefore $B(d)\subseteq\S^*$ and $D_{\eta}(\S^*)>D_{\eta}(\S)$, by Property~\eqref{cond:3} of Lemma~\ref{stronggeneratingset}.  As a result, there are at most $\left|\S\setminus\S^*\right|-1$ distinct $\eta$-colorings of $\S^*$ that extend non-uniquely to $\eta$-colorings of $\S$.  

In the next two sections, we obtain a contradiction, thus completing the proof of the theorem.   
We show that there are at
    least $\left|\S\setminus\S^*\right|=\diam_{w_1}(\S)$ distinct $\eta$-colorings of $\S^*$ that extend non-uniquely
    to $\eta$-colorings of $\S$.
The colorings arise from two sources: we find $d$ such $\eta$-colorings that are of the form $(T^{\vec x}\beta)\rst{\S^*}$ (Section~\ref{sec:count-eta-2}) and we find $\diam_{w_1}(\S)-d$ additional $\eta$-colorings that we show are not of the form $(T^{\vec x}\beta)\rst{\S^*}$ (Section~\ref{sec:count-eta-1}).  All of these colorings  turn out to be $\alpha$-colorings of $\S^*$ that extend non-uniquely to $\alpha$-colorings of $\S$ (recall that by~\eqref{eq:equal-on-K}, $\alpha\rst{K}=\beta\rst{K}$).  
    This causes no problem since $\alpha\in X_{\eta}$, and so every $\alpha$-coloring of $\S^*$ that extends non-uniquely to an $\alpha$-coloring of $\S$ is also an $\eta$-coloring that extends non-uniquely.

\subsubsection{Counting colorings along the $w_1$-boundary}
\label{sec:count-eta-1} 
In this section, we find
\begin{equation*}
%\label{definitionofdbar}
\overline{d}:=\diam_{w_1}(\S)-d.
\end{equation*}
distinct $\alpha$-colorings of $\S^*$ that extend non-uniquely to $\alpha$-colorings of $\S$.  We show that none of the these colorings are of the form $(T^{\vec x}\beta)\rst{\S^*}$ for $\vec x\in\ZZ$, meaning that they are not also $\beta$-colorings of $\S^*$.

%\m
%
%\noindent {\bf Setup}.  
\subsubsection*{Setup}
Translating the coordinate system if necessary, we can assume that the edge of $\conv(K)$ parallel to $w_1$ is $\{(0,y)\in\ZZ\colon y\leq0\}$ and that the intersection of the $w_1$-extension of $K$ with the line $\{(-1,y)\colon y\in\Z\}$ is the semi-infinite line $\{(-1,y)\colon y\leq y_0\}$ for some $y_0\in\Z$.  Without loss of generality, assume that
$$
w_1=\{(-1,y)\in\ZZ\colon y_0-\left|w_1\cap\S\right|+1\leq y\leq y_0\}.
$$
Let $\S^*$ and $B(d)$ be as in Section~\ref{sec:setS}.

Let 
\begin{equation}\label{Defofc}
c_1,\dots,c_t\colon B(d)\to\A \text{ denote the set of all $\beta$-colorings of $B(d)$,}
    \end{equation}
    and note that this set coincides with $\eta$-colorings of $B(d)$ occurring as the restriction of
    $\alpha$ to the set $K$. 
    For  $i=1,\dots,t$, 
let $C_i\colon \S^*\to\A$ denote the unique $\beta$-coloring of $\S^*$ 
whose restriction to $B(d)$ is $c_i$, and note that the uniqueness follows from~\eqref{conditionond}.  Equivalently,  
this is the coloring $(T^{\vec u}\alpha)\rst{\S^*}$, where $\vec u\in\ZZ$ is chosen such that $\S+\vec u\subset K$ and $\alpha\rst{B(d)+\vec u}=c_i$.

As in Section~\ref{sec:periodic-half}, let $\tilde{\S}:=\S\setminus w_1$.  Let
\begin{equation*}
%\label{definitionofb}
\vec b\in\ZZ\text{ be the shortest }w_3\text{-period vector for }\alpha\rst{K}.
\end{equation*}
Let $\T_2$ be the $(\tilde{\S},w_1)$-border of $K$, as in Notation~\ref{notation:T1}.  Then the colorings of $\T_2$ given by $\alpha\rst{\T_2}$ and $(T^{\vec b}\alpha)\rst{\T_2}$ coincide.  
By maximality of $K$, 
the colorings of $\T_2\cup\{(-1,y)\colon y\leq y_0\}$ given by $\alpha$ and $T^{\vec b}\alpha$ do not coincide.  We begin by comparing the colorings $\alpha\rst{\{(-1,y)\colon y\leq y_0\}}$ and $(T^{\vec b}\alpha)\rst{\{(-1,y)\colon y\leq y_0\}}$.

%\m

\subsubsection*{The line $\{(-1,y)\colon y\leq y_0\}$ and behavior of $\alpha$}.  By the first part of Claim~\ref{claim:bounds-on-periods}, $\alpha\rst{K}$ is vertically periodic of period at most $\left\lfloor\frac{\left|w_1\cap\S\right|}{2}\right\rfloor$.

Let $(0,-p)$
denote the shortest vertical period for $(T^{\vec b}\alpha)\rst{\{(-1,y)\colon y\leq y_0\}}$.
Then $p$ is a divisor of the smallest vertical period of $\alpha\rst{K}$.  In particular,
\begin{equation}
\label{boundonp}
p\leq\left\lfloor\frac{\left|w_1\cap\S\right|}{2}\right\rfloor.
\end{equation}

\begin{claim}
\label{verticalperiodclaim}
$\alpha\rst{\{(-1,y)\colon y\leq y_0\}}$ is vertically periodic with period $q\leq \left\lfloor\frac{\left|w_1\cap\S\right|}{2}\right\rfloor$.
\end{claim}

To prove the claim,  we first show that there are no integers $0\leq j_1, j_2<p$ such that
$$
(T^{(0,-j_1)}\alpha)\rst{\S}=(T^{(0,-j_2)+\vec b}\alpha)\rst{\S}.
$$
For contradiction, suppose not.    We consider the case that $j_1\neq j_2$ first and then address the case  $j_1=j_2$.

Suppose $j_1<j_2$ and observe that 
$$
(T^{(0,-j_1)}\alpha)\rst{\tilde{\S}}=(T^{(0,-j_2)+\vec b}\alpha)\rst{\tilde{\S}}.
$$
Since $\vec b$ is a period vector for $\alpha\rst{K}$,
$$
(T^{(0,-j_1)}\alpha)\rst{\tilde{\S}}=(T^{(0,-j_2)}\alpha)\rst{\tilde{\S}}.
$$
Since $\S$ is $w_1$-balanced, every line parallel to $w_1$ that has nonempty intersection with $\tilde{\S}$ intersects in at least $\left|w_1\cap\S\right|-1$ places.  Since $\alpha\rst{K}$ is vertically periodic of period at most $\left\lfloor\frac{\left|w_1\cap\S\right|}{2}\right\rfloor\leq\left|w_1\cap\S\right|-2$, this implies that $j_2-j_1$ is a vertical period for $\T_2$ (the $(\tilde{\S},w_1)$-border of $K$).  By Claim~\ref{claim:bounds-on-periods}, the minimal $w_3$-period of $\alpha\rst{K}$ is smaller than the $w_3$-width of $\T_2$, so $j_2-j_1$ is a vertical period for $\alpha\rst{K}$.  This contradicts minimality of $p$, and 
we conclude that $j_1$ cannot be smaller than $j_2$.  A similar argument shows that $j_1$ cannot be larger     than $j_2$.

Suppose $j_1=j_2$.  Then since $\S$ is $\eta$-generating and
\begin{eqnarray*}
(T^{(0,-j_1)}\alpha)\rst{\S}&=&(T^{(0,-j_1)+\vec b}\alpha)\rst{\S}; \\
\alpha\rst{\T_2}&=&(T^{\vec b}\alpha)\rst{\T_2},
\end{eqnarray*}
we have that 
$$
\alpha\rst{\{(-1,y)\colon y\leq y_0\}}=(T^{\vec b}\alpha)\rst{\{(-1,y)\colon y\leq y_0\}}.
$$
This contradicts maximality of $K$.  We conclude that no such integers $j_1, j_2$ exist.
 
Now, there are at most $P_{\eta}(\S)-P_{\eta}(\tilde{\S})$ distinct $\eta$-colorings of $\tilde{\S}$ that extend non-uniquely to $\eta$-colorings of $\S$.  Each of the colorings
$$
\{(T^{(0,-j)}\alpha)\rst{\tilde{\S}}\colon j\in\N\}
$$
is such a coloring, by maximality of $K$ and the fact that $\S$ is $\eta$-generating.  However,
\begin{equation}
\label{eq:nul-inter}
\{(T^{(0,-j)}\alpha)\rst{\S}\colon j\in\N\}\cap\{(T^{\vec b+(0,-j)}\alpha)\rst{\S}\colon j\in\N\}=\emptyset.
\end{equation}
On the other hand,
\begin{eqnarray*}
\left|\{(T^{(0,-j)}\alpha)\rst{\S}\colon j\in\N\}\cup\{(T^{\vec b+(0,-j)}\alpha)\rst{\S}\colon j\in\N\}\right|&& \\
&&\hspace{-1 in}\leq P_{\eta}(\S)-P_{\eta}(\tilde{\S})+\left|\{(T^{(0,-j)}\alpha)\rst{\tilde{\S}}\colon j\in\N\}\right|.
\end{eqnarray*}
Since
$$
\left|\{(T^{\vec b+(0,-j)}\alpha)\rst{\S}\colon j\in\N\}\right|\geq\left|\{(T^{(0,-j)}\alpha)\rst{\tilde{\S}}\colon j\in\N\}\right|,
$$
there are at most $P_{\eta}(\S)-P_{\eta}(\tilde{\S})\leq\left\lfloor\frac{\left|w_1\cap\S\right|}{2}\right\rfloor$ elements of the set
$$
\{(T^{(0,-j)}\alpha)\rst{\S}\colon j\in\N\}.
$$
By Proposition~\ref{prop:extendedambiguousperiod} $\alpha\rst{\{(-1,y)\colon y\leq y_0\}}$ is vertically periodic and by the above bound, 
\begin{equation}
\label{eq:vert-period}
\text{the minimal vertical period of }\alpha\rst{\{(-1,y)\colon y\leq y_0\}}\text{ is } q\leq  \left\lfloor\frac{\left|w_1\cap\S\right|}{2}\right\rfloor.
\end{equation}
This establishes the claim.  

Using the bounds on $p$ and $q$ given by~\eqref{boundonp} and~\eqref{eq:vert-period}, we establish the following claim.
\begin{claim}
\label{distinguishing}
There do not exist integers $0\leq i<\overline{d}$ and $y\leq0$ such that
$$
(T^{(-i,-y)}\alpha)\rst{\S^*}\text{ is a }\beta\text{-coloring of }\S^*.
$$
\end{claim}

We establish the claim by contradiction, and so choose $i$ and $y$ for which the claim fails.  Then 
by definition of $\overline{d}=\diam_{w_1}(\S)-d$, 
we have that $B(d)+(i,y)$ is a subset of $K$, and since $\alpha\rst{K}=\beta\rst{K}$, we have 
that $(T^{(-i,-y)}\alpha)\rst{B(d)}$ is a $\beta$-coloring of $B(d)$.  By definition of $d$, this extends uniquely to a $\beta$-coloring of $\S^*$.  
By choice of $(-i,-y)$, $T^{(-i,-y)}\alpha)\rst{\S^*}$ is a $\beta$-coloring, so
    $$
    (T^{(-i,-y)}\alpha)\rst{\S^*}=(T^{(-i,-y)+\vec b}\alpha)\rst{\S^*}.
    $$ 
Therefore,
\begin{equation}
\label{contradiction1}
\begin{tabular}{l}
there is a set of $\left|w_1\cap\S\right|-1$ consecutive integer points \\
on the line $\{(-1,y)\colon y\leq y_0\}$ where $\alpha$ and $T^{\vec b}\alpha$ coincide.
\end{tabular}
\end{equation}
By~\eqref{boundonp}, the vertical period of the coloring $\alpha\rst{\{(-1,y)\colon y\leq y_0\}}$ is $p\leq\left\lfloor\frac{\left|w_1\cap\S\right|}{2}\right\rfloor$ and by~\eqref{eq:vert-period} the vertical period of the coloring $(T^{\vec b}\alpha)\rst{\{(-1,y)\colon y\leq y_0\}}$ is $q\leq\left\lfloor\frac{\left|w_1\cap\S\right|}{2}\right\rfloor$.

If $p=q$, then $\alpha\rst{\{(-1,y)\colon y\leq y_0\}}=(T^{\vec b}\alpha)\rst{\{(-1,y)\colon y\leq y_0\}}$, contradicting maximality of $K$.  Otherwise $p\neq q$, and since both are integers, $p+q-\gcd(p,q)\leq\left|w_1\cap\S\right|-2$.  
By the Fine-Wilf Theorem and~\eqref{contradiction1}, we again have that
$$
\alpha\rst{\{(-1,y)\colon y\leq y_0\}}=(T^{\vec b}\alpha)\rst{\{(-1,y)\colon y\leq y_0\}},
$$
again a contradiction.  We conclude that no such $0\leq i<\overline{d}$ and $y\leq0$ exist, establishing
the claim.

If there were some $j=0,\dots,\overline{d}-1$ such that the restriction of $\alpha$ to the strip given by
$$
\bigcup_{s\in\Z}(\tilde{\S}+(-j,s))
$$
is vertically periodic, then $\eta$ would be vertically periodic by Corollary~\ref{periodicstripextension}, a contradiction.  On the other hand, by Corollary~\ref{periodicstripextension}, the restriction of $\alpha$ to each such strip is eventually vertically periodic, since $\alpha\rst{K}$ is.  Therefore for all $j=0,\dots,\overline{d}-1$, 
there exists maximal $s_j\in\Z$ such that 
\begin{equation}
\label{eq:semi-inf-strip}
\text{the restriction of $\alpha$ to }
\bigcup_{s=-\infty}^{s_j-1}(\S+(-j,s))
\text{ is vertically periodic.}
\end{equation}

%\m

\subsubsection*{Counting $\alpha$-colorings of $\S^*$ that extend non-uniquely}

\begin{claim}
With the integers $\{s_j\}_{j=0}^{\overline{d}-1}$ as defined above,
\label{lastclaimofsection}
\begin{enumerate}
\item the $\eta$-colorings of $\S^*$ given by $\alpha\rst{(\S^*+(-j,s_j))}$ are distinct for $j=1,\dots,\overline{d}$;
\item for each such $j$, the coloring of $\S^*$ given by $\alpha\rst{(\S^*+(-j,s_j))}$ extends non-uniquely to an $\alpha$-coloring of $\S$.
\end{enumerate}
\end{claim}

 We begin by establishing the first statement.  Observe that $(B(d)-(j,0))\subset K$.  By maximality of $s_j$, the coloring $(T^{(j,-s_j)}\alpha)\rst{B(d)}$ is a $\beta$-coloring of $B(d)$.  
 Using the colorings of~\eqref{Defofc},
 there is some $i=1,\dots,t$ such that $(T^{(j,-s_j)}\alpha)\rst{B(d)}=c_i$ and
    $C_i$ is the unique $\beta$-coloring of $\S^*$ whose restriction to $B(d)$ is $c_i$.  We claim that
    $(T^{(j,-s_j)}\alpha)\rst{\S^*}\neq C_i$.
    
For each $j\geq0$, the coloring $(T^{(0,-j)+\vec b}\alpha)\rst{\S}$ is a $\beta$-coloring of $\S$ since $\alpha\rst{K}=\beta\rst{K}$.  By~\eqref{eq:nul-inter}, none of the colorings $\{(T^{(0,-j)}\alpha)\rst{\S}\colon j\in\N\}$ are $\beta$-colorings.  By~\eqref{eq:vert-period} and maximality of $s_j$, $\alpha\rst{\{(-1,s_j-y)\colon y\in\N\cup\{0\}\}}$ is vertically periodic with period at most $\left\lfloor\frac{\left|w_1\cap\S\right|}{2}\right\rfloor$.  Since every vertical line that has nonempty intersection with $\S^*$ intersects in at least $\left|w_1\cap\S\right|-2$ integer points, the restrictions of $\alpha$ and $\beta$ to the set $\{(-1,y)\colon y\in\ZZ\}\cap(\S^*+(-j,s_j))$ cannot coincide (otherwise by the Fine-Wilf Theorem they would coincide everywhere on the semi-infinite line).  On the other hand, the restrictions of $\alpha$ and $\beta$ to $\{(x,y)\in\ZZ\colon x\geq0\}\cap(\S^*+(-j,s_j))$ do coincide, since they agree on $K$ and $s_j$ was chosen such that the semi-infinite $\S$-strip below it was vertically periodic.  
Consequently, the rightmost vertical line where $\alpha\rst{\S^*+(-j,s_j)}$ differs from $\beta\rst{\S^*+(-j,s_j)}$ has $x$-coordinate $x_{\min}+j-1$.  Therefore, for distinct $1\leq j_1<j_2\leq\overline{d}$, 
the $\eta$-colorings of $\S^*$ given by $\alpha\rst{(\S^*+(-j_1,s_{j_1}))}$ and $\alpha\rst{(S^*+(-j_2,s_{j_2}))}$ are distinct.

For the second statement, by~\eqref{eq:semi-inf-strip} the restriction of $\alpha$ to the semi-infinite $\S$-strip given by
$$
\bigcup_{s=-\infty}^{s_j-1}(\S+(-j,s))
$$
is vertically periodic and $s_j$ is the largest integer with this property.  If the vertical period is $p\in\N$, then by periodicity, the $\eta$-colorings of $\S^*$ given by $\alpha\rst{\S^*+(-j,s_j)}$ and $\alpha\rst{\S^*+(-j,s_j-p)}$ coincide.  But by maximality of $s_j$, the $\eta$-colorings of $\S^*$ given by the functions $\alpha\rst{\S+(-j,s_j)}$ and $\alpha\rst{\S+(-j,s_j-p)}$ are distinct, establishing the claim. 

In total, we have counted $\diam_{w_1}(\S^*)-d$ distinct $\eta$-colorings of $\S^*$ that extend non-uniquely to $\eta$-colorings of $\S$.  Moreover, for each such coloring, the coloring of $\S^*$ was 
not of the form $(T^{\vec x}\beta)\rst{\S^*}$ for any $\vec x\in\ZZ$,
since there was a vertical line in $\S^*$ where the coloring can be distinguished from the $\beta$-coloring induced from its restriction to $B(d)$.

\subsubsection{Counting colorings along the $w_3$-boundary}
\label{sec:count-eta-2}
In this section we find $d$ distinct $\alpha$-colorings of $\S^*$ that extend non-uniquely to $\alpha$-colorings of $\S$.  Each of these colorings is of the form $(T^{\vec x}\beta)\rst{\S^*}$ for some $\vec x\in\ZZ$, and hence they are all distinct from those found in Section~\ref{sec:count-eta-1}.

Recall that $\T_1$, as defined in Notation~\ref{notation:T1} is the $(\S\setminus w_3,w_3)$-border of $K$ and that the restriction $\alpha\rst{\T_1}$ is periodic with period vector parallel to $w_3$.  Fix $\vec d\in\ZZ$ such that the sets $\{(\S\setminus w_3)+\vec d+iw_3\colon i=-1,0,1\}$ are subsets of $\T_1$, but none of the sets $\{\S+\vec d+iw_3\colon i=-1,0,1\}$ are.

Let $A, B\in\Z$ denote the minimal and maximal $x$-coordinates of elements of $w_3$, respectively.  Enumerate the elements of $\S\setminus\S^*$ whose $x$-coordinates are between $A$ and $B$ as $z_1,\dots,z_{\diam_{w_1}(B-A+1)}$, 
where the $x$-coordinate of $z_{i+1}$ is always larger than that of $z_i$.  By Claim~\ref{claim:bounds-on-periods}, $\beta$ is $w_3$-periodic with period at most $\left|w_3\cap\S\right|-1$.    It follows that $d\leq B-A+1$ (recall that $d$ is the integer defined by~\eqref{defofd}).  Let $\vec u_1,\dots,\vec u_d\in\ZZ$ denote the vectors $\vec u_i=z_1-z_i$.  Observe that $(\S^*+\vec d+\vec u_i)\subset \T_1\subset K$ for $i=1,\dots,d$.  For $i=1,\dots,d$, 
we claim that the $\eta$-colorings of $\S^*$ given by $\alpha\rst{\S^*+\vec d+\vec u_i}$ are distinct.  If not, suppose that the colorings given by $\alpha\rst{\S^*+\vec d+\vec u_{\diam_{w_1}(\S^*)-j_1}}$ and $\alpha\rst{\S^*+\vec d+\vec u_{\diam_{w_1}(\S^*)-j_2}}$ coincide for some $1\leq j_1<j_2\leq d$.  Then $0<j_2-j_1<d-1$.  
By~\eqref{eq:Bd}, $B(d)$ is the intersection of $\ZZ$ with the disjoint union of vertical line segments, each of which contains at least $\left|w_1\cap\S\right|-2$ integer points.  Since the vertical period of $\beta$ is at most $\left|w_1\cap\S\right|-2$, we have that the vector $u_{j_2}-u_{j_1}$ must be a period vector for $\beta$, and the $x$-component of this vector is $j_2-j_1\leq d-1$.  Thus any $\beta$-coloring of $\S$ can be deduced from the $\beta$-coloring of $B(j_2-j_1)$, contradicting the minimality of $d$.

Finally, since $\alpha\rst{K}$ is vertically periodic, $(\S^*+\vec d+u_i)\subset K$, and $(\S+\vec d+u_i)\not\subset K$, 
the $\eta$-colorings of $\S^*$ given by $\alpha\rst{(\S^*+\vec d+u_i)}$ and $\alpha\rst{(\S^*+\vec d+u_{i-(0,P)})}$ coincide, where $P$ denotes the minimal vertical period of $\alpha\rst{K}$.  
But by the maximality of $K$ and Corollary~\ref{uniqueextension}, the $\eta$-colorings of $\S$ given by $\alpha\rst{(\S+\vec d+u_i)}$ and $\alpha\rst{(\S+\vec d+u_{(i-(0,p))})}$ cannot coincide.  Therefore we obtain at least $d$ distinct $\eta$-colorings of $\S^*$ that extend non-uniquely to $\eta$-colorings of $\S$ that are of the form $(T^{\vec u}\beta)\rst{\S^*}$ for some $\vec u\in\ZZ$.

\subsubsection{Total number of colorings}
In Sections~\ref{sec:count-eta-1} and~\ref{sec:count-eta-2}, we have described at least 
$\diam_{w_1}(\S^*)$ distinct $\eta$-colorings of $\S^*$ that extend non-uniquely 
to $\eta$-colorings of $\S$.  However, since the discrepancy of $\S^*$ is larger than that of $\S$, 
this produces more than $P_{\eta}(\S)-P_{\eta}(\S^*)\leq\diam_{w_1}(\S^*)-1$ colorings of $\S^*$ that extend non-uniquely to colorings of $\S$, the desired contradiction.  
This completes the proof of Theorem~\ref{twoormore}.
\hfill\qedsymbol

\end{document}